\pdfoutput=1
\documentclass[12pt,leqno]{amsart}
\usepackage{amssymb,amsmath,amsthm,bm}
\usepackage{graphicx}
\usepackage{color}
 \usepackage[textsize=tiny]{todonotes}
\usepackage[normalem]{ulem}
\usepackage[bookmarksopen,bookmarksdepth=3,colorlinks,citecolor=red,pagebackref,hypertexnames=true]{hyperref}
\usepackage[msc-links,nobysame,sorted-cites,initials]{amsrefs}
\usepackage[inline]{enumitem}
\setlist[enumerate]{font=\upshape}
\usepackage{mathtools}
\mathtoolsset{showonlyrefs}
\usepackage{tikz}
\usetikzlibrary{arrows.meta,calc}

\definecolor{darkblue}{RGB}{0,0,160}

\usepackage[lining]{libertine}
\usepackage{cabin}
\usepackage[libertine]{newtxmath}

\def\eps{\varepsilon}
\def\d{{\mathrm d}}
\def\dist{\mathrm {dist}}

\def\R {\mathbb{R}}
\def\N {\mathbb{N}}

\def\C {{\mathbb C}}

\def\M {{\mathrm M}}
\def\Z {{\mathbb Z}}

\def\P {\mathbb P}

\def\ind{\mathbf{1}}

\newcommand{\diam}{\operatorname{diam}}
\newcommand{\supp}{\operatorname{supp}}

\newcommand{\Sy}{\mathsf{Sy}}
\newcommand{\PM}{\mathsf{P}}
\newcommand{\w}{\mathsf w}
\newcommand{\nat}[1]{\langle #1\rangle^{\natural}}
\newcommand{\natl}[2]{\langle #1\rangle^{\natural,[#2]}}
\newcommand{\flt}[1]{\langle #1\rangle^{\flat}}

   \def\XXint#1#2#3{{\setbox0=\hbox{$#1{#2#3}{\int}$}
       \vcenter{\hbox{$#2#3$}}\kern-.5\wd0}}

\makeatletter
\DeclareRobustCommand\widecheck[1]{{\mathpalette\@widecheck{#1}}}
\def\@widecheck#1#2{%
    \setbox\z@\hbox{\m@th$#1#2$}%
    \setbox\tw@\hbox{\m@th$#1%
       \widehat{%
          \vrule\@width\z@\@height\ht\z@
          \vrule\@height\z@\@width\wd\z@}$}%
    \dp\tw@-\ht\z@
    \@tempdima\ht\z@ \advance\@tempdima2\ht\tw@ \divide\@tempdima\thr@@
    \setbox\tw@\hbox{%
       \raise\@tempdima\hbox{\scalebox{1}[-1]{\lower\@tempdima\box
\tw@}}}%
    {\ooalign{\box\tw@ \cr \box\z@}}}
\makeatother

\newcounter{counter}

\numberwithin{equation}{section}
\numberwithin{counter2}{section}

\newtheorem{theorem}[counter]{Theorem}
\newtheorem*{theoremx}{Theorem B$'$}
\newtheorem{introthm}{Theorem}

\numberwithin{withinsub}{subsection}
\newtheorem{corollary}{Corollary}
\numberwithin{corollary}{counter}
\newtheorem{lemma}[subsection]{Lemma}

\newtheorem{proposition}[subsection]{Proposition}

\theoremstyle{definition}
\newtheorem{definition}[subsection]{Definition}

\newtheorem*{remark*}{Remark}
\newtheorem*{warn*}{A word of warning}
\newtheorem{example}[subsection]{Example}
\newtheorem{remark}[subsection]{Remark}

\theoremstyle{definition}

 \theoremstyle{plain}

\numberwithin{figure}{section}

\begin{document}

\title[Sparse domination modulo polynomials]{Sparse domination modulo polynomials}

  \author[F.\ Di Plinio]{Francesco Di Plinio}
\address[F.\ Di Plinio]{Dipartimento di Matematica e Applicazioni ``R.\ Caccioppoli'' \\ \newline\indent Universit\`a di Napoli ``Federico II'', Via Cintia, Monte S.\ Angelo 80126 Napoli, Italy}
\email{\href{mailto:francesco.diplinio@unina.it}{\textnormal{francesco.diplinio@unina.it}}}


\begin{abstract}
We develop a sparse domination theory that captures both local polynomial approximation and cancellation of the input below the scale of each sparse cube. Its coefficients are percentiles of localized smooth maximal functions of the input minus its local polynomial approximation, combining features of mean-oscillation sparse bounds and of the cancellative sparse bounds of Conde-Alonso, Lorist and Rey. Using wavelet representations, we characterize the Calder\'on--Zygmund operators whose derivatives admit these bounds by explicit polynomial testing conditions.

In the homogeneous setting, the same conditions characterize weighted Hardy--Sobolev inequalities with linear dependence on the $A_\infty$ characteristic. This quantitative dependence detects the transition from $\mathrm{BMO}$ to $L^\infty$ testing conditions that Lerner found for pointwise bounds by Riesz potentials, a distinction that qualitative weighted estimates cannot recover. In the inhomogeneous setting, we prove corresponding weighted inequalities and show that they are strictly weaker than pointwise sparse domination from the second derivative order onward. Consequences include Hardy--Sobolev estimates below the Banach range, sharp weighted Sobolev inequalities, and pointwise bounds by Riesz potentials.
\end{abstract}
\maketitle

\section{Introduction}\label{s:intro}
Singular integral operators of Calder\'on--Zygmund type, non-local and non-positive by nature, are pointwise controlled by a sum of  positive  localized  averages of the input function over  the cubes of a \emph{sparse}---in different jargon, Carleson packing---collection $\mathcal S$, see \S\ref{ss:classes} and \cites{CR16,LN19,lacey17,lerner16,  CDPO18b}.
In symbols, for each test function $f$ \begin{equation}\label{e:introavg}
|Tf(x)|\lesssim\sum_{S\in\mathcal S}\langle f\rangle_{S}\ind_S(x) , \qquad x\in \R^d,
\end{equation} with $\langle f\rangle_{S}$ standing for the average of $|f|$ on the cube $S$, and the sparse collection $\mathcal S$ depending on the input function. The original motivation for \eqref{e:introavg} is that $L^p$ and
weighted norm inequalities for $T$ reduce to those of the positive averaging operator on the right hand side, as in Lerner's proof \cite{lerner13} of the $A_2$ theorem \cite{hytonen12}, which relied on the local mean oscillation decomposition of \cite{lerner10}. Later developments beyond the Calder\'on--Zygmund class include  positive sparse domination of rough and non-kernel  singular integrals \cites{BFP16,CCDO17}, modulation invariant  operators \cites{CDPO18,DPL14}, maximal and singular Radon transforms, rough multipliers \cites{BRS24,CDPV22,L19,BRS23}, and discrete operators \cites{CKL19,HLY21}. A more complete overview may be found in the survey \cite{pereyra19}.

However, the averages in \eqref{e:introavg} do not see the vanishing moments of atoms, and for this reason cannot  lead to Hardy space estimates. Moreover, they fail to discount the local proximity of $f$ to polynomials, which governs Sobolev estimates for the derivatives of $T$. Sparse bounds retaining cancellation have developed along two complementary lines. For a class of Calder\'on--Zygmund operators satisfying \(T1=0\), Benea and Bernicot \cite{BB18} replace the averages in \eqref{e:introavg} by mean oscillations of \(f\).   Lerner \cite{lerner26}, still under the   \(T1=0\) assumption, established the same type of  mean oscillation sparse bound for Dini continuous kernels, which in turn yields a characterization, for \(d\geq2\), of the pointwise inequality \(|Tf|\lesssim I_1(|\nabla f|)\) by \(T1\in L^\infty(\R^d)\), where \(I_1\) is the Riesz potential of order one. Previous  work of  Hoang, Moen and P\'erez \cites{HMP25a,HMP25b,HMP25c}  had established this type of pointwise bound for rough singular convolution kernels. In parallel,  the recent work of Conde-Alonso, Lorist and Rey \cite{CLR26} develops instead a sparse domination principle based on percentiles of localized smooth grand maximal functions, which detect cancellation occurring below the scale of the reference cube, and lead to weighted Hardy space estimates. The contrast between the two lines is that mean oscillations subtract constants, but do not see this subscale cancellation, while the percentile coefficients of \cite{CLR26} see it, but do not discount locally constant behavior. This article designs a set of coefficients, \eqref{e:intropo}  below, combining and generalizing the scope  of both operations simultaneously.

\subsection{Percentile modulo polynomials}Below,  \(\M^\tau_{3Q}\) denotes a version of  the localized smooth grand maximal function of \cite{CLR26},
with test functions of smoothness \(\tau>0\), and \(g^{*,Q}(r)\) is the upper \(r\)-percentile of \(|g|\) on \(Q\), at  a sufficiently small level \(r>0\).
 For $m\geq 0$, let \(\PM_Q^m f\) denote the local polynomial projection of \(f\) onto \(\P^{m-1}\), the space of polynomials of degree at most $m-1$, with $\P^{-1}=\{0\}$, see \eqref{e:proj}. Set
\begin{equation}\label{e:intropo}
\nat{f}_{m,Q}
\coloneqq
\left(\M^\tau_{3Q}\big[f-\PM^m_{3Q}f\big]\right)^{*,Q}(r),
\qquad m=0,1,2,\ldots;  \end{equation}
 precise definitions are given in \S\S\ref{ss:classes}--\ref{ss:perc}. Thus one first subtracts the local polynomial, then propagates the remainder through the smooth maximal function, and finally retains a fixed percentile. For \(m=0\), \eqref{e:intropo} is the coefficient of \cite{CLR26}. For \(m\ge1\), it vanishes whenever \(f\) agrees on \(3Q\) with a polynomial of degree less than \(m\), and it is dominated both by the coefficient of order zero and, up to the factor $r^{-1}$, by the Campanato oscillation \eqref{e:campanato} \cite{campanato63}, which removes the part of $f$ that is locally polynomial, see Lemma \ref{l:po}. It may be much smaller than either, for instance for a polynomial plus a concentrated atom with vanishing mean, see Remark \ref{r:axes}.

The two operations are shown to play complementary roles in the results below. The subtraction of polynomials determines  the testing conditions on the operator being sparsely dominated.  The percentile retains the cancellation at the scales below those of the sparse cubes, which the weighted Hardy space and $A_\infty$ estimates for $p\leq1$ require. The underlying scope of this article  is then the identification of the polynomial testing conditions under which  the derivatives of Calder\'on--Zygmund operators satisfy sparse bounds by the coefficients \eqref{e:intropo}, and the corresponding quantitative weighted Hardy--Sobolev inequalities, in particular depending linearly on the $A_\infty$ characteristic.

\subsection{Quantitative weighted theory}\label{ss:introweights}
Let \[\mathcal N_mf\coloneqq\sup_Q\nat{f}_{m,Q}\ind_Q\] be the maximal percentile oscillation of order \(m\). For fixed \(0<p<\infty\) and \(w\in A_\infty\), the norms \[\|\mathcal N_mf\|_{L^p(w)}, \qquad  m\geq0,\]
 are all comparable to the \(L^p(w)\)-norm of the smooth grand maximal function of \(f\), and hence, in a range of \(p\) and \(w\), to a weighted Hardy norm \cite{ST89}; see Proposition \ref{c:hardynorm}. Thus, for fixed \(p\) and \(w\), these norms do not distinguish the polynomial order. Remarkably, the distinction reappears in the \emph{quantitative dependence on the weight}. The simplest instance is the following case of Theorem \ref{t:HS}, for the Calder\'on--Zygmund operators $T$ of order zero in the sense of Definition \ref{d:czo}. The kernel of $T$ satisfies the standard size and H\"older estimates, with H\"older exponent $\delta$, and $T$ satisfies the weak boundedness property.
\begin{introthm}\label{t:introI}
Let $T$ be a Calder\'on--Zygmund operator of order zero, $0<\tau<\delta$ and $r>0$ sufficiently small. The following are equivalent.
\begin{enumerate}[label=(\roman*)]
\item $T1\in L^\infty(\R^d)$ and $T^*1\in\mathrm{BMO}(\R^d)$.
\item For every $f\in \mathcal{C}^\infty_c(\R^d)$, the distribution $Tf$ is a locally integrable function, and there exists a sparse collection $\mathcal S$ of cubes of $\R^d$ such that
\[
|Tf|\lesssim\sum_{S\in\mathcal S}\nat{f}_{1,S}\ind_S\qquad\textrm{a.e.\ on }\R^d.
\]
\item For all $0<p<\infty$, $w\in A_\infty$ and $f\in \mathcal{C}^\infty_c(\R^d)$,
\[
\|Tf\|_{L^p(w)}\lesssim_p[w]_{A_\infty}\|\mathcal N_1f\|_{L^p(w)}.
\]
\end{enumerate}
With $\nat{f}_{0,S}$ and $\mathcal N_0$ in place of $\nat{f}_{1,S}$ and $\mathcal N_1$, the same equivalence holds with $T1\in\mathrm{BMO}(\R^d)$ in place of $T1\in L^\infty(\R^d)$.
\end{introthm}
By Lemma \ref{l:po}(iii), the Poincar\'e inequality, \eqref{e:band} and Lemma \ref{l:sparsepot}, the sparse bound (ii) implies, for $d\geq2$, the pointwise inequality \[|Tf|\lesssim I_1(|\nabla f|),\]
which Lerner characterized by $T1\in L^\infty(\R^d)$ \cite{lerner26}*{Theorem 1.2}; see \S\ref{ss:previous}. Theorem \ref{t:introI} gives this characterization for the percentile oscillations, in every dimension, together with the quantitative converse that (iii) implies (i), while (iii) with $\mathcal N_0$ in place of $\mathcal N_1$ characterizes the weaker condition $T1\in\mathrm{BMO}(\R^d)$. Thus the transition from \(\mathrm{BMO}\) to \(L^\infty\) found by Lerner for the pointwise estimate by \(I_1(|\nabla f|)\) persists at the level of weighted norm inequalities and  is detected by the uniform linear dependence on \([w]_{A_\infty}\).
The converse direction is proved by testing on logarithmic functions together with power weights whose \(A_\infty\) characteristic records the number of active scales. Roughly, \(\mathrm{BMO}\) oscillation of $T1$ can create an additional logarithmic accumulation, while the linear \(A_\infty\) factor accounts for only one scale factor. Uniform linear dependence therefore forces $T1\in L^\infty(\R^d)$.

\subsection{Derivatives and the residual-order rule}\label{ss:introjet}
Let \(T\) be a Calder\'on--Zygmund operator of order \(k\geq1\), in the sense of Definition \ref{d:czo}: its kernel has \(k\) derivatives in the first variable, with the corresponding Calder\'on--Zygmund estimates, and \(T\) satisfies the weak boundedness property. This is the global version of the framework of Green, Wick and the author \cite{DGWdom}. In its wavelet resolution, recalled in Proposition \ref{p:rep}, the derivative \(\partial^\alpha T\), \(|\alpha|=\kappa\leq k\), acts on the derivatives \(\partial^\gamma f\), \(|\gamma|\leq\kappa\), of the input as on distinct \emph{channels}; schematically, if \(\kappa\geq1\),
\[
\partial^\alpha Tf=\sum_{|\beta|=\kappa}V_{\alpha,\beta}\left(\partial^\beta f\right)+\sum_{|\gamma|\leq\kappa}\frac1{\gamma!}P^{\alpha,\gamma}\left(\partial^\gamma f\right),
\]
where the \(V_{\alpha,\beta}\) are model wavelet operators with cancellative input packets, and \(P^{\alpha,\gamma}\) is the paraproduct with symbol \(b^{\alpha,\gamma}\), which records the action of \(\partial^\alpha T\) on monomials of degree \(|\gamma|\) and is attached to the channel \(\partial^\gamma f\). The symbols with \(|\gamma|<\kappa\) are the \emph{lower} symbols, those with \(0<|\gamma|<\kappa\) the \emph{intermediate} symbols, and those with \(|\gamma|=\kappa\) the \emph{top} symbols. When \(\kappa=0\), the representation also contains an adjoint paraproduct, with symbol \(b^{0,\star}\), and \(b^{0,0}\), \(b^{0,\star}\) are the wavelet coefficients of \(T1\) and \(T^*1\).

The testing conditions for sparse domination by \eqref{e:intropo}-type coefficients follow a  residual-order rule. In a sparse bound modulo polynomials of degree less than \(m\), the symbol $b^{\alpha,\gamma}$ attached to the channel $\partial^\gamma f$ is measured at residual order $m-|\gamma|$: residual order zero requires a $\mathrm{BMO}$ condition, while every positive residual order requires $L^\infty$. The adjoint symbol is always subject to the $\mathrm{BMO}$ condition. In the homogeneous theory of Theorem \ref{t:HS}, which contains Theorem \ref{t:introI}, the lower channels are genuinely absent, and the corresponding symbols must vanish. In the inhomogeneous theory of Theorem \ref{t:jet}, the conditions on the top symbols are local, and the rule also applies to the intermediate channels, which do not appear explicitly in the sparse bounds: their percentile oscillations are controlled by those of the two extreme channels $f$ and $\nabla^\kappa f$, see Lemma \ref{l:sandwichperc}. Their residual order is still positive. Thus zero residual order recovers the usual $\mathrm{BMO}$ testing of the $T(1)$ theorem \cite{DJ84}, while positive residual order leads to bounded testing. The simplest inhomogeneous instance is that of the first derivatives, with $\kappa=m=1$.
\begin{introthm}\label{t:introII}
Let $T$ be a Calder\'on--Zygmund operator of order $k\geq1$, $|\alpha|=1$, $0<\tau<\delta$ and $r>0$ sufficiently small. The following are equivalent.
\begin{enumerate}[label=(\roman*)]
\item The symbol $b^{\alpha,0}$ belongs to $L^\infty(\R^d)$, and the top symbols $b^{\alpha,\gamma}$, $|\gamma|=1$, belong locally to $\mathrm{BMO}(\R^d)$, that is, at the scales at most one, see Definition \ref{d:symbols}.
\item For every $f\in \mathcal{C}^\infty_c(\R^d)$, the distribution $\partial^\alpha Tf$ is a locally integrable function, and there exists a sparse collection $\mathcal S$ of cubes of $\R^d$ such that
\[
|\partial^\alpha Tf|\lesssim\sum_{S\in\mathcal S}\left(\nat{f}_{1,S}+\nat{\nabla f}_{0,S}\right)\ind_S\qquad\textrm{a.e.\ on }\R^d.
\]
\item For all $0<p<\infty$, $w\in A_\infty$ and $f\in \mathcal{C}^\infty_c(\R^d)$,
\[
\|\partial^\alpha Tf\|_{L^p(w)}\lesssim_p[w]_{A_\infty}\bigg(\|\mathcal N_1f\|_{L^p(w)}+\sum_{j=1}^d\|\mathcal N_0(\partial_jf)\|_{L^p(w)}\bigg).
\]
\end{enumerate}
\end{introthm}
Theorem \ref{t:introII} is the case $\kappa=m=1$ of Theorem \ref{t:jet} and Proposition \ref{p:rigidity}(b). In accordance with the residual-order rule, the lower channel $f$, at residual order one, requires $L^\infty$, and the top channel $\nabla f$, at residual order zero, requires local $\mathrm{BMO}$. In general, for $|\alpha|=\kappa\geq1$ and $m\geq\kappa$, Theorem \ref{t:jet} characterizes the sparse bound
\[
|\partial^\alpha Tf|\lesssim\sum_{S\in\mathcal S}\left(\nat{f}_{m,S}+\nat{\nabla^\kappa f}_{m-\kappa,S}\right)\ind_S
\]
by the conditions that the lower symbols belong to $L^\infty(\R^d)$, and the top symbols belong locally to $\mathrm{BMO}(\R^d)$ if $m=\kappa$, and locally to $L^\infty(\R^d)$ if $m>\kappa$. Under the same conditions, but with $\mathrm{BMO}$ in place of $L^\infty$ for the intermediate symbols, Theorem \ref{t:HSjet} gives the corresponding weighted inequality, with $\mathcal N_mf$ and $\mathcal N_{m-\kappa}(\partial^\beta f)$, $|\beta|=\kappa$, on the right hand side. Proposition \ref{p:rigidity}(c) shows that, as soon as $\kappa\geq2$, this weighted inequality does not imply the boundedness of the intermediate symbols, so that the pointwise and norm theories separate; whether it forces the top symbols to belong locally to $L^\infty(\R^d)$ when $m>\kappa$ is open, see Remark \ref{r:open}.

For fixed $p$ and $w$, the right hand sides of these weighted inequalities are comparable, by Proposition \ref{c:hardynorm}, to the sums of the $L^p(w)$-norms of the smooth grand maximal functions of $f$ and $\nabla^\kappa f$, and thus, in a range of $p$ and $w$, to the sums of the weighted Hardy norms of $f$ and $\nabla^\kappa f$ \cite{ST89}. Weighted inhomogeneous Hardy--Sobolev, and more generally Triebel--Lizorkin, spaces with $A_\infty$ weights go back to Bui \cite{Bui82}. Even with matrix $A_\infty$ weights, smooth pointwise multipliers, and pseudo-differential operators whose adjoints map the polynomials of low degree to polynomials, are bounded on them \cite{BHYY26}*{Theorems 6.10 and 5.13}. The same holds for the Calder\'on--Zygmund operators which, together with their adjoints, annihilate the polynomials of low degree, that is, have vanishing paraproducts, and whose kernels decay at infinity faster than $|x-y|^{-d}$, at a rate depending on the space \cite{BHYY26}*{Theorem 6.24}. The decay, which excludes the Riesz transforms \cite{BHYY26}*{Remark 6.23}, handles the atoms supported on the largest cubes, which need not have vanishing mean in these inhomogeneous spaces \cite{BHYY26}*{Proposition 6.25}; a Riesz transform maps such an atom with nonvanishing mean outside $L^p(\R^d)$ when $p\leq1$. The inequalities above are of a different type: global Hardy norms of $f$ and $\nabla^\kappa f$ on the right, $L^p(w)$ on the left. To the best of our knowledge, weighted Hardy--Sobolev estimates of this type for the derivatives of Calder\'on--Zygmund operators with nonvanishing paraproducts, under $\mathrm{BMO}$ and $L^\infty$ testing conditions on their symbols, and with linear dependence on $[w]_{A_\infty}$, have not appeared in previous work. They are out of reach of standard sparse averages, and rely on the cancellation of $f$ and $\nabla^\kappa f$ below the scale of the sparse cubes, which the percentile retains.

The symbol $b^{\alpha,0}$, which records the action of $\partial^\alpha T$ on constants, need not vanish. For instance, multiplication by a bounded function \(a\) is, up to a constant multiple, a Calder\'on--Zygmund operator of every order, and \(\nabla(af)=a\nabla f+f\nabla a\). The term \(f\nabla a\) prevents a sparse bound in terms of \(\nabla f\) alone unless \(a\) is constant, while the lower channel \(f\) in (ii) of Theorem \ref{t:introII} absorbs it exactly when \(\nabla a\in L^\infty(\R^d)\), for instance for smooth periodic \(a\); see Example \ref{r:mult}.
\subsection{Methods and organization}\label{ss:intromethods}
Section \ref{s:principle} gives an operator-free stopping principle: local weak control, together with control of the oscillation between nested cubes from suitable good points, produces a sparse pointwise bound; see Remark \ref{r:principle} for its relation with the operator-free principle of \cite{LLO22}, used in \cite{CLR26}. In Section \ref{s:A}, the principle is applied to the smooth model wavelet operators $V$ of Definition \ref{d:model}, with the input renormalized at each cube $Q$ to $f-\PM^m_{3Q}f$. Theorem \ref{t:model} gives, for every $m\geq0$,
\begin{equation}\label{e:intromodel}
|Vf|
\lesssim
\sum_{S\in\mathcal S}\nat{f}_{m,S}\ind_S ,
\end{equation}
with no additional assumption if $m=0$, and, if $m\geq1$, under the single cutoff testing condition $\sup_Q\|V(\omega^\infty_Q)\|_{L^\infty(\R^d)}<\infty$, where $\omega^\infty_Q$ is a smooth cutoff which equals one on a fixed dilate of $Q$. This condition is also necessary, see \S\ref{ss:remA}. When the stopping cube changes, the two renormalizations differ by a polynomial: the good points control this polynomial by the percentile oscillation, while the cutoff testing condition controls its image under $V$. In this sense the argument lifts the subtraction of constants in \cite{BB18} to the subtraction of polynomials of arbitrary order. Cube-dependent polynomial projections inside a sparse stopping argument appear in \cite{LLO22}*{Proposition 5.4}, for the identity operator, and sparse local polynomial approximation is studied in \cite{DM21}. Mean-zero input packets satisfy the cutoff testing condition automatically; for the canonical paraproducts, boundedness of the symbol implies it by Lemma \ref{l:para}(ii): this is the model origin of the residual-order rule. The percentile Poincar\'e inequalities of Section \ref{s:poincare}, proved by the anti-integration method of \cite{DGWdom}, pass from polynomial oscillations to Sobolev information. The wavelet resolution of Proposition \ref{p:rep} reduces the derivatives of a Calder\'on--Zygmund operator to cancellative model operators and paraproducts, with one polynomial testing symbol attached to each channel. Unlike the rough dyadic representation of \cite{hytonen12}, it retains both the regularity of the kernel and the polynomial testing structure needed by the percentile bounds. The converse implications are obtained by testing on localized polynomials and, for the quantitative $A_\infty$ characterizations, on the logarithmic functions of \S\ref{ss:signedlog}.

Section \ref{s:po} develops percentile oscillations and their Hardy-space maximal functions, Section \ref{s:main} contains the main statements and their relation with previous work, and Section \ref{s:aux} the auxiliary lemmas. The pointwise machinery is developed in Sections \ref{s:principle}--\ref{s:A}, and Section \ref{s:cons} turns sparse bounds into weighted norm inequalities and Hardy norms. Section \ref{s:cz} proves the pointwise results and the sharp $A_p$ Sobolev estimates, Section \ref{s:HS} the weighted Hardy--Sobolev characterizations, and Section \ref{s:examples} collects examples and open questions. A companion work \cite{DPSquareF}, completed after the first version of this article became public, treats Wilson's intrinsic square functions \cite{Wilson07} modulo polynomials.

\section{Percentile oscillations}\label{s:po}
This section contains the notation, the percentile oscillations with their basic properties, and their comparison with Hardy norms. The auxiliary facts on cubes, maximal functions and polynomial projections, and the proof of Lemma \ref{l:po}, are in Section \ref{s:aux}.

\subsection{Notation and main objects}\label{ss:classes}

Throughout, $d\geq 1$ is the dimension of the ambient Euclidean space. \subsubsection*{Constants}The almost inequality $A\lesssim B$ means that $A\leq CB$, where $C$ is allowed to depend on $d$, on the order $m$ of the polynomials involved, and on the smoothness, decay and dilation parameters of the classes introduced below, and on nothing else. Such a constant is called \emph{admissible}, and $A\sim B$ stands for $A\lesssim B\lesssim A$. A subscript, as in $\lesssim_p$, indicates a further dependence of the constant, and constants which are kept track of are named explicitly.
\subsubsection*{Cubes and grids}\emph{Cubes} are Cartesian products of $d$ left-closed, half-open intervals of equal length. A cube $Q$ has center $c(Q)$, sidelength $\ell(Q)$ and interior $Q^\circ$, and $aQ$ is the cube with center $c(Q)$ and sidelength $a\ell(Q)$; balls $B(x,\rho)$ are Euclidean, while $\diam$ and $\dist$ refer to the maximum norm on $\R^d$. To each cube $Q$ and each $f\in L^1_{\mathrm{loc}}(\R^d)$ we associate the rescalings
\begin{equation} \label{e:norm1}
f_Q \coloneqq |Q|^{-1} \Sy_Q f, \qquad f^\infty_Q\coloneqq\Sy_Qf, \qquad \Sy_Q f(x)\coloneqq f\left(\frac{x-c(Q)}{\ell(Q)}\right),
\end{equation}
which preserve the $L^1$ and the $L^\infty$ norm, respectively. The standard dyadic grid is $\mathcal D$, and $\mathcal D(Q)$ is the collection of the \emph{dyadic subcubes} of a cube $Q$, obtained from $Q$ by $j\geq 0$ successive bisections of its sides, so that $\mathcal D(Q)=\{R\in \mathcal D: R\subseteq Q\}$ when $Q\in \mathcal D$. The dyadic parent of $L\in \mathcal D(Q)\setminus\{Q\}$ is $\widehat L$, and the \emph{shadow} of a collection $\mathcal R$ of cubes is $\mathsf{sh}(\mathcal R)\coloneqq\bigcup\{R:R\in\mathcal R\}$.

\subsubsection*{Sparse collections}A collection $\mathcal S$ of cubes is \emph{$\eta$-sparse}, $0<\eta<1$, if there exist pairwise disjoint measurable sets $G_S\subseteq S$ with
\[
|G_S|\geq \eta|S|,\qquad S\in\mathcal S,
\]
and \emph{sparse} if it is $\eta$-sparse for some $\eta$. Sparseness is equivalent to the Carleson packing condition \cite{Han18}*{Theorem 1.3}, see also \cite{LN19}*{Lemma 6.3} for dyadic cubes. Consequently, if $\mathcal S_1,\dots,\mathcal S_N$ are $\eta$-sparse, then $\mathcal S_1\cup\dots\cup\mathcal S_N$ is $\frac\eta N$-sparse.

In this paper, we deliberately keep the value of the sparseness coefficient $\eta>0$ unspecified, but admissibly bounded below. All sparse collections appearing in this article are meant to be $\eta$-sparse for $\eta\geq\eta_0$, where $\eta_0>0$ is a small admissible constant. Moreover, the sparse collections $\mathcal S$ will be subcollections of some dyadic grid on $\R^d$, therefore inheriting the grid property
\[
P,Q\in\mathcal S\implies P\cap Q\in\{P,Q,\varnothing\}.
\]
The standard $3^d$ dyadic grid trick \cite{LN19}*{Theorem 3.1} and the doubling nature of the coefficients, Lemma \ref{l:po}(v), may be used to reduce to this case, splitting a sparse collection into boundedly many subcollections of dyadic grids at the cost of a smaller percentile level, so that this property will be used without further specific mention.

\subsubsection*{Local norms, percentiles, maximal functions}The local norms of $f\in L^1_{\mathrm{loc}}(\R^d)$ on the cube $Q$ are indicated by
\[
\langle f\rangle_{p,Q}\coloneqq |Q|^{-\frac1p}\left\|f\ind_Q\right\|_{L^p(\R^d)}, \qquad \langle f \rangle_Q \coloneqq \langle f \rangle_{1,Q},
\]
and the local nonincreasing rearrangement of $f$ on $Q$ by
\begin{equation}\label{e:rearr}
f^{*,Q}(t)\coloneqq\inf\left\{s>0: \left|\{x\in Q: |f(x)|>s\}\right|\leq t|Q|\right\}, \qquad 0<t<1.
\end{equation}
In other words, $f^{*,Q}(t)$ is the upper $t$-percentile of $|f|$ on $Q$. The notation  $\M$ with no further decoration stands for the canonical   Hardy--Littlewood maximal function on $\R^d$, and $\M_s$, ${0\leq s<d}$, for the fractional maximal function,
\[
\M f \coloneqq \sup_{\substack{ Q \subset \R^d\\ Q \, \textrm{cube} } } \langle f\rangle_Q\ind_Q,\qquad \M_s f \coloneqq \sup_{\substack{ Q \subset \R^d\\ Q \, \textrm{cube} } } \ell(Q)^s\langle f\rangle_Q\ind_Q,
\]
so that $\M_0=\M$.

\subsubsection*{Weights}A \emph{weight} stands for a positive locally integrable function on $\R^d$. As usual, the weight $w$ belongs to $A_p$, $1<p<\infty$, if  \[[w]_{A_p}\coloneqq\sup_{Q\subset \R^d \textrm{ cube}}\langle w\rangle_Q\langle w^{-\frac1{p-1}}\rangle_Q^{p-1}\] is finite.
A weight $w$ on $\R^d$ belongs to $A_\infty$ if its Fujii--Wilson characteristic
\[
[w]_{A_\infty}\coloneqq\sup_{Q\subset \R^d \textrm{ cube}}\frac{1}{w(Q)}\int_Q\M(w\ind_Q) <\infty.
\]

\subsubsection*{Wavelet classes}
We recall  the smooth $L^1$-normalized wavelet classes of \cite{DGWdom}*{\S2.1}.  For ${0\leq s<1}$ and $\rho\geq 0$, introduce the norm
\[
\|\varphi\|_{s,\rho}\coloneqq \sup_{x\in \R^d}(1+|x|)^{d+\rho}\left[|\varphi(x)|+\sup_{0<|h|\leq 1}\frac{|\varphi(x+h)-\varphi(x)|}{|h|^s}\right].
\]
Throughout, let $\sigma\geq 0$ be a smoothness parameter, and $\{\sigma\}=\sigma-\lfloor\sigma\rfloor$ be its fractional part. For each cube $R$ of $\R^d$, recalling \eqref{e:norm1}, define the $L^1$-normalized class
\[
\Phi^{\sigma,\rho}(R)\coloneqq \left\{ \varphi_R: \varphi\in W^{\lfloor\sigma\rfloor,\infty}(\R^d),\, \sup_{0\leq |\alpha|\leq \lfloor\sigma\rfloor}\left\|\partial^\alpha\varphi\right\|_{\{\sigma\},\sigma+\rho}\leq 1\right\},
\]
and, when $\{\sigma\}+\rho>0$, its cancellative subclass
\[
\Psi^{\sigma,\rho}(R)\coloneqq\left\{\varphi\in \Phi^{\sigma,\rho}(R): \int_{\R^d}x^\alpha\varphi(x)\,\d x=0 \quad \forall\, 0\leq |\alpha|\leq\lfloor\sigma\rfloor\right\}.
\]
Finally, we fix once and for all a dilation parameter $\w\geq 1$, and replace the decay parameter $\rho$ by the superscript $\Subset$ to indicate the subset of the corresponding class having compact support in $\w R$, namely
\[
\Phi^{\sigma,\Subset}(R)\coloneqq \left\{\varphi\in \Phi^{\sigma,1}(R): \supp \varphi\Subset \w R \right\}, \qquad \Psi^{\sigma,\Subset}(R)\coloneqq \Phi^{\sigma,\Subset}(R)\cap \Psi^{\sigma,1}(R).
\]
The parameter $\w$ plays the same role as in \cite{DGWdom}, allowing for instance for the smooth, compactly supported wavelets of \cite{DGWdom}*{Remark 2.2 and Proposition 2.3}. To each $\varphi\in\Phi^{\sigma,\rho}(R)$ we associate the linear functional
\[
\varphi(f)\coloneqq \int_{\R^d} \varphi(x)f(x)\,\d x.
\]
Consistently with the above convention, $\varphi\in C\Phi^{\sigma,\rho}(R)$ stands for the assertion that $\varphi$ is an admissible constant multiple of an element of $\Phi^{\sigma,\rho}(R)$, and likewise for the other classes.

\subsection{Local maximal functions and polynomial projections}\label{ss:max}

Let $\tau>0$ and $Q$ be a cube. The \emph{localized smooth grand maximal function} of Conde-Alonso, Lorist and Rey \cite{CLR26} is, in the normalization of \S\ref{ss:classes},
\begin{equation}\label{e:max}
\M^\tau_Q f(x)\coloneqq \sup\left\{|\varphi(f)|:\; \varphi\in\Phi^{\tau,1}(R),\; \supp\varphi\subseteq R,\; x\in R\subseteq Q\right\}, \qquad x\in Q.
\end{equation}
It is easy to see that   $\M^\tau_Qf$ is lower semicontinuous on $Q^\circ$.   The test functions in \eqref{e:max} satisfy $\|\varphi\|_\infty\leq |R|^{-1}$ and $\|\varphi\|_1\leq 1$. Therefore \[
  \M^\tau_Qf(x)\leq \M[f\ind_Q](x), \qquad x\in Q.\]  Also, $\M^\tau_Q f\leq \M^\tau_{R}f$ on $Q$ whenever $Q\subseteq R$. Conversely, $|f|\lesssim\M^\tau_Qf$ almost everywhere on $Q$, see Lemma \ref{l:diff} below. With $\R^d$ in place of $Q$, \eqref{e:max} defines the \emph{smooth grand maximal function} $\M^\tau_{\R^d}$. For vector valued $g=(g_1,\dots,g_N)$ and $Q$ a cube or $\R^d$, we set
   \begin{equation}
\label{e:msigmavec}
\M^\tau_Qg\coloneqq\M^\tau_Qg_1+\dots+\M^\tau_Qg_N.
\end{equation}
 In particular, for $f\in W^{m,1}_{\mathrm{loc}}(\R^d)$,
\[
\M^\tau_Q\left[\nabla^mf\right]=\sum_{|\alpha|=m}\M^\tau_Q\left[\partial^\alpha f\right].
\]

\subsubsection*{Local Taylor polynomials} Hereafter,  \[
\P^{-1}\coloneqq \{0\},\qquad \P^{m}\coloneqq\left\{\textrm{polynomials on } \R^d \textrm{ of degree at most } m\right\},\; m\geq 0.\]
Let $\theta$ be a fixed smooth nonnegative function with $\int \theta=1$ and support in $\frac12[-\frac14,\frac14)^d$. For an integer $m\geq 1$ and $f\in W^{m-1,1}_{\mathrm{loc}}(\R^d)$, the Taylor type polynomial of \cite{DGWdom}*{(2.28)} is
\begin{equation}\label{e:proj} \P^{m-1}\ni
\PM^m_Qf(x)\coloneqq \sum_{|\alpha|\leq m-1}\frac{1}{\alpha!}\int_{\R^d}\theta_Q(y)\,\partial^\alpha f(y)\,(x-y)^\alpha\,\d y = \phi_{Q,x}(f),
\end{equation}
where $\theta_Q$ refers to \eqref{e:norm1} and  the second equality is obtained by integrating by parts, with
\[
\phi_{Q,x}(y)\coloneqq \sum_{|\alpha|\leq m-1}\frac{(-1)^{|\alpha|}}{\alpha!}\,\partial_y^\alpha\left[\theta_Q(y)(x-y)^\alpha\right].
\]
As $\phi_{Q,x}$ is smooth and compactly supported, the rightmost expression in \eqref{e:proj} is meaningful for all $f\in L^1_{\mathrm{loc}}(\R^d)$, and we adopt it as the definition of $\PM^m_Q f$ in this generality. We also set $\PM^0_Q\coloneqq 0$. By Lemma \ref{l:proj}, $\PM^m_Q$ is a linear projection of $L^1_{\mathrm{loc}}(\R^d)$ onto $\P^{m-1}$. The lemma also collects its other properties needed below, see also \cite{DGWdom}*{Lemma 2.26}.

\subsection{Percentile oscillations}\label{ss:perc}

We come to the main definition of this section.
Let $\tau>0$ and $0<r<1$, and let $m\geq 0$ be an integer. The \emph{percentile oscillation of order $m$} of $f\in L^1_{\mathrm{loc}}(\R^d)$ on the cube $Q$ is
\begin{equation}\label{e:po}
\nat{f}_{m,Q}\coloneqq \left(\M^\tau_{3Q}\left[f-\PM^m_{3Q}f\right]\right)^{*,Q}(r).
\end{equation}
The smoothness parameter $\tau$ and the percentile level $r$ are kept fixed and omitted from the notation. When the percentile level has to be indicated, we write $\natl{f}{r}_{m,Q}$. Unless stated otherwise, $0<r\leq r_0$ for a small admissible constant $r_0$, which may decrease from statement to statement. Implied constants are independent of $r$ unless its dependence is indicated; every change of percentile level is written explicitly. For comparison, the \emph{Campanato oscillation of order $m$} of $f$ on $Q$ is
\begin{equation}\label{e:campanato}
\flt{f}_{m,Q}\coloneqq\inf_{\pi\in\P^{m-1}}\left\langle f-\pi\right\rangle_{3Q},
\end{equation}
so that $\flt{f}_{0,Q}=\langle f\rangle_{3Q}$, and $\flt{f}_{1,Q}$ is comparable to the mean oscillation of $f$ on $3Q$ \cite{campanato63}. Weighted bounds are stated in terms of the \emph{maximal percentile oscillations}
\begin{equation}\label{e:maxfns}
\mathcal N_mf\coloneqq\sup_{Q}\nat{f}_{m,Q}\ind_Q,\qquad \mathcal N^sg\coloneqq\sup_Q\ell(Q)^s\left(\M^\tau_{3Q}g\right)^{*,Q}(r)\ind_Q,\quad s\geq0,
\end{equation}
the suprema being taken over all cubes. As $\PM^0_{3Q}=0$, $\mathcal N^0=\mathcal N_0$, thus $\mathcal N^s$ is a fractional version of $\mathcal N_0$, and it applies to vector valued $g$ as well, such as $g=\nabla^mf$, by \eqref{e:msigmavec}. Likewise, \eqref{e:po}, \eqref{e:campanato} and $\mathcal N_m$ apply to vector valued $f$, with vectors of polynomials in place of polynomials, and with the sum of the moduli of the components as the modulus $|f|$. For instance,
\[
\nat{\nabla^\kappa f}_{n,Q}=\left(\M^\tau_{3Q}\left[\nabla^\kappa f-\PM^n_{3Q}\nabla^\kappa f\right]\right)^{*,Q}(r),
\]
where $\PM^n_{3Q}$ acts componentwise and $\M^\tau_{3Q}$ is the sum \eqref{e:msigmavec} over the components. As for \eqref{e:po}, we write $\mathcal N^{[r]}_m$ when the percentile level has to be indicated.
\begin{lemma}\label{l:po}
Let $m\geq 0$ be an integer. The following hold for all cubes $Q$ and all $f\in L^1_{\mathrm{loc}}(\R^d)$, scalar, or vector valued with an admissible number of components, with the conventions of \S\ref{ss:perc}; \emph{(i)--(iii)} hold at every level $0<r<1$.
\begin{itemize}
\item[\emph{(i)}] $\nat{f+\pi}_{m,Q}=\nat{f}_{m,Q}$ for all $\pi\in \P^{m-1}$.
\item[\emph{(ii)}] $\displaystyle \nat{f}_{m,Q}\lesssim\inf_{\pi\in \P^{m-1}}\nat{f-\pi}_{0,Q}$. In particular, $\nat{f}_{n,Q}\lesssim\nat{f}_{m,Q}$ for all integers $n\geq m$.
\item[\emph{(iii)}] $\displaystyle\nat{f}_{m,Q}\lesssim \frac 1r\flt{f}_{m,Q}=\frac1r\inf_{\pi\in \P^{m-1}}\langle f-\pi\rangle_{3Q}$.
\item[\emph{(iv)}] If $0<r<\frac12$ and $|Q\cap \supp f|\leq \frac14|Q|$, then $\|\PM^m_{3Q}f\|_{L^\infty(3Q)}\lesssim\nat{f}_{m,Q}$, and $|\varphi(f)|\lesssim\nat{f}_{m,Q}$ for all $\varphi\in\Phi^{\tau,1}(Q)$ supported in $Q$.
\item[\emph{(v)}] If $0<r<1$, $A\geq1$, and $S\subseteq P$ are cubes with $\ell(P)\leq A\ell(S)$, then $\natl{f}{r}_{m,S}\lesssim\natl{f}{A^{-d}r}_{m,P}$, the implied constant being independent of $A$.
\end{itemize}
\end{lemma}
\begin{remark}\label{r:Nfrac}
By Lemma \ref{l:po}(iii) and the Bramble--Hilbert  \cite{BH70} estimate \[
\inf_{\pi\in\P^{m-1}}\langle f-\pi\rangle_{3Q}\lesssim\ell(Q)^m\langle|\nabla^mf|\rangle_{3Q},\] there holds $\nat{f}_{m,Q}\lesssim r^{-1}\ell(Q)^m\langle|\nabla^mf|\rangle_{3Q}$ for $m\geq1$ and $f\in W^{m,1}_{\mathrm{loc}}(\R^d)$. Therefore, with the fractional maximal functions of \S\ref{ss:classes},
\begin{equation}\label{e:Nfrac}
\mathcal N^sg\lesssim\frac1r\M_s\left(|g|\right),\quad 0\leq s<d,\qquad\qquad\mathcal N_mf\lesssim\frac1r\M_m\left(|\nabla^mf|\right),\quad 1\leq m<d.
\end{equation}
Combined with the weighted inequalities for $\M_s$, $0<s<d$, of Muckenhoupt and Wheeden \cite{MW74},  weighted bounds with right hand sides involving $\mathcal N^s$ and $\mathcal N_m$  lead to usual weighted Sobolev inequalities.
\end{remark}
\begin{remark}\label{r:axes}  Let $m\geq1$, let $Q$ be the unit cube centered at the origin, $\pi\in\P^{m-1}$ with $\|\pi\|_{L^\infty(Q)}=1$, and $a_\varepsilon\coloneqq\varepsilon^{-d}a(\varepsilon^{-1}\cdot)$, where $a\in\mathcal{C}^\infty_c(Q)$ has vanishing mean and $\|a\|_1=1$. For $g_\varepsilon\coloneqq\pi+a_\varepsilon$ and $0<\varepsilon<\varepsilon_0$, with $\varepsilon_0$ admissible,
\[
\flt{g_\varepsilon}_{m,Q}\gtrsim1,\qquad\nat{g_\varepsilon}_{0,Q}\gtrsim1,\qquad \nat{g_\varepsilon}_{m,Q}\lesssim r^{-2}\varepsilon^{\min\{\tau,1\}}.
\]
The first two bounds follow from \eqref{e:polylarge}: for the first, separate the cases of a small or a large approximating polynomial; for the second, use the size of $\pi$ off the support of $a_\varepsilon$ and Lemma \ref{l:diff}. For the third, set $s\coloneqq\min\{\tau,1\}$. The vanishing mean gives $\M^\tau_{3Q}a_\varepsilon(x)\lesssim\varepsilon^s|x|^{-d-s}$ for $|x|\gtrsim\varepsilon$. If $\varepsilon^d\leq cr$, with $c$ admissible, discard the ball of measure $r|Q|$ centered at the origin and use Lemma \ref{l:po}(i)--(ii) to obtain $\nat{g_\varepsilon}_{m,Q}\lesssim r^{-1-s/d}\varepsilon^s\leq r^{-2}\varepsilon^s$. Otherwise, Lemma \ref{l:po}(iii) gives $\nat{g_\varepsilon}_{m,Q}\lesssim r^{-1}\lesssim r^{-2}\varepsilon^s$. By contrast, the average of $\M^\tau_{3Q}[g_\varepsilon-\PM^m_{3Q}g_\varepsilon]$ on $Q$ is $\gtrsim1$, by Lemma \ref{l:diff}, as $\langle a_\varepsilon\rangle_Q=1$ and $|\PM^m_{3Q}a_\varepsilon|\lesssim\varepsilon$.
\end{remark}

\subsection{Hardy norms}\label{ss:hardymain}
For fixed $p$ and $w$, the weighted norms of the maximal percentile oscillations $\mathcal N_mf$ are comparable to smooth grand maximal norms, which are weighted Hardy norms in a range of $p$ and $w$ depending on $\tau$ \cite{ST89}. For instance, $\|\M^\tau_{\R^d}f\|_{L^p(\R^d)}$ is equivalent to the norm of $f$ in the Hardy space $H^p(\R^d)$ \cite{FS72} for $\frac d{d+\tau}<p\leq1$. Testing \eqref{e:max} with the cube $R$ itself shows that $\nat{f}_{0,R}\geq|\varphi(f)|$ whenever $\varphi\in\Phi^{\tau,1}(R)$ is supported in $R$, whence
\begin{equation}
\label{e:pwMN}
\M^\tau_{\R^d}f\leq\mathcal N_0f.
\end{equation}
Conversely, by Lemma \ref{l:po}(ii) and Chebyshev's inequality, for every $s>0$,
\[
\nat{f}_{m,Q}\lesssim\nat{f}_{0,Q}\leq r^{-\frac1s}\left\langle\left(\M^\tau_{\R^d}f\right)^s\right\rangle_Q^{\frac1s}.
\]
Choosing $s$ with $w\in A_{p/s}$, it follows that $\|\mathcal N_mf\|_{L^p(w)}\lesssim\|\M^\tau_{\R^d}f\|_{L^p(w)}$ for all $0<p<\infty$ and $w\in A_\infty$, with a constant depending on $p$, $r$ and $w$. The next proposition is the reverse inequality, with an explicit dependence on the weight; it is proved in \S\ref{ss:hardy}, through a linear sparse bound modulo polynomials for $\M^\tau_{\R^d}$, Proposition \ref{p:grand}.
\begin{proposition}\label{c:hardynorm}
Let $m\geq0$ be an integer, $0<r\leq r_0$, $0<p<\infty$ and $w\in A_\infty$. For every $f\in L^2(\R^d)$ with compact support,
\begin{equation}\label{e:hardynorm}
\big\|\M^\tau_{\R^d}f\big\|_{L^p(w)}\lesssim_p[w]_{A_\infty}\left\|\mathcal N_mf\right\|_{L^p(w)},
\end{equation}
where the implied constant grows at most linearly in $p$ as $p\to\infty$. Consequently, for fixed $p$ and $w$, the quantities $\|\mathcal N_nf\|_{L^p(w)}$, $n\geq0$, are comparable to $\|\M^\tau_{\R^d}f\|_{L^p(w)}$; in particular, $\|\mathcal N_nf\|_{L^p(\R^d)}\sim\|f\|_{H^p(\R^d)}$ for $\frac d{d+\tau}<p\leq1$.
\end{proposition}
\begin{remark} \label{r:weightuseful}
For $m=0$, \eqref{e:hardynorm} holds with constant one, by the pointwise bound \eqref{e:pwMN}. For $m\geq1$, it is a Fefferman--Stein inequality \cite{FS72} modulo polynomials.
Unlike the latter, \eqref{e:hardynorm} remains meaningful for $p\leq1$, where the sharp maximal function of a nonzero compactly supported function does not belong to $L^p(\R^d)$, being at least a constant multiple of $|x|^{-d}$ at infinity. The factor $[w]_{A_\infty}$ cannot be improved for $m\geq1$. Indeed, as $|\Omega|\lesssim\M^\tau_{\R^d}\Omega$ and ${\mathcal N_m\Omega\lesssim\mathcal N_1\Omega}$, it suffices to test \eqref{e:hardynorm} on the function $\Omega=\Omega_{x_0,N}$ and the cubes $Q_j$ of \S\ref{ss:signedlog}, with the weight $w\coloneqq w_{1/N}$ defined there, for which $[w]_{A_\infty}\sim N$, see Lemma \ref{l:power}. As $\Omega=N$ on $Q_1$, and by Lemma \ref{l:signedlog}(a), polar coordinates give, for $N$ large depending on $p$,
\[
\|\Omega\|_{L^p(w)}\gtrsim_p Nw(Q_N)^{\frac1p},\qquad\|\mathcal N_1\Omega\|_{L^p(w)}\lesssim_p r^{-1}w(Q_N)^{\frac1p}.
\]
As $\mathcal N_0\Omega\geq\M^\tau_{\R^d}\Omega\gtrsim|\Omega|$, by \eqref{e:pwMN} and Lemma \ref{l:diff}, the ratio of the norms of $\mathcal N_0\Omega$ and $\mathcal N_1\Omega$ in $L^p(w)$, which are comparable for fixed $p$ and $w$, is thus $\gtrsim_pr[w]_{A_\infty}$: the comparability of Proposition \ref{c:hardynorm} does not preserve the quantitative dependence on the weight, which distinguishes the residual orders in Theorem \ref{t:HS}.
\end{remark}
\section{Main results}\label{s:main}
This section introduces the Calder\'on--Zygmund operators with their symbols, and states the main results; their relation with previous work is discussed in \S\ref{ss:previous}.

\subsection{Wavelets}\label{ss:wavelets}

Throughout, $k\geq0$ is an integer, the order of the Calder\'on--Zygmund operators below, and $0<\delta\leq1$ is the H\"older exponent of their kernels. By \cite{DGWdom}*{Proposition 2.3 and Remark 2.2}, there is a family
\[
\left\{\psi_S\in c \Psi^{k+1,\Subset}(S):S\in\mathcal D\right\}
\]
of real valued functions, $c$ being an admissible constant, such that $\{|S|^{\frac12}\psi_S:S\in\mathcal D\}$ is an orthonormal basis of $L^2(\R^d)$. As usual, each $S$ carries $2^d-1$ wavelets, and sums over $S$ include this index. Thus
\begin{equation}\label{e:onb}
f=\sum_{S\in\mathcal D}|S|\psi_S(f)\psi_S,\qquad \langle f,g\rangle=\sum_{S\in\mathcal D}|S|\psi_S(f)\psi_S(g),\qquad f,g\in L^2(\R^d),
\end{equation}
where the pairing $\langle\cdot,\cdot\rangle$ is bilinear, so that $\psi_S(f)=\langle f,\psi_S\rangle$.
Fix a smooth function $\omega$ with $0\leq\omega\leq1$, equal to one on $[-\frac34,\frac34]^d$ and supported in $(-1,1)^d$, so that, in the notation \eqref{e:norm1}, $\omega^\infty_Q=1$ on $\frac32Q$ and $\supp\omega^\infty_Q\Subset2Q$.

\subsection{Calder\'on--Zygmund operators}\label{ss:forms}

As in \cite{DGWdom}*{Remark 2.6}, the operators below act on spaces $\mathcal U_j\subseteq L^2(\R^d)$, $j\geq0$, defined in \S\ref{ss:czloc} through the wavelet coefficients. They decrease in $j$; for $j\leq k$, they contain $\mathcal{C}^\infty_c(\R^d)$ and the classes $\Phi^{j+\delta,\Subset}(R)$, and $\partial^\alpha$ maps $\mathcal U_{|\alpha|}$ into $\mathcal U_0$, by Lemma \ref{l:Uj}. As $\mathcal U_0$ is strictly smaller than $L^2(\R^d)$, the continuity required in Definition \ref{d:czo} is weaker than boundedness on $L^2(\R^d)$.

\begin{definition}\label{d:czo}
Let $0\leq j\leq k$. A \emph{Calder\'on--Zygmund operator of order $(k,j)$} is a continuous linear operator $T$ from $\mathcal U_0$ to the dual $\mathcal U_j'$ of $\mathcal U_j$ with the following properties, $\langle Tf,g\rangle$ standing for the action of $Tf$ on $g\in\mathcal U_j$.
\begin{enumerate}[label=(\roman*)]
\item The \emph{weak boundedness property} holds, that is
\begin{equation}\label{e:wbp}
|R|\,\ell(R)^j\left|\langle T\zeta_R,\xi_R\rangle\right|\leq1,\qquad \zeta_R\in\Phi^{k+\delta,\Subset}(R),\quad \xi_R\in\Phi^{j+\delta,\Subset}(R),\quad R\in\mathcal D.
\end{equation}
\item For $f\in\mathcal U_0$ and $g\in\mathcal U_j$ with disjoint supports,
\[
\langle Tf,g\rangle=\iint K(x,y)f(y)g(x)\,\d y\,\d x,
\]
with a kernel $K$ satisfying, for $0\leq i\leq k-j$, $x\neq y$, and $|x-x'|+|y-y'|\leq\frac12|x-y|$,
\begin{equation}\label{e:czkernel}
\left|\nabla^i_xK(x,y)\right|\leq|x-y|^{-d-i-j},\qquad \left|\nabla^i_xK(x',y')-\nabla^i_xK(x,y)\right|\leq\frac{\left(|x-x'|+|y-y'|\right)^\delta}{|x-y|^{d+i+j+\delta}}.
\end{equation}
\end{enumerate}
A \emph{Calder\'on--Zygmund operator of order $k$} is one of order $(k,0)$.
\end{definition}

The \emph{derivatives} of a Calder\'on--Zygmund operator $T$ of order $k$ are defined by
\begin{equation}\label{e:dalpha}
\langle\partial^\alpha Tf,g\rangle\coloneqq(-1)^{|\alpha|}\left\langle Tf,\partial^\alpha g\right\rangle,\qquad f\in\mathcal U_0,\quad g\in\mathcal U_{|\alpha|},\quad |\alpha|\leq k,
\end{equation}
cf.\ \cite{DGWdom}*{(5.2) and (5.11)}.
The setting follows \cite{DGWdom}*{\S5}: properties (i) and (ii) are those of \cite{DGWdom}*{Definition 5.1} for $\Omega=\R^d$, up to the value of the constants, which we normalize to one, and the continuity of $T$ is the counterpart of that of the forms there. For $|\alpha|=\kappa$, $\partial^\alpha T$ is a constant multiple of a Calder\'on--Zygmund operator of order $(\kappa,\kappa)$, by Lemma \ref{l:reduce}(i).

\begin{definition}\label{d:symbols}
A \emph{symbol} is a sequence $b=\{b_S:S\in\mathcal D\}$, indexed correspondingly to the wavelets $\psi_S$. Let $X\in\{\mathrm{BMO}(\R^d),L^\infty(\R^d)\}$ in the discussion that follows.
\begin{enumerate}[label=(\roman*)]
\item The symbol $b$ \emph{vanishes} if $b_S=0$ for all $S\in\mathcal D$.
\item The symbol $b$ \emph{belongs to} $X$ if there is a function $\mathfrak b\in X$ with $\psi_S(\mathfrak b)=b_S$ for all $S\in\mathcal D$. Such a function is unique modulo constants, and the norm of $b$ is
\[
\|b\|_X\coloneqq\inf_c\,\|\mathfrak b-c\|_X,
\]
the infimum being taken over all constants $c$.
\item The \emph{high frequency part} of $b$ is the series \[\sum_{\ell(S)\leq1}|S|b_S\psi_S.\] The symbol $b$ \emph{belongs locally to} $X$ if its high frequency part is the wavelet expansion of a function $\mathfrak b\in X$, that is, if the symbol $\{b_S\ind_{\{\ell(S)\leq1\}}:S\in\mathcal D\}$ belongs to $X$. The \emph{local norm} $\|b\|_{X,\mathrm{loc}}$ of $b$ is the norm of the latter symbol.
\end{enumerate}
\end{definition}

Let $0\leq\kappa\leq k$ and let $\mathfrak L$ be a Calder\'on--Zygmund operator of order $(\kappa,\kappa)$. For $S\in\mathcal D$ and a polynomial $\pi$ of degree at most $\kappa$, define, with the cutoffs $\omega^\infty_Q$ of \S\ref{ss:wavelets},
\begin{equation}\label{e:polyaction}
\langle\mathfrak L\pi,\psi_S\rangle\coloneqq\lim_{A\to\infty}\left\langle\mathfrak L\left(\pi\,\omega^\infty_{AS}\right),\psi_S\right\rangle,\qquad \langle\mathfrak L\psi_S,\mathbf 1\rangle\coloneqq\lim_{A\to\infty}\left\langle\mathfrak L\psi_S,\omega^\infty_{AS}\right\rangle.
\end{equation}
The limits exist by Lemma \ref{l:cutoff}(ii)--(iii), by which the dilates $AS$ in the first limit may be replaced by the dilates $AQ$ of any cube $Q$, and they coincide with the actions of \cite{DGWdom}*{(5.3)--(5.4)}. The \emph{symbols of order $\kappa$} of a Calder\'on--Zygmund operator $T$ of order $k$ are, for $|\alpha|=\kappa$ and $S\in\mathcal D$,
\begin{equation}\label{e:symbols}
b^{\alpha,\gamma}_S\coloneqq\left\langle\partial^\alpha T\left[(\cdot-c(S))^\gamma\right],\psi_S\right\rangle,\quad |\gamma|\leq\kappa,\qquad b^{\alpha,\star}_S\coloneqq\langle\partial^\alpha T\psi_S,\mathbf 1\rangle,
\end{equation}
cf.\ \cite{DGWdom}*{\S5.8 and \S6.2}; those with $|\gamma|<\kappa$ are the \emph{lower symbols}, those with $0<|\gamma|<\kappa$ the \emph{intermediate symbols}, and those with $|\gamma|=\kappa$ the \emph{top symbols}. When $\kappa=0$, $b^{0,0}_S=\psi_S(T1)$ and $b^{0,\star}_S=\psi_S(T^*1)$ in the usual sense, while $b^{\alpha,\star}$ vanishes for $\kappa\geq1$, by Lemma \ref{l:cutoff}(iii). By Lemma \ref{l:size}(i),
\begin{equation}\label{e:size}
\left|b^{\alpha,\gamma}_S\right|\lesssim\ell(S)^{|\gamma|-\kappa},\qquad |\alpha|=\kappa,\quad |\gamma|\leq\kappa,\quad S\in\mathcal D,
\end{equation}
so that the top symbols are bounded sequences and the lower symbols are small at large scales. When the symbols of order $\kappa$ of $T$ belong, or belong locally, to $X$, we write $\mathfrak b^{\alpha,\gamma}$, $\mathfrak b^{\alpha,\star}$ for the corresponding functions $\mathfrak b$, cf.\ \cite{DGWdom}*{(4.1)}. The local conditions impose nothing on the low frequency part $\{b^{\alpha,\gamma}_S:\ell(S)>1\}$, which is bounded by \eqref{e:size}.
By Lemma \ref{l:size}(ii), a lower symbol belongs to $X$ if and only if it belongs locally to $X$.

The symbols are the coefficients of the paraproducts in the wavelet representation of $\partial^\alpha T$, Proposition \ref{p:rep}, in which the symbol $b^{\alpha,\gamma}$ acts on the channel $\partial^\gamma f$.

\subsection{Sparse bounds modulo polynomials}\label{ss:czmain}

Fix $k\geq0$ and a Calder\'on--Zygmund operator $T$ of order $k$.
Throughout, $0\leq \kappa=|\alpha|\leq k$ is the order of  the applied derivative $\partial^\alpha$, $m\geq\kappa$ is the polynomial subtraction index, $\delta$ is the H\"older exponent of the kernel. The percentile oscillations \eqref{e:po} and the functions \eqref{e:maxfns} are computed with smoothness parameter $\tau$ and percentile level $r$. When a statement assumes qualitative conditions on the symbols of $T$, $\mathfrak B(T)$ denotes the sum of the corresponding symbol norms. The kernel and weak boundedness constants of $T$ are normalized to one in Definition \ref{d:czo}, and admissible constants depend on $T$ only through its order $k$ and the H\"older exponent $\delta$ of its kernel. The dependence on $T$ is carried by $\mathfrak B(T)$, and by the constants of the inequalities for $T$ assumed in a statement, such as $C$ or $C_p(T)$, which are named when a conclusion depends on them. Unless stated otherwise, the implied constants in the conclusions are admissible multiples of $1+\mathfrak B(T)$, respectively of $1+C$, while in an assumed inequality $\lesssim$ allows a constant depending on $T$. In the norm inequalities for $\partial^\alpha Tf$, such as \eqref{e:PHS} and \eqref{e:czextreme}, it is part of the statement that $\partial^\alpha Tf$ is a locally integrable function.

Theorem \ref{t:HS} shows that, in the homogeneous setting, the conditions on the symbols \eqref{e:symbols}, the pointwise sparse bounds and the weighted Hardy--Sobolev inequalities with linear dependence on $[w]_{A_\infty}$ are equivalent. Theorem \ref{t:jet} characterizes, in terms of the symbols, the inhomogeneous pointwise sparse bounds modulo polynomials for the derivatives $\partial^\alpha T$, $|\alpha|=\kappa$. Theorem \ref{t:HSjet} gives the corresponding weighted Hardy--Sobolev inequalities under weaker conditions on the intermediate symbols, and Proposition \ref{p:rigidity} shows that, once $\kappa\geq2$, the norm theory is strictly weaker than the pointwise one. On their right hand sides, the channels $\partial^\gamma f$ enter through their percentile oscillations \eqref{e:po} modulo polynomials in $\P^{m-|\gamma|-1}$.
For $1<p<\infty$ and $A_p$ weights, sparse bounds by averages suffice: they yield the sharp weighted Sobolev inequalities of Proposition \ref{c:czweighted}, in \S\ref{ss:Ap}. By Lemma \ref{l:po}(iii), the sparse bounds of Theorems \ref{t:HS} and \ref{t:jet} imply their Campanato versions, with the oscillations \eqref{e:campanato} in place of the percentile oscillations, and the proofs of necessity only use the latter.

\subsection{Homogeneous sparse and weighted Hardy--Sobolev bounds}\label{ss:HS}

In Theorem \ref{t:HS}, the requirement that the right hand side of the sparse bound only involve the derivatives of order $\kappa$ of $f$ forces the lower symbols to vanish and the conditions on the top symbols to be global.

\begin{theorem}\label{t:HS}
Let $T$ be a Calder\'on--Zygmund operator of order $k$, $0\leq\kappa\leq k$, $|\alpha|=\kappa$, $0<\tau<\delta$, $0<r\leq r_0$, and let $m\geq\kappa$ be an integer. The following are equivalent.
\begin{enumerate}[label=(\roman*)]
\item The symbols $b^{\alpha,\gamma}$ vanish for $|\gamma|<\kappa$, and, for $|\gamma|=\kappa$, they belong to $\mathrm{BMO}(\R^d)$ if $m=\kappa$, and to $L^\infty(\R^d)$ if $m>\kappa$; if $\kappa=0$, the symbol $b^{0,\star}$ belongs to $\mathrm{BMO}(\R^d)$.
\item For every $f\in \mathcal{C}^\infty_c(\R^d)$, the distribution $\partial^\alpha Tf$ is a locally integrable function, and there exists a sparse collection $\mathcal S$ of cubes of $\R^d$ such that
\begin{equation}\label{e:Dfine}
\left|\partial^\alpha Tf\right|\lesssim\sum_{S\in\mathcal S}\sum_{|\beta|=\kappa}\nat{\partial^\beta f}_{m-\kappa,S}\ind_S\qquad\textrm{a.e.\ on }\R^d.
\end{equation}
\item   For each $0<p<\infty$,  one has the homogeneous weighted Hardy--Sobolev inequality
\begin{equation}\label{e:PHS}
\left\|\partial^\alpha Tf\right\|_{L^p(w)}\lesssim_p [w]_{A_\infty}\sum_{|\beta|=\kappa}\left\|\mathcal N_{m-\kappa}\left(\partial^\beta f\right)\right\|_{L^p(w)},\qquad w\in A_\infty,\quad f\in \mathcal{C}^\infty_c(\R^d).
\end{equation}
\end{enumerate}
Quantitatively, if $C$ is the infimum of the implied constants in \emph{(ii)}, then
\[
1+C\lesssim1+ \mathfrak B(T)\lesssim\frac{1+C}r,\] and, under \emph{(i)}, the implied constants in \emph{(iii)} are $\lesssim_p1+\mathfrak B(T)$, and grow at most linearly in $p$ as $p\to\infty$; conversely, if $C_p$ is the implied constant in \emph{(iii)}, then $1+\mathfrak B(T)\lesssim r^{-2}(1+C_2+C_{2d})$.
\end{theorem}

Corollaries \ref{c:czpot}--\ref{c:hardy} are proved in \S\ref{ss:czcor} and Section \ref{s:HS}.
\begin{corollary}\label{c:czpot}
In the setting of Theorem \ref{t:HS}, assume \emph{(i)}. Then, for each $f\in \mathcal{C}^\infty_c(\R^d)$, the sparse collection $\mathcal S$ of \emph{(ii)} satisfies the Sobolev sparse bound
\begin{equation}\label{e:D}
\left|\partial^\alpha Tf\right|\lesssim\sum_{S\in\mathcal S}\ell(S)^{m-\kappa}\left(\M^\tau_{3S}\left[\nabla^mf\right]\right)^{*,S}(r)\ind_S\qquad\textrm{a.e.\ on }\R^d,
\end{equation}
and, for $0<p<\infty$ and $w\in A_\infty$,
\begin{equation}\label{e:czMW}
\left\|\partial^\alpha Tf\right\|_{L^p(w)}\lesssim_p[w]_{A_\infty}\left\|\mathcal N^{m-\kappa}\left[\nabla^mf\right]\right\|_{L^p(w)},
\end{equation}
with an implied constant growing at most linearly in $p$ as $p\to\infty$.
\end{corollary}
By Lemma \ref{l:po}(iii), with $\nabla^mf$ and $0$ in place of $f$ and $m$, $(\M^\tau_{3S}[\nabla^mf])^{*,S}(r)\lesssim r^{-1}\langle|\nabla^mf|\rangle_{3S}$, and the corresponding fractional sparse sums are dominated by Riesz potentials, see \eqref{e:band} and Lemma \ref{l:sparsepot}. Hence, if $0<m-\kappa<d$, the Sobolev sparse bound \eqref{e:D} implies
\begin{equation}\label{e:czpot}
\left|\partial^\alpha Tf(x)\right|\lesssim\frac1r\int_{\R^d}\frac{|\nabla^mf(y)|}{|x-y|^{d-m+\kappa}}\,\d y\qquad\textrm{for almost every }x\in\R^d,
\end{equation}
and, for $0\leq m-\kappa<d$, by \eqref{e:Nfrac}, the inequality \eqref{e:czMW} implies the same inequality with $r^{-1}\M_{m-\kappa}(|\nabla^mf|)$ in place of $\mathcal N^{m-\kappa}[\nabla^mf]$.
\begin{corollary}\label{c:endpoint}
In the setting of Theorem \ref{t:HS}, let $s\coloneqq m-\kappa<d$.
\begin{enumerate}[label=(\alph*)]
\item If $s=0$, then \emph{(i)} holds if and only if $\|\partial^\alpha Tf\|_{L^p(\R^d)}\lesssim_p\|\nabla^\kappa f\|_{L^p(\R^d)}$ for all $1<p<\infty$ and $f\in \mathcal{C}^\infty_c(\R^d)$.
\item If $0<s<d$, then \emph{(i)} holds if and only if the following hold:
\begin{equation}\label{e:Linfty}
\left\|\partial^\alpha Tf\right\|_{L^\infty(\R^d)}\lesssim\left\|\nabla^mf\right\|_{L^{\frac ds,1}(\R^d)},\qquad f\in \mathcal{C}^\infty_c(\R^d),
\end{equation}
and, if $\kappa=0$, $T$ is bounded on $L^2(\R^d)$.
\end{enumerate}
\end{corollary}
When $\kappa=0$, the $L^2$ boundedness in (b) is not implied by \eqref{e:Linfty}, see Example \ref{r:D}.
\begin{corollary}\label{c:hardy}
In the setting of Theorem \ref{t:HS} with $m=\kappa$, assume \emph{(i)}, and let $f\in \mathcal{C}^\infty_c(\R^d)$. For $0<p<\infty$ and $w\in A_\infty$, with a constant depending on $p$ and $w$,
\[
\left\|\partial^\alpha Tf\right\|_{L^p(w)}\lesssim\left\|\M^\tau_{\R^d}\left[\nabla^\kappa f\right]\right\|_{L^p(w)}.
\]
In particular, for $\frac d{d+\tau}<p\leq1$,
\[
\left\|\partial^\alpha Tf\right\|_{L^p(\R^d)}\lesssim\sum_{|\beta|=\kappa}\left\|\partial^\beta f\right\|_{H^p(\R^d)}.
\]
\end{corollary}

\subsection{Inhomogeneous sparse and weighted Hardy--Sobolev bounds} \label{ss:inhmain}
In the inhomogeneous setting, the sparse bounds involve the two extreme channels $f$ and $\nabla^\kappa f$, as in an inhomogeneous Sobolev norm: the lower symbols need not vanish, and the conditions on the top symbols are local.
 \begin{theorem}\label{t:jet}
Let $T$ be a Calder\'on--Zygmund operator of order $k$, $1\leq\kappa\leq k$, $|\alpha|=\kappa$, $0<\tau<\delta$, $0<r\leq r_0$, and let $m\geq\kappa$ be an integer. The following are equivalent.
\begin{enumerate}[label=(\roman*)]
\item The lower symbols $b^{\alpha,\gamma}$, $|\gamma|<\kappa$, belong to $L^\infty(\R^d)$, and the top symbols $b^{\alpha,\gamma}$, $|\gamma|=\kappa$, belong locally to $\mathrm{BMO}(\R^d)$ if $m=\kappa$, and locally to $L^\infty(\R^d)$ if $m>\kappa$.
\item For every $f\in \mathcal{C}^\infty_c(\R^d)$, the distribution $\partial^\alpha Tf$ is a locally integrable function, and there exists a sparse collection $\mathcal S$ of cubes of $\R^d$ such that
\begin{equation}\label{e:jet}
\left|\partial^\alpha Tf\right|\lesssim\sum_{S\in\mathcal S}\left(\nat{f}_{m,S}+\nat{\nabla^\kappa f}_{m-\kappa,S}\right)\ind_S\qquad\textrm{a.e.\ on }\R^d.
\end{equation}
\end{enumerate}
Quantitatively, if  $C$ is the infimum of the implied constants in \emph{(ii)}, then
\[
1+C\lesssim1+ \mathfrak B(T)\lesssim\frac{1+C}r.\]
\end{theorem}
With $T$, $\kappa$, $\alpha$, $\tau$ and $m$ as in Theorem \ref{t:jet}, and $0<r\leq r_0$, we say that $T$ satisfies \eqref{e:czextreme} if, for every $0<p<\infty$, there exists $C_p(T)$ such that, for all $w\in A_\infty$ and $f\in \mathcal{C}^\infty_c(\R^d)$,
\begin{equation}\label{e:czextreme}
\left\|\partial^\alpha Tf\right\|_{L^p(w)}\leq C_p(T)[w]_{A_\infty}\bigg(\left\|\mathcal N_mf\right\|_{L^p(w)}+\sum_{|\beta|=\kappa}\left\|\mathcal N_{m-\kappa}\left(\partial^\beta f\right)\right\|_{L^p(w)}\bigg).
\end{equation}
\begin{theoremx}\phantomsection\makeatletter\def\@currentlabel{B$'$}\makeatother\label{t:HSjet}
Let $T$ be a Calder\'on--Zygmund operator of order $k$, $1\leq\kappa\leq k$, $|\alpha|=\kappa$, $0<\tau<\delta$, $0<r\leq r_0$, and let $m\geq\kappa$ be an integer. Assume that \begin{enumerate}[label=(\alph*)]
\item $b^{\alpha,0}$ belongs to $L^\infty(\R^d)$,\item the intermediate symbols $b^{\alpha,\gamma}$, $0<|\gamma|<\kappa$, belong to $\mathrm{BMO}(\R^d)$, \item the top symbols $b^{\alpha,\gamma}$, $|\gamma|=\kappa$, belong locally to $\mathrm{BMO}(\R^d)$ if $m=\kappa$, and locally to $L^\infty(\R^d)$ if $m>\kappa$.
\end{enumerate} Then $T$ satisfies \eqref{e:czextreme} with $C_p(T)\lesssim_p 1+ \mathfrak B(T)$, where the implied constants grow at most linearly in $p$ as $p\to\infty$.
\end{theoremx}
In the proof of Theorem \ref{t:HSjet}, the intermediate derivatives are absent from the right hand side, as the corresponding paraproducts are controlled by the two extreme channels by integrating by parts, see Lemma \ref{l:sandwich}, while their symbols enter only through the John--Nirenberg inequality, see Lemma \ref{l:middle}. Therefore $\mathrm{BMO}$ conditions suffice for the intermediate symbols. If $m=\kappa$, Theorem \ref{t:HSjet} with $w\equiv1$ and Proposition \ref{c:hardynorm} give, for $\frac d{d+\tau}<p\leq1$,
\[
\left\|\partial^\alpha Tf\right\|_{L^p(\R^d)}\lesssim_p\|f\|_{H^p(\R^d)}+\sum_{|\beta|=\kappa}\left\|\partial^\beta f\right\|_{H^p(\R^d)},
\]
with no intermediate derivative of $f$. The converse information is collected in the next proposition.
\begin{proposition}\label{p:rigidity}
Let $1\leq\kappa\leq k$, $|\alpha|=\kappa$, $0<\tau<\delta$, $0<r\leq r_0$, and let $m\geq\kappa$ be an integer.
\begin{enumerate}[label=(\alph*)]
\item Let $T$ be a Calder\'on--Zygmund operator of order $k$. If $T$ satisfies \eqref{e:czextreme}, then $b^{\alpha,0}$ belongs to $L^\infty(\R^d)$, and the top symbols belong locally to $\mathrm{BMO}(\R^d)$, with norms $\lesssim(1+C_2(T))/r$.
\item Let $T$ be a Calder\'on--Zygmund operator of order $k$. If $\kappa=m=1$, then $T$ satisfies \eqref{e:czextreme} if and only if the conditions \emph{(i)} of Theorem \ref{t:jet} hold.
\item If $\kappa\geq2$ and $0<|\gamma|<\kappa$, there exists a Calder\'on--Zygmund operator $T$ of order $k$ such that $T$ satisfies \eqref{e:czextreme} for every $m\geq\kappa$, while $b^{\alpha,\gamma}$ does not belong to $L^\infty(\R^d)$; in particular, the conditions of Theorem \ref{t:jet} fail for every $m\geq\kappa$.
\end{enumerate}
\end{proposition}
The questions left open by Proposition \ref{p:rigidity} are discussed in Remark \ref{r:open}.

\begin{corollary}\label{c:sobjet}
Let $T$ be a Calder\'on--Zygmund operator of order $k\geq1$, $0<\tau<\delta$, $0<r\leq r_0$, and let $m\geq k$ be an integer. The following are equivalent.
\begin{enumerate}[label=(\roman*)]
\item The conditions \emph{(i)} of Theorem \ref{t:jet} hold for every $1\leq\kappa\leq k$ and $|\alpha|=\kappa$.
\item For every $f\in \mathcal{C}^\infty_c(\R^d)$, the distributions $\partial^\alpha Tf$, $1\leq|\alpha|\leq k$, are locally integrable functions, and there exists a sparse collection $\mathcal S$ of cubes of $\R^d$ such that
\begin{equation}\label{e:sobjet}
\sum_{1\leq|\alpha|\leq k}\left|\partial^\alpha Tf\right|\lesssim\sum_{S\in\mathcal S}\left(\nat{f}_{m,S}+\nat{\nabla^kf}_{m-k,S}\right)\ind_S\qquad\textrm{a.e.\ on }\R^d.
\end{equation}
\end{enumerate}
Quantitatively, if $C$ is the infimum of the implied constants in \emph{(ii)}, then $1+C\lesssim1+\mathfrak B(T)\lesssim(1+C)/r$, where $\mathfrak B(T)$ sums the symbol norms in \emph{(i)} over $1\leq|\alpha|\leq k$.
If, in addition, $T1\in L^\infty(\R^d)$ and $T^*1\in\mathrm{BMO}(\R^d)$, then \eqref{e:sobjet} holds with $|Tf|$ added to its left hand side.
\end{corollary}
 In terms of the symbols of positive order, the whole jet is thus characterized by the same conditions as each derivative separately. Corollary \ref{c:sobjet} is proved in \S\ref{ss:jetthm}.

\subsection{Relation with previous work}\label{ss:previous}

\subsubsection*{Cancellative sparse bounds}
The percentile sparse bound of \cite{CLR26} is the predecessor at order zero. In our setting, it is characterized by \(T1,T^*1\in\mathrm{BMO}(\R^d)\), by Theorem \ref{t:HS} with $\kappa=m=0$, which recovers the David--Journ\'e conditions. At the first positive order, the sparse bounds by mean oscillations of \cite{BB18} and \cite{lerner26}*{Theorem 1.1} assume \(T1=0\). For mean oscillations, the conditions
\[
T1\in L^\infty(\R^d),\qquad T^*1\in\mathrm{BMO}(\R^d)
\]
already suffice, by the decomposition \(Tf=\widetilde Tf+bf\), with \(\widetilde T1=0\) and \(b\in L^\infty(\R^d)\), of \cite{lerner26}*{Proposition 2.2}, and the sparse bound for the identity, cf.\ \cite{LLO22}*{Proposition 5.4}. For \(d\geq2\) they are also necessary, as such a sparse bound implies the \(L^2\) boundedness of \(T\) and, by the Poincar\'e inequality, \eqref{e:band} and Lemma \ref{l:sparsepot}, the pointwise Sobolev inequality \(|Tf|\lesssim I_1(|\nabla f|)\), which the same threshold characterizes \cite{lerner26}*{Theorem 1.2}. Theorem \ref{t:introI}, the case $\kappa=0$, $m=1$ of Theorem \ref{t:HS}, gives this characterization for the percentile oscillations, with the quantitative converse, and Theorems \ref{t:HS} and \ref{t:jet} extend the $\mathrm{BMO}$-to-$L^\infty$ distinction, channel by channel, to arbitrary derivatives and polynomial orders.

\subsubsection*{Pointwise Sobolev bounds}
The sparse Sobolev bound \eqref{e:D} is a local and cancellation-sensitive refinement of the pointwise estimates by Riesz potentials of \cites{HMP25a,HMP25b,HMP25c,lerner26}, valid for the derivatives of \(T\) and for arbitrary order. For \(0<m-\kappa<d\), the potential estimate \eqref{e:czpot} and its critical Lorentz endpoint \eqref{e:Linfty} characterize the same testing conditions as the sparse bound \eqref{e:Dfine}, with \(L^2\) boundedness also assumed when \(\kappa=0\), see Corollaries \ref{c:czpot} and \ref{c:endpoint}(b). The latter assumption is essential, as Example \ref{r:D} shows. For \(m\geq\kappa+2\), the finer bound \eqref{e:Dfine} also records the subtraction of polynomials of higher degree, whereas the potential estimate alone follows from the case \(m=\kappa+1\), the classical estimate \(|g|\lesssim I_1(|\nabla g|)\) applied to the subsequent derivatives, and the semigroup identity for Riesz potentials, provided \(m-\kappa<d\).

\subsubsection*{Testing and quantitative weighted inequalities} The polynomial testing symbols and the wavelet resolution used here are those of the global case of \cite{DGWdom}. These symbols characterize the pointwise sparse bounds and, in the homogeneous setting, the \(A_\infty\) Hardy--Sobolev inequalities, through conditions of \(\mathrm{BMO}\), \(L^\infty\), and vanishing type. Strong cancellation conditions, including the vanishing of lower-order testing data, also occur in the classical \(T1\) theory on Besov and Triebel--Lizorkin spaces; see, for instance, Han and Hofmann \cite{han-hofmann}, and \cite{BHYY26}*{Theorem 6.24} for matrix $A_\infty$ weights. In the homogeneous case, the present results recover and strengthen, at the sparse and Hardy levels, the sharp Sobolev \(A_p\) theory of \cite{DWW23}. Proposition \ref{c:czweighted} gives the sharp $A_p$ dependence under the present testing hypotheses; in the inhomogeneous setting, it assumes that the lower symbols belong to $\mathrm{BMO}(\R^d)$, a condition that boundedness at a single unweighted Sobolev exponent does not force, as Example \ref{r:DGW} shows.

With $\kappa=0$, the inequality \eqref{e:PHS} contains, at residual order zero, the Coifman--Fefferman inequality \cite{CF74}, with the linear \(A_\infty\) dependence of \cite{OPR13}*{Theorem 2.1}, and, at the first positive order and for the identity operator, the corresponding Fefferman--Stein sharp maximal estimate.

\section{Auxiliary lemmas}\label{s:aux}
This section contains the elementary facts used in the proofs, the properties of the smooth maximal functions and of the polynomial projections of Section \ref{s:po}, and the proof of Lemma \ref{l:po}.
Five elementary facts will be used repeatedly. First, if $R\subseteq S$ are cubes with $\ell(S)\leq A\ell(R)$, a change of variables shows that
\begin{equation}\label{e:change}
\Phi^{\sigma,1}(R) \subseteq C\Phi^{\sigma,1}(S),
\end{equation}
with $C$ depending on $A$ besides the admissible parameters; if moreover $\w R\subseteq S$, the elements of $\Phi^{\sigma,\Subset}(R)$ are supported in $S$. Second, for each $x\in \R^d$ and $i\in \Z$,
\begin{equation}\label{e:count}
\#\left\{R\in \mathcal D: \ell(R)=2^i,\, x\in \w R\right\}\leq (\w+1)^d.
\end{equation}
Third, if $P$ is a cube, $E\subseteq P$, and $0<s\leq \ell(P)$ is such that either $\diam E<s$ or $s=\ell(P)$, then there exists a cube $R$ with
\begin{equation}\label{e:enclosing}
E\subseteq R\subseteq P, \qquad \ell(R)=s.
\end{equation}
Fourth, if $\mathcal S$ is $\eta$-sparse, the cubes $S\in\mathcal S$ with $2^i\leq\ell(S)<2^{i+1}$ containing a given point lie in a cube of sidelength $2^{i+2}$, and their sets $G_S$ are pairwise disjoint, with $|G_S|\geq\eta2^{id}$, so that
\begin{equation}\label{e:band}
\#\left\{S\in\mathcal S:x\in S,\;2^i\leq\ell(S)<2^{i+1}\right\}\leq4^d\eta^{-1},\qquad x\in\R^d,\quad i\in\Z.
\end{equation}
Fifth, let $Q_*\coloneqq[-\frac12,\frac12)^d$. There is an admissible $c_0>0$ such that, for all cubes $Q$ and all nonzero $p\in\P^{m-1}$,
\begin{equation}\label{e:polylarge}
\left|\left\{x\in Q: |p(x)|\leq c_0\|p\|_{L^\infty(Q)}\right\}\right|\leq \tfrac14|Q|.
\end{equation}
By affine rescaling and homogeneity, it suffices to consider $Q=Q_*$ and $p$ in the compact set $\mathbb S\coloneqq\{p\in\P^{m-1}:\|p\|_{L^\infty(Q_*)}=1\}$. As the zero set of each $p\in\mathbb S$ is null, there is $c_p>0$ with $|\{x\in Q_*:|p(x)|\leq2c_p\}|<\frac14$, and $\{x\in Q_*:|q(x)|\leq c_p\}\subseteq\{x\in Q_*:|p(x)|\leq2c_p\}$ whenever $\|q-p\|_{L^\infty(Q_*)}<c_p$; a finite subcover of $\mathbb S$ by these balls gives $c_0$, the least of their radii.

\begin{lemma}\label{l:sparsepot}
Let $0<s<d$ and let $\mathcal R$ be a collection of cubes such that, for each $i\in\Z$, every point of $\R^d$ belongs to at most $N$ cubes $R\in\mathcal R$ with $2^i\leq\ell(R)<2^{i+1}$. Then, for $g\in L^1_{\mathrm{loc}}(\R^d)$,
\[
\sum_{R\in\mathcal R}\ell(R)^s\left\langle|g|\right\rangle_{3R}\ind_R(x)\lesssim_{N,s}\int_{\R^d}\frac{|g(y)|}{|x-y|^{d-s}}\,\d y.
\]
\end{lemma}
\begin{proof}
The left hand side equals
\[
3^{-d}\int|g(y)|\sum_{R\in\mathcal R}\ell(R)^{s-d}\ind_R(x)\ind_{3R}(y)\,\d y.
\]
If $x\in R$ and $y\in3R$ then $\ell(R)\gtrsim|x-y|$, and at most $N$ cubes with sidelength in each interval $[2^i,2^{i+1})$ contribute, so that, as $s<d$, the sum over $R$ is $\lesssim|x-y|^{s-d}$.
\end{proof}

\begin{lemma}\label{l:diff}
For every cube $Q$ and every $f\in L^1_{\mathrm{loc}}(\R^d)$, $|f|\lesssim \M^\tau_Q f$ a.e.\ on  $Q$.
\end{lemma}
\begin{proof}
Let $\phi\geq0$ be smooth, with $\int\phi=1$ and support in $[-\frac14,\frac14]^d$. At each Lebesgue point $x\in Q^\circ$ of $f$, test \eqref{e:max} with $\phi_{R_k}$, where $R_k\subseteq Q$ is the cube of center $x$ and sidelength $2^{-k}$, $k$ large, and let $k\to\infty$.
\end{proof}

The polynomial projections \eqref{e:proj} have the following properties.
\begin{lemma}\label{l:proj}
Let $m\geq 1$ and $Q$ be a cube.
\begin{itemize}
\item[\emph{(i)}] $\PM^m_Q$ is a linear projection of $L^1_{\mathrm{loc}}(\R^d)$ onto $\P^{m-1}$. In particular,
\begin{equation}\label{e:projdiff}
\PM^m_L f-\PM^m_Qf=\PM^m_L\left(f-\PM^m_Qf\right)
\end{equation}
for all cubes $L,Q$ and all $f\in L^1_{\mathrm{loc}}(\R^d)$.
\item[\emph{(ii)}] For all $f\in L^1_{\mathrm{loc}}(\R^d)$,
\[
\sup_{x\in 3Q}\left|\PM^m_Qf(x)\right|\lesssim\inf_{y\in Q}\M^\tau_Qf(y).
\]
\item[\emph{(iii)}] For all $f\in W^{m,1}_{\mathrm{loc}}(\R^d)$ and $|\alpha|\leq m$, $\partial^\alpha\PM^m_Qf=\PM^{m-|\alpha|}_Q\partial^\alpha f$.
\end{itemize}
\end{lemma}
\begin{proof}
Linearity is obvious, and $\PM^m_Qf\in \P^{m-1}$ by \eqref{e:proj}. If $\pi\in \P^{m-1}$, then $\PM^m_Q\pi=\pi$, by Taylor's formula at each point $y$ and as $\int\theta_Q=1$. In particular $\PM^m_L\PM^m_Q f=\PM^m_Qf$, which is \eqref{e:projdiff}. Turning to (ii), if $x\in 3Q$ then $\dist(x,y)\leq 2\ell(Q)$ for all $y\in \supp\theta_Q$, so that $|\partial^\beta\phi_{Q,x}|\lesssim |Q|^{-1}\ell(Q)^{-|\beta|}$ for all multi-indices $\beta$. Thus $\phi_{Q,x}\in C\Phi^{\tau,1}(Q)$ and $\supp\phi_{Q,x}\subseteq Q$. Testing \eqref{e:max} with the cube $Q$ itself at an arbitrary point $y\in Q$ proves (ii). Finally, let $f\in W^{m,1}_{\mathrm{loc}}(\R^d)$ and $|\alpha|\leq m$. Differentiating the first expression in \eqref{e:proj} in $x$, only the terms of index $\beta\geq\alpha$ survive, and writing $\beta=\alpha+\mu$ gives
\[
\partial^\alpha\PM^m_Qf(x)=\sum_{|\mu|\leq m-1-|\alpha|}\frac{1}{\mu!}\int_{\R^d}\theta_Q(y)\,\partial^\mu\partial^\alpha f(y)\,(x-y)^\mu\,\d y=\PM^{m-|\alpha|}_Q\partial^\alpha f(x),
\]
which is (iii).
\end{proof}

\begin{proof}[Proof of Lemma \ref{l:po}]
There is nothing to prove in (i), and in the first claim of (iv), when $m=0$, so assume $m\geq 1$ there. Assertion (i) follows from $\PM^m_{3Q}(f+\pi)=\PM^m_{3Q}f+\pi$, cf.\ Lemma \ref{l:proj}(i). For (ii), let $\pi\in \P^{m-1}$ and $g\coloneqq f-\pi$. Then $f-\PM^m_{3Q}f=g-\PM^m_{3Q}g$ and, by Lemma \ref{l:proj}(ii) applied to the cube $3Q$,
\[
\M^\tau_{3Q}\left[\PM^m_{3Q}g\right]\leq \left\|\PM^m_{3Q}g\right\|_{L^\infty(3Q)}\lesssim\inf_{3Q}\M^\tau_{3Q}g
\]
pointwise on $3Q$. Hence $\M^\tau_{3Q}[f-\PM^m_{3Q}f]\lesssim\M^\tau_{3Q}g$ on $3Q$, and the first claim of (ii) follows from the monotonicity of the rearrangement. The second claim is obtained by choosing $\pi=\PM^m_{3Q}f\in \P^{m-1}\subseteq\P^{n-1}$ in the first claim with $n$ in place of $m$. For (iii), use $\M^\tau_{3Q}g\leq \M(g\ind_{3Q})$ and the weak-type $(1,1)$ inequality for $\M$ to obtain $(\M^\tau_{3Q}g)^{*,Q}(r)\lesssim r^{-1}\langle g\rangle_{3Q}$, and combine with (ii).
We turn to (iv). Let $p\coloneqq\PM^m_{3Q}f$ and $h\coloneqq f-p$; if $p=0$, the conclusion is immediate. As $h=-p$ on $Q\setminus \supp f$, Lemma \ref{l:diff} yields $\M^\tau_{3Q}h\gtrsim\|p\|_{L^\infty(Q)}$ almost everywhere on the set of those $x\in Q\setminus\supp f$ with $|p(x)|>c_0\|p\|_{L^\infty(Q)}$, whose measure is at least $\frac12|Q|>r|Q|$, by \eqref{e:polylarge}. Therefore $\nat{f}_{m,Q}\gtrsim\|p\|_{L^\infty(Q)}$, and the proof is completed by the equivalence of norms $\|p\|_{L^\infty(3Q)}\lesssim\|p\|_{L^\infty(Q)}$ on $\P^{m-1}$, whose constant is independent of $Q$ by scaling. The second claim of (iv) follows, for all $m\geq0$: testing \eqref{e:max} with the cube $Q$ bounds $|\varphi(f-\PM^m_{3Q}f)|$ by $\inf_Q\M^\tau_{3Q}[f-\PM^m_{3Q}f]\leq\nat{f}_{m,Q}$, while $|\varphi(\PM^m_{3Q}f)|\leq\|\PM^m_{3Q}f\|_{L^\infty(Q)}$.

For (v), as $3S\subseteq3P$, the tests admissible in \eqref{e:max} for $\M^\tau_{3S}$ are admissible for $\M^\tau_{3P}$, and, as $|S|\geq A^{-d}|P|$, the upper $r$-percentile on $S$ of a function is at most its upper $A^{-d}r$-percentile on $P$, by \eqref{e:rearr}. Thus (ii), with $\pi\coloneqq\PM^m_{3P}f$, gives
\[
\natl{f}{r}_{m,S}\lesssim\left(\M^\tau_{3S}\left[f-\PM^m_{3P}f\right]\right)^{*,S}(r)\leq\left(\M^\tau_{3P}\left[f-\PM^m_{3P}f\right]\right)^{*,P}(A^{-d}r)=\natl{f}{A^{-d}r}_{m,P}.
\]

For vector valued $f=(f_1,\dots,f_N)$, the pointwise bounds proving (i) and (ii) hold for each component, and are summed before the rearrangement; for (iii), $\M^\tau_{3Q}g_i\leq\M[|g|\ind_{3Q}]$ on $3Q$ for each component $g_i$ of $g\coloneqq f-\pi$; and (iv) follows from the scalar case, as $\nat{f_i}_{m,Q}\leq\nat{f}_{m,Q}$ for each $i$, while the proof of (v) applies verbatim.
\end{proof}

\section{Local weak type implies sparse domination}\label{s:principle}
The next proposition is an operator-free version of the pointwise sparse domination argument of \cite{DPZyg}*{\S5.3}. It concerns families of functions indexed by the dyadic subcubes of a fixed cube, and relies on a local weak-type estimate for each member $F_Q$ of the family, \eqref{e:P1}--\eqref{e:P2}, and on an estimate for the oscillation of its differences with the members indexed by smaller cubes, tested from the points outside the exceptional set of the weak-type estimate, \eqref{e:P3}.
\begin{proposition}\label{p:principle}
Let $Q_0$ be a cube and $K\geq 1$, and set $\eps_d\coloneqq 2^{-d-4}$. For each $Q\in \mathcal D(Q_0)$, let $F_Q:Q\to \C$ and $E_Q\subseteq Q$ be measurable, and $\lambda_Q\geq 0$. Assume that, for all $Q\in \mathcal D(Q_0)$,
\begin{align}
\label{e:P1} & |E_Q|\leq \eps_d|Q|,
\\ \label{e:P2} & |F_Q(x)|\leq K\lambda_Q, \qquad x\in Q\setminus E_Q,
\\ \label{e:P3} & \left|[F_Q-F_L](x)-[F_Q-F_L](y)\right|\leq K\lambda_Q, \qquad L\in \mathcal D(Q)\setminus\{Q\}, \; y\in L\setminus E_Q,\; \textrm{a.e.\ }x\in L.
\end{align}
In \eqref{e:P3}, the exceptional null set in $x$ may depend on $L$ and $y$.
Then there exists a $\frac34$-sparse collection $\mathcal S\subseteq\mathcal D(Q_0)$ with $Q_0\in \mathcal S$, depending on the sets $E_Q$ alone, such that
\begin{align}
\label{e:principle} & |F_{Q_0}|\leq 3K\sum_{S\in \mathcal S}\lambda_S\ind_S \qquad \textrm{a.e.\ on } Q_0,
\\ \label{e:children} & |\mathsf{sh}(\{P\in\mathcal S:P\subsetneq S\})|\leq\tfrac14|S|, \qquad |S'\cap E_S|\leq\tfrac14|S'|, \qquad S\in\mathcal S,
\end{align}
where $S'$ ranges over the maximal elements of $\{P\in\mathcal S: P\subsetneq S\}$.
\end{proposition}

The stopping cubes are those where the exceptional set of the parent weak-type estimate becomes dense. The difference $F_Q-F_L$ is then estimated from a single good point of $L$, outside the exceptional sets of both $Q$ and $L$.
\begin{proof}
\emph{One generation.} For $Q\in \mathcal D(Q_0)$, let
\[\mathcal L(Q)\coloneqq \textrm{maximal elements } L \textrm{ of } \mathcal D(Q)\setminus\{Q\} \textrm{ such that } |L\cap E_Q|>2^{-d-2}|L|,\]
which depends on $E_Q$ alone. We first show that
\begin{align}
\label{e:iter1} & |\mathsf{sh}(\mathcal L(Q))|\leq \tfrac14|Q|,
\\ \label{e:iter2} & |F_Q(x)-F_L(x)|\leq K(2\lambda_Q+\lambda_L), \qquad \textrm{a.e.\ }x\in L\in\mathcal L(Q),
\\ \label{e:iter3} & |E_Q\setminus\mathsf{sh}(\mathcal L(Q))|=0,
\\ \label{e:iter4} & |L\cap E_Q|\leq\tfrac14|L|, \qquad L\in\mathcal L(Q).
\end{align}
The elements of $\mathcal L(Q)$ are pairwise disjoint, whence, by \eqref{e:P1},
\[
\sum_{L\in\mathcal L(Q)}|L|\leq 2^{d+2}\sum_{L\in\mathcal L(Q)}|L\cap E_Q|\leq 2^{d+2}|E_Q|\leq \tfrac14|Q|,
\]
which is \eqref{e:iter1}. If $L\in\mathcal L(Q)$, either $\widehat L=Q$, and then $|\widehat L\cap E_Q|\leq \eps_d|Q|\leq 2^{-d-2}|\widehat L|$, or $|\widehat L\cap E_Q|\leq 2^{-d-2}|\widehat L|$ by maximality. In both cases $|L\cap E_Q|\leq \frac14|L|$, which is \eqref{e:iter4}, and as $|E_L|\leq \frac14|L|$ by \eqref{e:P1}, we may pick a point
\[
y_L\in L\setminus(E_Q\cup E_L).
\]
Then, for almost every $x\in L$, using \eqref{e:P3} with $y=y_L$, and \eqref{e:P2} at the point $y_L$ for both $Q$ and $L$,
\[
|F_Q(x)-F_L(x)|\leq \left|[F_Q-F_L](x)-[F_Q-F_L](y_L)\right|+|F_Q(y_L)|+|F_L(y_L)|\leq 2K\lambda_Q+K\lambda_L,
\]
which is \eqref{e:iter2}. Finally, almost every $x\in E_Q$ is a point of density of $E_Q$, so that $|L\cap E_Q|>2^{-d-2}|L|$ for some $L\in \mathcal D(Q)\setminus\{Q\}$ containing $x$, and $x$ belongs to some element of $\mathcal L(Q)$. Therefore, almost every $x\in Q\setminus\mathsf{sh}(\mathcal L(Q))$ lies outside $E_Q$, which is \eqref{e:iter3}.

\emph{Iteration.} Set $\mathcal S^{(0)}\coloneqq\{Q_0\}$ and, inductively,
\[
\mathcal S^{(n+1)}\coloneqq\bigcup_{Q\in \mathcal S^{(n)}}\mathcal L(Q), \qquad n\geq 0,
\]
and let $\mathcal S\coloneqq\bigcup_{n\geq 0}\mathcal S^{(n)}$, which depends on the sets $E_Q$ alone, as the families $\mathcal L(Q)$ do. For $S\in\mathcal S$, the maximal elements of $\mathcal S$ strictly contained in $S$ are those of $\mathcal L(S)$, so that \eqref{e:children} is \eqref{e:iter1} together with \eqref{e:iter4}. The elements of each $\mathcal S^{(n)}$ are pairwise disjoint and, by \eqref{e:iter1}, the pairwise disjoint sets
\[
G_S\coloneqq S\setminus \mathsf{sh}(\mathcal L(S)), \qquad S\in \mathcal S,
\]
satisfy $|G_S|\geq\frac34|S|$; thus $\mathcal S$ is $\frac34$-sparse. Also,
\[
\sum_{S\in \mathcal S^{(n)}}|S|\leq 4^{-n}|Q_0|,\qquad n\geq0,
\]
so that almost every $x\in Q_0$ belongs to finitely many elements of $\mathcal S$. These form a chain $Q_0=I_0\supsetneq I_1\supsetneq\cdots\supsetneq I_N$ with $I_n\in \mathcal S^{(n)}$, and $x\in G_{I_N}$. Telescoping, and applying \eqref{e:iter2}, then \eqref{e:iter3} and \eqref{e:P2}, we obtain for almost every such $x$
\[
|F_{Q_0}(x)|\leq \sum_{n=0}^{N-1}\left|F_{I_n}(x)-F_{I_{n+1}}(x)\right|+|F_{I_N}(x)|\leq K\sum_{n=0}^{N-1}\left(2\lambda_{I_n}+\lambda_{I_{n+1}}\right)+K\lambda_{I_N}\leq 3K\sum_{n=0}^N\lambda_{I_n}.
\]
The rightmost side equals the right hand side of \eqref{e:principle} at $x$, and the proof is complete.
\end{proof}

\begin{remark}\label{r:principle}
Enlarging the sets $E_Q$ preserves \eqref{e:P2} and \eqref{e:P3}: thus $E_Q$ may contain any further exceptional set, as long as \eqref{e:P1} is preserved, and, as the stopping cubes depend on the sets $E_Q$ alone, one collection serves several families, by taking the union of their exceptional sets. The proof only uses the triangle inequality, and not the measurability of $F_Q$, so that Proposition \ref{p:principle} holds verbatim for families with values in a vector space endowed with an extended norm, taking values in $[0,\infty]$, as in the proof of Proposition \ref{p:grand}. The bound \eqref{e:principle} may also be derived from the operator-free principle of Lerner, Lorist and Ombrosi \cite{LLO22}*{Theorem 3.2}, the route of \cite{CLR26}.
\end{remark}

\section{Percentile Poincar\'e inequalities}\label{s:poincare}

This section bounds the percentile oscillations of a function by those of its derivatives, Proposition \ref{p:poincare}, through the anti-integration by parts of \cite{DGWdom}.

\subsection{Dual projectors}\label{ss:dual}

For an integer $m\geq1$ and a cube $S$, let
\begin{equation}\label{e:dual}
\Pi^m_S\varphi\coloneqq\int_{\R^d}\varphi(x)\phi_{S,x}\,\d x,
\end{equation}
\noeqref{e:dual}
with $\phi_{S,x}$ as in \eqref{e:proj}, so that $\varphi(\PM^m_Sf)=\left(\Pi^m_S\varphi\right)(f)$ for all $f\in L^1_{\mathrm{loc}}(\R^d)$ and all bounded, compactly supported $\varphi$.
\begin{lemma}\label{l:dual}
Let $m\geq1$, let $R\subseteq S$ be cubes, and let $\varphi\in\Phi^{\tau,1}(R)$ be supported in $R$. Then
\begin{align}
\label{e:dual1} & \varphi\left(f-\PM^m_Sf\right)=\left(\varphi-\Pi^m_S\varphi\right)(f),\qquad f\in L^1_{\mathrm{loc}}(\R^d),
\\ \label{e:dual2} & \Pi^m_S\varphi\in C\Phi^{\tau,1}(S), \qquad \supp\Pi^m_S\varphi\subseteq S,
\end{align}
and both $\varphi-\Pi^m_S\varphi$ and $\Pi^m_S\varphi-\Pi^m_{S'}\varphi$, $S'$ a further cube containing $R$, annihilate $\P^{m-1}$.
\end{lemma}
\begin{proof}
Identity \eqref{e:dual1} is Fubini's theorem, as $\PM^m_Sf(x)=\phi_{S,x}(f)$ by \eqref{e:proj}. If $\pi\in\P^{m-1}$, then $(\Pi^m_S\varphi)(\pi)=\varphi(\PM^m_S\pi)=\varphi(\pi)$ by Lemma \ref{l:proj}(i), which does not depend on $S$; this gives both annihilation claims. Finally, by \eqref{e:proj},
\[
\Pi^m_S\varphi(y)=\sum_{|\alpha|\leq m-1}\frac{(-1)^{|\alpha|}}{\alpha!}\,\partial^\alpha_y\left[\theta_S(y)\int_{\R^d}\varphi(x)(x-y)^\alpha\,\d x\right],
\]
which is supported in $\supp\theta_S\subseteq S$. As $\|\varphi\|_1\leq1$ and $\varphi$ is supported in $R\subseteq S$, the inner integral and each of its derivatives in $y$ of order $\beta$ are bounded by $C\ell(S)^{|\alpha|-|\beta|}$ for $y\in S$, while $|\partial^\gamma\theta_S|\lesssim|S|^{-1}\ell(S)^{-|\gamma|}$. The Leibniz rule gives $|\partial^\beta\Pi^m_S\varphi|\lesssim|S|^{-1}\ell(S)^{-|\beta|}$ for every $\beta$, which is \eqref{e:dual2}.
\end{proof}

\subsection{Anti-integration}\label{ss:anti}
The next lemma is the smoothness-preserving, compactly supported case of the anti-integration by parts of \cite{DGWdom}. There, the primitives are only required to be bounded, and the moments are subtracted through $\PM^m_S$; here they are already absent, and the primitives inherit the regularity of $\psi$, as required by the test class of \eqref{e:max}.
\begin{lemma}\label{l:anti}
Let $m\geq1$, $\sigma>0$, and let $\psi\in\Phi^{\sigma,1}(S)$ be supported in the cube $S$ and annihilate $\P^{m-1}$. There exist $\psi_\alpha\in C\Phi^{\sigma,1}(S)$, $|\alpha|=m$, supported in $S$, such that
\begin{equation}\label{e:anti}
\psi=\ell(S)^m\sum_{|\alpha|=m}\partial^\alpha\psi_\alpha
\end{equation}
in the sense of distributions. If $\psi$ annihilates $\P^m$, each $\psi_\alpha$ has vanishing mean. Consequently, for all $f\in W^{m,1}_{\mathrm{loc}}(\R^d)$,
\begin{equation}\label{e:antipair}
\psi(f)=(-1)^m\ell(S)^m\sum_{|\alpha|=m}\psi_\alpha\left(\partial^\alpha f\right),\qquad\textrm{so that}\qquad|\psi(f)|\leq\ell(S)^m\sum_{|\alpha|=m}\left|\psi_\alpha\left(\partial^\alpha f\right)\right|.
\end{equation}
\end{lemma}
\begin{proof}
By rescaling, let $S=[0,1)^d$. If $d=1$, the $m$-th iterated primitive
\[
\psi_m(y)\coloneqq\int_{-\infty}^y\frac{(y-t)^{m-1}}{(m-1)!}\,\psi(t)\,\d t
\]
satisfies $\psi_m^{(m)}=\psi$ and, by the vanishing moments of $\psi$, is supported in $S$; its regularity is that of $\psi$. If $d\geq2$, write $y=(y_1,y')$, fix $\xi\in\mathcal{C}^\infty_c((0,1))$ with $\int\xi=1$ and $\int t^n\xi(t)\,\d t=0$ for $1\leq n<m$, and set $\xi_j\coloneqq\frac{(-1)^j}{j!}\xi^{(j)}$, so that $\int t^i\xi_j=\delta_{ij}$ for $0\leq i,j<m$. With $\mu_j(y')\coloneqq\int y_1^j\psi(y_1,y')\,\d y_1$, the function
\[
\psi-\sum_{j<m}\xi_j\otimes\mu_j
\]
has vanishing $y_1$-moments of order less than $m$ for each $y'$, hence is $\partial_1^m$ of a function supported in $S$, by the one dimensional case applied with parameter $y'$. Each $\mu_j$ annihilates the polynomials of degree less than $m-j$ in $y'$, so that induction on the dimension gives
\[
\mu_j=\sum_{|\alpha'|=m-j}\partial^{\alpha'}\mu_{j,\alpha'},\qquad \xi_j\otimes\mu_j=\frac{(-1)^j}{j!}\sum_{|\alpha'|=m-j}\partial_1^j\partial^{\alpha'}\left(\xi\otimes\mu_{j,\alpha'}\right).
\]
All the functions constructed are supported in $S$ and inherit the bounds of $\psi$, up to admissible constants. The identity in \eqref{e:antipair} follows from \eqref{e:anti} by integration by parts. Finally, if $\psi$ annihilates $\P^m$, testing \eqref{e:anti} against $x^\beta$, $|\beta|=m$, gives $0=\int x^\beta\psi=(-1)^m\ell(S)^m\beta!\int\psi_\beta$.
\end{proof}

\begin{proposition}\label{p:poincare}
Let $1\leq n\leq m$ be integers, let $Q$ be a cube, $f\in W^{m,1}_{\mathrm{loc}}(\R^d)$ and $h\coloneqq f-\PM^m_{3Q}f$. Then
\begin{equation}\label{e:poincare}
\M^\tau_{3Q}h(x)\lesssim\ell(Q)^n\,\M^\tau_{3Q}\left[\nabla^nh\right](x),\qquad x\in 3Q,
\end{equation}
and consequently, for $0<r<1$,
\begin{equation}\label{e:poincareperc}
\nat{f}_{m,Q}\lesssim\ell(Q)^n\nat{\nabla^nf}_{m-n,Q},\qquad \nat{f}_{m,Q}\lesssim\ell(Q)^m\left(\M^\tau_{3Q}\left[\nabla^mf\right]\right)^{*,Q}(r)\lesssim\frac{\ell(Q)^m}{r}\left\langle\left|\nabla^mf\right|\right\rangle_{3Q}.
\end{equation}
\end{proposition}
\begin{proof}
Let $\varphi\in\Phi^{\tau,1}(R)$ be supported in $R$, with $x\in R\subseteq3Q$. Using \eqref{e:enclosing} repeatedly, with $P=3Q$ and $E$ the previously constructed cube, choose
\[
R=S_0\subseteq S_1\subseteq\cdots\subseteq S_N=3Q,\qquad \ell(S_i)=\min\left\{2^i\ell(R),3\ell(Q)\right\},
\]
which is legitimate as $\diam S_i=\ell(S_i)<\ell(S_{i+1})$ unless $\ell(S_{i+1})=\ell(3Q)$. By \eqref{e:dual1} and telescoping,
\[
\varphi\left(f-\PM^m_{3Q}f\right)=\left(\varphi-\Pi^m_{S_0}\varphi\right)(f)+\sum_{i=0}^{N-1}\left(\Pi^m_{S_i}\varphi-\Pi^m_{S_{i+1}}\varphi\right)(f).
\]
By Lemma \ref{l:dual}, each function paired with $f$ on the right annihilates $\P^{m-1}$, which contains $\PM^m_{3Q}f$ and $\P^{n-1}$, so that $f$ may be replaced by $h$ in each pairing; and, as $\ell(S_{i+1})\leq2\ell(S_i)$, by \eqref{e:change}, it belongs to $C\Phi^{\tau,1}(S_{i+1})$ and is supported in $S_{i+1}$, respectively to $C\Phi^{\tau,1}(S_0)$ and is supported in $S_0$. Lemma \ref{l:anti} of order $n$, applied on the cube $S_{i+1}$, and \eqref{e:max}, tested at $x\in S_{i+1}\subseteq3Q$, bound the $i$-th term by a constant multiple of
\[
\ell(S_{i+1})^n\,\M^\tau_{3Q}\left[\nabla^nh\right](x),
\]
and likewise for the first term, with $S_0$ in place of $S_{i+1}$. As $n\geq1$,
\[
\sum_{i\leq N}\ell(S_i)^n\lesssim\ell(Q)^n,
\]
and \eqref{e:poincare} follows. By Lemma \ref{l:proj}(iii), $\nabla^nh=\nabla^nf-\PM^{m-n}_{3Q}\nabla^nf$, so that \eqref{e:poincare} and \eqref{e:rearr} give the first inequality in \eqref{e:poincareperc}. For $n=m$, $\nabla^mh=\nabla^mf$, which gives the second, and the third is Lemma \ref{l:po}(iii), with $\nabla^mf$ and $0$ in place of $f$ and $m$.
\end{proof}

The next lemma compares, cube by cube, the percentile oscillations of the derivatives of $f$ of order at most $\kappa$ with those of $f$ and of $\nabla^\kappa f$, with the convention of \S\ref{ss:perc} for vector valued functions. It is used in the proofs of Theorem \ref{t:jet} and Corollary \ref{c:sobjet}.
\begin{lemma}\label{l:sandwichperc}
Let $1\leq\kappa\leq m$ be integers and $f\in\mathcal{C}^\infty(\R^d)$. For every cube $S$ the following hold.
\begin{enumerate}[label=(\roman*)]
\item If $0<j<\kappa$, then
\begin{equation}\label{e:sandwichsmall}
\nat{\nabla^jf}_{m-j,S}\lesssim\ell(S)^{\kappa-j}\nat{\nabla^\kappa f}_{m-\kappa,S}.
\end{equation}
\item If $\ell(S)\geq1$ and $0\leq j\leq\kappa$, then
\begin{equation}\label{e:sandwichlarge}
\natl{\nabla^jf}{r}_{m-j,S}\lesssim\natl{f}{2^{-d-1}r}_{m,2S}+\natl{\nabla^\kappa f}{2^{-d-1}r}_{m-\kappa,2S}.
\end{equation}
\end{enumerate}
\end{lemma}
\begin{proof}
Let $h\coloneqq f-\PM^m_{3S}f$. By Lemma \ref{l:proj}(iii), $\nabla^jh=\nabla^jf-\PM^{m-j}_{3S}\nabla^jf$, whence $\nat{\nabla^jf}_{m-j,S}=(\M^\tau_{3S}[\nabla^jh])^{*,S}(r)$ for $0\leq j\leq m$.

We prove (i). For $|\gamma|=j$, Proposition \ref{p:poincare}, with $\partial^\gamma f$, $m-j$, $\kappa-j$ and $S$ in place of $f$, $m$, $n$ and $Q$, gives $\M^\tau_{3S}[\partial^\gamma h]\lesssim\ell(S)^{\kappa-j}\M^\tau_{3S}[\nabla^\kappa h]$ on $3S$. Sum over $|\gamma|=j$ and take the percentiles on $S$.

We turn to (ii), and first prove the interpolation estimate \eqref{e:sandwichmax} below. Let $0<j<\kappa$, $|\gamma|=j$ and $n\coloneqq\kappa-j$. Let $B$ be the cube of sidelength $\frac18$ centered at the origin, $\xi\in\mathcal{C}^\infty_c(B)$ with $\int\xi=1$ and vanishing moments of orders $1$ to $n-1$, and $\xi_s\coloneqq s^{-d}\xi(s^{-1}\cdot)$. For $\varphi\in\Phi^{\tau,1}(R)$ supported in $R$, $x\in R\subseteq3S$,
\[
\varphi=\varphi*\xi+\sum_{i\geq0}\varphi*\left(\xi_{2^{-i-1}}-\xi_{2^{-i}}\right)
\]
uniformly, all terms being supported in $R+B\subseteq4S$, as $\ell(S)\geq1$. First, $(\varphi*\xi)(\partial^\gamma h)=(-1)^j(\varphi*\partial^\gamma\xi)(h)$, with $\varphi*\partial^\gamma\xi\in C\Phi^{\tau,1}(R+B)$, as a convolution inherits the bounds of its factor at the larger scale, the other factor having $L^1$ norm $\lesssim1$. Second, $\xi_{s/2}-\xi_s$ is supported in $sB$ and annihilates $\P^{n-1}$, so that Lemma \ref{l:anti} writes it as
\[
\xi_{s/2}-\xi_s=s^n\sum_{|\mu|=n}\partial^\mu\zeta_\mu,
\]
with $\zeta_\mu\in C\Phi^{\tau,1}(sB)$ supported in $sB$; as, likewise, $\varphi*\zeta_\mu\in C\Phi^{\tau,1}(R+sB)$, the term of index $i$ paired with $\partial^\gamma h$ is $\lesssim2^{-in}\M^\tau_{4S}[\nabla^\kappa h](x)$. Summing over $i$ and $|\gamma|=j$,
\begin{equation}\label{e:sandwichmax}
\M^\tau_{3S}\left[\nabla^jh\right]\lesssim\M^\tau_{4S}h+\M^\tau_{4S}\left[\nabla^\kappa h\right]\qquad\textrm{on }3S,
\end{equation}
which is trivial for $j\in\{0,\kappa\}$. We next compare the polynomial projections. Let $g\coloneqq f-\PM^m_{6S}f$, so that $h=g-\PM^m_{3S}g$, by \eqref{e:projdiff}. By Lemma \ref{l:proj}(ii) on $3S$, and as $|S|>r|2S|$,
\[
\sup_{9S}\left|\PM^m_{3S}g\right|\lesssim\inf_S\M^\tau_{6S}g\leq\natl{f}{r}_{m,2S},
\]
and likewise for $\nabla^\kappa f$ and $m-\kappa$ if $m>\kappa$. With \eqref{e:sandwichmax}, as $\M^\tau_{4S}\pi\leq\sup_{4S}|\pi|$ for polynomials $\pi$ and $\M^\tau_{4S}\leq\M^\tau_{6S}$ on $4S$, on $S$,
\[
\M^\tau_{3S}\left[\nabla^jh\right]\lesssim\M^\tau_{6S}g+\M^\tau_{6S}\left[\nabla^\kappa f-\PM^{m-\kappa}_{6S}\nabla^\kappa f\right]+\natl{f}{r}_{m,2S}+\natl{\nabla^\kappa f}{r}_{m-\kappa,2S}.
\]
Finally, we pass to the percentiles and to the enlarged cube. As $(F+G)^{*,S}(r)\leq F^{*,S}(\frac r2)+G^{*,S}(\frac r2)$, $F^{*,S}(t)\leq F^{*,2S}(2^{-d}t)$, and the percentiles do not increase with the level, \eqref{e:sandwichlarge} follows.
\end{proof}

\section{Model wavelet operators}\label{s:models}
This section introduces the model wavelet operators, the paraproducts and the wavelet representation of Calder\'on--Zygmund operators, Proposition \ref{p:rep}, and contains the kernel estimates of the models, the cutoff testing condition for mean zero input packets, and the bounded action of the models on localized polynomials under that condition.

\subsection{Model operators}\label{ss:modeldef}
The model operators below are the counterparts of the wavelet forms of \cite{DGWdom}*{Definition 2.7}, in which the requirement that the packets belong to cancellative classes is replaced by an $L^2$ bound, uniform over truncations.
Fix smoothness parameters $\sigma,\sigma'>0$, decay parameters $\rho,\rho'\geq 0$, and two families of packets
\begin{equation}\label{e:packets}
\left\{\varphi_R\in \Phi^{\sigma,\rho}(R): R\in \mathcal D\right\}, \qquad \left\{\psi_R\in \Phi^{\sigma',\rho'}(R): R\in \mathcal D\right\}.
\end{equation}
In general, these are \emph{not} assumed to be rescalings of fixed functions. To each finite collection $\mathcal R\subset \mathcal D$, associate the operator
\begin{equation}\label{e:model}
V_{\mathcal R}f\coloneqq \sum_{R\in \mathcal R}|R|\varphi_R(f)\psi_R.
\end{equation}
\begin{definition}\label{d:model}
The collection $V=\{V_{\mathcal R}: \mathcal R\subset\mathcal D \textrm{ finite}\}$ is a \emph{$(\sigma,\sigma';\rho,\rho')$-model wavelet operator} if
\begin{equation}\label{e:L2}
\sup_{\substack{\mathcal R\subset \mathcal D\\ \mathcal R \textrm{ finite}}}\left\|V_{\mathcal R}\right\|_{L^2(\R^d)\to L^2(\R^d)}\leq 1.
\end{equation}
Its operator is
\begin{equation}\label{e:modelv}
Vf\coloneqq \sum_{R\in \mathcal D}|R|\varphi_R(f)\psi_R,\qquad f\in L^2(\R^d),
\end{equation}
the series converging unconditionally in $L^2(\R^d)$, see below. As usual, fixed constants in the packet bounds and in \eqref{e:L2} may be absorbed into an overall normalization of the operator.
\end{definition}
\begin{definition}\label{d:ctc}
Say that the $(\sigma,\sigma';\rho,\rho')$-model wavelet operator $V $ satisfies the \emph{cutoff testing condition} if
\begin{equation}\label{e:V1}
\sup_{Q \textrm{ cube}} \left\|V \left(\omega^\infty_Q\right)\right\|_{L^\infty(\R^d)}\leq1,
\end{equation}
where $\omega^\infty_Q$ are the cutoffs of \S\ref{ss:wavelets}, equal to one on $\frac32Q$ and supported in $2Q$. This is the counterpart for models of the condition $T1\in L^\infty(\R^d)$, and holds automatically, for every $\mathcal R$, when the input packets have mean zero, see Lemma \ref{l:V1mean}. Unlike \eqref{e:L2}, it is not inherited by subcollections, as the wavelet expansion of a bounded function, restricted to a subcollection of $\mathcal D$, need not be bounded, cf.\ the canonical paraproducts of Remark \ref{r:models}(b).
\end{definition}
By \eqref{e:L2}, for each $f\in L^2(\R^d)$, the finite subsums of the series \eqref{e:modelv} are uniformly bounded in $L^2(\R^d)$, whence the series converges unconditionally in $L^2(\R^d)$, by the Bessaga--Pe{\l}czy\'nski theorem \cite{BP58}, as $L^2(\R^d)$ contains no isomorphic copy of $c_0$. So do its subseries, which define $V_{\mathcal R}$ for arbitrary $\mathcal R\subseteq\mathcal D$, with $\|V_{\mathcal R}\|_{2\to2}\leq 1$. Each $V_{\mathcal R}$ is itself the operator $V$ of a model with the same parameters, namely the one obtained by replacing $\varphi_R$ with $0$ for $R\notin\mathcal R$. The kernel estimates of \S\ref{ss:kernel} are uniform over finite collections, and are inherited by $V_{\mathcal R}$ for arbitrary $\mathcal R\subseteq\mathcal D$.

\begin{remark}\label{r:models}
No cancellation is required of either family of packets, and the definition is symmetric, in the sense that the adjoint of a $(\sigma,\sigma';\rho,\rho')$-model is a $(\sigma',\sigma;\rho',\rho)$-model, and \eqref{e:L2} holds for both. Thus, up to adjoints, the models used in this article are of two types.
\begin{enumerate}[label=(\alph*)]
\item If all packets have mean zero, \eqref{e:L2} holds up to a constant depending on the smoothness and decay parameters and on $d$ only, by almost orthogonality \cite{DGWdom}*{Proposition 3.4 and Remark 3.5}. Examples are wavelet multipliers \cite{DWW23} and the operators $V_{\alpha,\beta}$ of Proposition \ref{p:rep}.
\item The paraproducts
\[
f\mapsto\sum_{R\in\mathcal R}|R|b_R\vartheta_R(f)\psi_R,\qquad \vartheta_R\in \Phi^{\sigma,\Subset}(R),
\]
where $\psi_R$ are the wavelets of \S\ref{ss:wavelets}, have noncancellative input packets. They satisfy \eqref{e:L2} when $\{b_R\}$ is a Carleson sequence, up to a constant depending on its Carleson norm, by the Carleson embedding theorem. The canonical paraproducts \eqref{e:canpara} are of this form, with $b_R=\psi_R(b)$ and $\mathcal R=\mathcal D$. They satisfy \eqref{e:L2} when $b\in\mathrm{BMO}(\R^d)$ and, under the stronger assumption $b\in L^\infty(\R^d)$, also \eqref{e:V1}, see Lemma \ref{l:para}, while their adjoints, having mean zero input packets, satisfy \eqref{e:V1} for every $b\in\mathrm{BMO}(\R^d)$, by Lemma \ref{l:V1mean}.
\end{enumerate}
Both types appear in the wavelet representation of Calder\'on--Zygmund operators, from the classical \cite{BCR91} and \cite{MC97}*{Chapter 8} to the representation theorems of \cites{DWW23,DGW23,HL22} and Proposition \ref{p:rep}. Their Haar counterparts, the dyadic shifts and paraproducts, are the building blocks of the dyadic representation theorem \cite{hytonen12}.
\end{remark}

\subsection{Paraproducts and representation}\label{ss:pararep}

Fix $\vartheta\in \mathcal{C}^\infty_c(\R^d)$  satisfying
\[\supp \vartheta \Subset \left({\textstyle-\frac12,\frac12}\right)^d, \qquad \int\vartheta=1, \qquad \int x^\mu\vartheta(x)\,\d x=0,\quad 1\leq|\mu|\leq k,\] and let $\vartheta_Q$ be its $L^1$-normalized rescaling \eqref{e:norm1}, so that $\vartheta_Q\in C\Phi^{\sigma,\Subset}(Q)$ for every $\sigma>0$ and
\begin{equation}\label{e:thetamom}
\vartheta_Q\left((\cdot-c(Q))^\mu\right)=\delta_{\mu0},\qquad |\mu|\leq k.
\end{equation}
For a symbol $b=\{b_S:S\in\mathcal D\}$ with $\sup_S\ell(S)^{n}|b_S|<\infty$ for some $0\leq n\leq k$, the \emph{paraproduct form} with symbol $b$ is
\begin{equation}\label{e:paraform}
\Pi_b(h,g)\coloneqq\sum_{S\in\mathcal D}|S|b_S\vartheta_S(h)\psi_S(g),\qquad h,g\in \mathcal{C}^\infty_c(\R^d),
\end{equation}
the series converging absolutely, as $|\psi_S(g)|\lesssim\ell(S)^{k+2}$ at small scales, by the vanishing moments of $\psi_S$, and $|\vartheta_S(h)\psi_S(g)|\lesssim|S|^{-2}$ at large ones. It is the paraproduct form \cite{DGWdom}*{(4.2)} of the coefficients $b_S$, with the averaging functions $\vartheta_S$. For $b\in\mathrm{BMO}(\R^d)$, the \emph{canonical paraproduct} with symbol $b$ and its adjoint are
\begin{equation}\label{e:canpara}
P_bh\coloneqq\sum_{Q\in\mathcal D}|Q|\psi_Q(b)\vartheta_Q(h)\psi_Q,\qquad P^*_bh\coloneqq\sum_{Q\in\mathcal D}|Q|\psi_Q(b)\psi_Q(h)\vartheta_Q,
\end{equation}
so that $\Pi_{\{\psi_S(b)\}}(h,g)=\langle P_bh,g\rangle$. The coefficients $\psi_Q(b)$ are well defined, as $\psi_Q$ has vanishing mean and compact support, and $|\psi_Q(b)|\lesssim\|b\|_{\mathrm{BMO}}$. The properties of $P_b$, $P_b^*$ and $\Pi_b$ needed below are collected in Lemma \ref{l:para}.

The next proposition recalls the global case $\Omega=\R^d$ of \cite{DGWdom}*{Theorem A}, at every order $\kappa\leq k$, in a normalized form in which the paraproducts are the forms \eqref{e:paraform}, the input packets of the wavelet part have vanishing mean, and the regularity of all packets is recorded. The paraproduct forms of the symbols are well defined by Lemma \ref{l:size}(i).
\begin{proposition}\label{p:rep}
Referring to Definition \ref{d:czo}, let $T$ be a Calder\'on--Zygmund operator of order $k$, $0\leq\kappa\leq k$ and $|\alpha|=\kappa$. For every $0<\delta_1<\delta$, there are packets $\varphi^{\alpha,\beta}_R\in C\Psi^{\delta_1,0}(R)$, $|\beta|=\kappa$, $R\in\mathcal D$, with $C$ depending on $\delta_1$, such that, for all $f,g\in \mathcal{C}^\infty_c(\R^d)$,
\begin{equation}\label{e:rep}
\left\langle\partial^\alpha Tf,g\right\rangle=\sum_{|\beta|=\kappa}\left\langle V_{\alpha,\beta}\left(\partial^\beta f\right),g\right\rangle+\sum_{|\gamma|\leq\kappa}\frac1{\gamma!}\Pi_{b^{\alpha,\gamma}}\left(\partial^\gamma f,g\right)+\ind_{\{\kappa=0\}}\Pi_{b^{0,\star}}(g,f),
\end{equation}
where $V_{\alpha,\beta}$ is the cancellative model sum \eqref{e:modelv} with input packets $\varphi^{\alpha,\beta}_R$ and output packets $\psi_R$. The latter, up to a constant depending on $\delta_1$, is a $(\delta_1,k+1;0,1)$-model wavelet operator, by Remark \emph{\ref{r:models}(a)}.
\end{proposition} The proof of the proposition is given in \S\ref{ss:rep}.
The symbols enter \eqref{e:rep} as sequences, through the forms \eqref{e:paraform}. When they belong to $\mathrm{BMO}(\R^d)$, the paraproduct forms in \eqref{e:rep} are given by the canonical paraproducts $P_{\mathfrak b^{\alpha,\gamma}}$ and $P^*_{\mathfrak b^{0,\star}}$, and, up to constant multiples, all operators on the right hand side of \eqref{e:rep} are model wavelet operators, see Definition \ref{d:model} and Remark \ref{r:models}. The input packets of $V_{\alpha,\beta}$ and $P^*_{\mathfrak b^{0,\star}}$ have vanishing mean, while those of the paraproducts $P_{\mathfrak b^{\alpha,\gamma}}$ do not.

\subsection{Kernel estimates}\label{ss:kernel}

For finite $\mathcal R$, the integral kernel of $V_{\mathcal R}$ is
\begin{equation}\label{e:kernel}
K_{\mathcal R}(x,z)\coloneqq \sum_{R\in\mathcal R}|R|\psi_R(x)\varphi_R(z),
\end{equation}
so that $V_{\mathcal R}f(x)=\int K_{\mathcal R}(x,z)f(z)\,\d z$ for all $x\in \R^d$. Fix a non-integer smoothness index $\tau$ with
\begin{equation}\label{e:range}
0<\tau\leq\sigma,\qquad \tau<b_0\coloneqq\min\{\sigma+\rho,\sigma'+\rho'\},
\end{equation}
and write $\nu\coloneqq\frac12\min\{1,\sigma',b_0-\tau\}$.
\begin{lemma}\label{l:kernel}
Let $x,z,h,u\in\R^d$ with $x\neq z$ and $D\coloneqq|x-z|\geq 4(|h|+|u|)$, and let $\alpha,\beta$ be multi-indices with $|\alpha|\leq\lfloor\tau\rfloor$ and $|\beta|=\lfloor\tau\rfloor$. Uniformly over finite collections $\mathcal R\subset \mathcal D$, the kernel $K=K_{\mathcal R}$ satisfies
\begin{align}
\label{e:ksize} \left|\partial_z^\alpha K(x,z)\right| &\lesssim D^{-(d+|\alpha|)},
\\ \label{e:kholder} \left|\partial_z^\beta K(x,z+u)-\partial_z^\beta K(x,z)\right| &\lesssim |u|^{\{\tau\}} D^{-(d+\tau)},
\\ \label{e:kx} \left|\partial_z^\alpha K(x+h,z)-\partial_z^\alpha K(x,z)\right| &\lesssim |h|^\nu D^{-(d+|\alpha|+\nu)},
\\ \label{e:kxz} \left|\partial_z^\beta\left[K(x+h,z+u)-K(x+h,z)-K(x,z+u)+K(x,z)\right]\right| &\lesssim |h|^\nu|u|^{\{\tau\}} D^{-(d+\tau+\nu)}.
\end{align}
\end{lemma}
\begin{proof}
Write $A\coloneqq\sigma+\rho$ and $B\coloneqq\sigma'+\rho'$. At a fixed sidelength $q\in 2^{\Z}$, the rescaled centers $c(R)/q$ form a translate of $\Z^d$, and
\begin{equation}\label{e:lattice}
\sum_{\ell(R)=q}\left(1+\frac{|x-c(R)|}{q}\right)^{-d-B}\left(1+\frac{|z-c(R)|}{q}\right)^{-d-A}\lesssim\left(1+\frac{D}{q}\right)^{-d-b_0}.
\end{equation}
In each summand at least one of $|x-c(R)|$, $|z-c(R)|$ is at least $\frac D2$; bound that factor by its value at $D/2$, and sum the other factor over the lattice, which is possible as $A,B>0$. By definition of the classes in \eqref{e:packets} and by \eqref{e:lattice}, the contribution of the cubes of sidelength $q$ to the left hand side of \eqref{e:ksize} is at most a constant multiple of $q^{-d-|\alpha|}(1+D/q)^{-d-b_0}$. The contributions to \eqref{e:kholder}, \eqref{e:kx} and \eqref{e:kxz} carry the additional factors $(|u|/q)^{\{\tau\}}$, $(|h|/q)^\nu$ and their product, respectively. When an increment is at most $q$, this follows from the regularity of the relevant packet, which is available as $\tau\leq\sigma$ and $\nu<\min\{1,\sigma'\}$; when it exceeds $q$, use the size of the packet at the two endpoints, the additional factor being at least one. The separation assumption keeps all relevant distances comparable to $D$. It remains to sum over $q$; for instance, for \eqref{e:kxz},
\[
|h|^\nu|u|^{\{\tau\}}\sum_{q\in 2^{\Z}}q^{-d-\tau-\nu}\left(1+\frac Dq\right)^{-d-b_0}\lesssim |h|^\nu|u|^{\{\tau\}} D^{-d-\tau-\nu},
\]
as the sum over $q\leq D$ is geometric with exponent $b_0-\tau-\nu>0$, and the sum over $q>D$ is geometric with exponent $d+\tau+\nu$. The other estimates follow from the same calculation with a lower differentiation order.
\end{proof}

The proof applies verbatim to the kernel $K_{\mathcal R}$ of $V_{\mathcal R}$ for infinite $\mathcal R\subseteq\mathcal D$, the series defining it converging absolutely away from the diagonal, and in particular to the kernel $K\coloneqq K_{\mathcal D}$ of $V$. As the proof bounds the series of the absolute values of the terms, dominated convergence along finite subcollections $\mathcal R_n$ increasing to $\mathcal R$, together with the convergence of $V_{\mathcal R_n}g$ to $V_{\mathcal R}g$ in $L^2(\R^d)$, and thus almost everywhere along a subsequence, gives, for all $\mathcal R\subseteq\mathcal D$ and $g\in L^2(\R^d)$,
\begin{equation}\label{e:kernelrep}
V_{\mathcal R}g(x)=\int_{\R^d}K_{\mathcal R}(x,z)g(z)\,\d z\qquad\textrm{for almost every }x\textrm{ with }\dist(x,\supp g)>0,
\end{equation}
the integral converging absolutely.

In using these bounds for the full H\"older seminorm of a localized test, increments outside the separated regime are estimated by the size bounds at the two endpoints, or by \eqref{e:kx} for an output difference. The increment is then comparable to the separation, so these estimates give the required quotient bound as well.

\subsection{Polynomials}\label{ss:poly}
A polynomial need not be integrable against the input packets, and acts through $V$ only once it is cut off. The next two lemmas show that the cutoff testing condition \eqref{e:V1} holds when the input packets have vanishing mean, and that it suffices for polynomials of any degree to act boundedly through $V$, once they are cut off.
\begin{lemma}\label{l:V1mean}
If the input packets have vanishing mean, then $V_{\mathcal R}$ satisfies \eqref{e:V1} up to an admissible constant, uniformly over $\mathcal R\subseteq\mathcal D$.
\end{lemma}
\begin{lemma}\label{l:poly}
Let $m\geq1$, and assume that $V$ satisfies \eqref{e:V1}. Uniformly over cubes $Q$ and polynomials $\pi\in\P^{m-1}$,
\begin{equation}\label{e:poly}
\left\|V\left(\pi\omega^\infty_Q\right)\right\|_{L^\infty(\R^d)}\lesssim \|\pi\|_{L^\infty(3Q)}.
\end{equation}
\end{lemma}
\begin{proof}[Proof of Lemmas \ref{l:V1mean} and \ref{l:poly}]
Let $Q_*$ stand for the unit cube $[-\frac12,\frac12)^d$, so that $\omega^\infty_{Q_*}=\omega$. Rescaling the packets and the grid by the affine map sending $Q$ to $Q_*$, we may assume $Q=Q_*$: the rescaled grid still has sidelengths in a geometric sequence of ratio $2$ and, at each sidelength $q$, centers in a translate of $q\Z^d$, which are the only properties of the grid used here and in Lemma \ref{l:kernel}. By homogeneity, we may also assume $\|\pi\|_{L^\infty(3Q_*)}=1$, with $\pi=1$ for Lemma \ref{l:V1mean}. Let $g\coloneqq\pi\omega$ and $\iota\coloneqq\frac12\min\{1,\sigma+\rho,\sigma'+\rho'\}$. The decay of the output packets gives, as $\iota<\sigma'+\rho'$,
\[
\sup_{q}\sup_{x\in\R^d}\sum_{\ell(R)=q}|R||\psi_R(x)|\left(1+\frac{|x-c(R)|}{q}\right)^\iota\lesssim1,
\]
the supremum being taken over the sidelengths $q$ of the grid, which form a geometric sequence of ratio $2$, so that $\min\{q^\iota,q^{-d}\}$ is summable over $q$, with sum $\lesssim1$. The function $g$ is supported in $[-1,1]^d$ and, by the equivalence of norms on $\P^{m-1}$, satisfies $\|g\|_\infty+\|\nabla g\|_\infty\lesssim1$.

For Lemma \ref{l:V1mean}, it suffices to prove that
\begin{equation}\label{e:polyR}
|\varphi_R(g)|\lesssim\min\{q^\iota,q^{-d}\},\qquad \ell(R)=q,
\end{equation}
for each $R\in\mathcal D$ and each $\varphi_R\in\Phi^{\sigma,\rho}(R)$ with $\int\varphi_R=0$. Indeed, \eqref{e:polyR} and the first display give, for all $x\in\R^d$,
\[
\sum_{R\in\mathcal D}|R||\varphi_R(g)||\psi_R(x)|\leq\sum_q\min\{q^\iota,q^{-d}\}\sum_{\ell(R)=q}|R||\psi_R(x)|\lesssim\sum_q\min\{q^\iota,q^{-d}\}\lesssim1.
\]
Hence, for each $\mathcal R\subseteq\mathcal D$, the series defining $V_{\mathcal R}g$ converges absolutely at every point, its terms being dominated by those on the left hand side, and its sum is bounded independently of $\mathcal R$; it coincides almost everywhere with $V_{\mathcal R}g$, the $L^2$ limit of the finite subsums. For the first estimate in \eqref{e:polyR}, observe that
\[
|g(y)-g(c(R))|\lesssim\min\{1,|y-c(R)|\}\leq|y-c(R)|^\iota,\qquad y\in\R^d.
\]
The decay of $\varphi_R\in\Phi^{\sigma,\rho}(R)$ and the change of variables $y=c(R)+qw$ give
\[
\int_{\R^d}|\varphi_R(y)||y-c(R)|^\iota\,\d y\leq q^\iota\int_{\R^d}\frac{|w|^\iota\,\d w}{(1+|w|)^{d+\sigma+\rho}}\lesssim q^\iota,
\]
the last integral being finite as $\iota<\sigma+\rho$. As $\varphi_R$ has vanishing mean, it follows that
\[
|\varphi_R(g)|=\left|\int_{\R^d}\varphi_R(y)\left[g(y)-g(c(R))\right]\d y\right|\lesssim q^\iota.
\]
For the second estimate, the size of the input packet gives $|\varphi_R(g)|\leq q^{-d}\|g\|_1\lesssim q^{-d}$.

For Lemma \ref{l:poly}, let first $x\notin3Q_*$. As $\dist(x,\supp g)\geq\frac12$, the representation \eqref{e:kernelrep} and the size estimate \eqref{e:ksize} of the kernel give $|Vg(x)|\lesssim\|g\|_1\lesssim1$ for almost every such $x$. For $x\in3Q_*$, write $g=\pi(x)\omega+g_x$, with $g_x\coloneqq(\pi-\pi(x))\omega$. As $|g_x(y)|\lesssim\min\{1,|y-x|\}\leq|y-c(R)|^\iota+|x-c(R)|^\iota$, the argument for \eqref{e:polyR} gives, with no use of cancellation,
\[
|\varphi_R(g_x)|\lesssim\left(1+\frac{|x-c(R)|}{q}\right)^\iota\min\{q^\iota,q^{-d}\},\qquad \ell(R)=q,
\]
so that, by the first display, the series in the next display converges absolutely, with sum bounded uniformly in $x\in3Q_*$. Passing to the limit over finite subcollections,
\[
Vg(x)=\pi(x)V\omega(x)+\sum_{R\in\mathcal D}|R|\varphi_R(g_x)\psi_R(x)\qquad\textrm{for almost every }x\in3Q_*,
\]
and the first summand is bounded by \eqref{e:V1}, as $|\pi(x)|\leq1$.
\end{proof}

\section{Percentile sparse domination of model wavelet operators}\label{s:A}

In this section, $V$ is a fixed $(\sigma,\sigma';\rho,\rho')$-model wavelet operator, $\tau$ is as in \eqref{e:range}, and the percentile oscillations \eqref{e:po} are computed with smoothness parameter $\tau$.
Throughout the section, the constants implicit in the notation $\lesssim$ are admissible in the sense of \S\ref{ss:classes}. In particular, they are independent of the percentile level $0<r\leq r_0$, where $r_0\leq\frac14\eps_d$, $\eps_d$ being as in Proposition \ref{p:principle}, and of the model $V$ itself, the normalization \eqref{e:V1} being in force whenever it is assumed.
\begin{theorem}\label{t:model}
Let $m\geq 0$ be an integer and $0<r\leq r_0$, and assume that $V$ satisfies the cutoff testing condition \eqref{e:V1}. For each $f\in L^2(\R^d)$ with compact support there exists a sparse collection $\mathcal S$ of cubes of $\R^d$ such that
\[
|Vf|\lesssim\sum_{S\in \mathcal S}\nat{f}_{m,S}\ind_S \qquad \textrm{a.e.\ on } \R^d.
\]
If $m=0$, the assumption \eqref{e:V1} is not needed.
\end{theorem}

By Lemma \ref{l:V1mean}, Theorem \ref{t:model} applies to every model with mean zero input packets. By Lemma \ref{l:para}, it also applies to the canonical paraproducts with symbol in $L^\infty(\R^d)$, and, for $m=0$, to those with symbol in $\mathrm{BMO}(\R^d)$. The cutoff testing condition is necessary for the conclusion of Theorem \ref{t:model} with $m\geq1$, see Remark \ref{r:V1}.

With $m$ as in Theorem \ref{t:model}, associate to each cube $Q$ and each $f\in L^2(\R^d)$ with compact support the quantities
\begin{equation}\label{e:FQ}
h_Q\coloneqq f-\PM^m_{3Q}f,\qquad F_Q\coloneqq\ind_Q V(\omega^\infty_Qh_Q),\qquad \lambda_Q\coloneqq\nat{f}_{m,Q}.
\end{equation}
Choose a representative of each $F_Q$.
The proof of Theorem \ref{t:model} verifies the assumptions of Proposition \ref{p:principle} for the family \eqref{e:FQ}, indexed by the dyadic subcubes of a large cube $Q_0$, and estimates $Vf$ outside $Q_0$ by a single term at each point: the local weak-type estimate is proved in \S\ref{ss:wt}, the oscillation estimate in \S\ref{ss:osc}, and the proof is completed in \S\ref{ss:pfA}.

\subsection{The local weak-type estimate}\label{ss:wt}
The next proposition uses neither the cancellation of the packets nor any structure of the function $h$. Its proof follows the level set estimate of \cite{CLR26}*{Proposition 3.5}: the bad parts of a Calder\'on--Zygmund decomposition at the level of the smooth maximal function have vanishing mean, and are paired with first differences of the kernel, which are tested from good points. We include the argument, as it only uses the regularity \eqref{e:ksize}--\eqref{e:kholder} of the kernel in the input variable.
\begin{proposition}\label{p:wt}
For every $0<\eps\leq 1$ there exists $C_\eps\lesssim\eps^{-1}$ such that
\[
\left|\left\{x\in Q: \left|V(\omega^\infty_Qh)(x)\right|>C_\eps\left(\M^\tau_{3Q}h\right)^{*,Q}(r)\right\}\right|\leq \eps|Q|
\]
for all cubes $Q$, all $0<r\leq \frac\eps2$, and all $h\in L^1_{\mathrm{loc}}(\R^d)$ with $h\ind_{3Q}\in L^2(\R^d)$.
\end{proposition}
\begin{proof}
We employ the local notations \[
v\coloneqq\omega^\infty_Q h, \qquad\lambda\coloneqq(\M^\tau_{3Q}h)^{*,Q}(r).\] If $\lambda=0$, then $\M^\tau_{3Q}h$ vanishes at some point of $Q$. Testing \eqref{e:max} with the cube $3Q$ shows that $h$, and thus $v$, vanish almost everywhere on $3Q$. By homogeneity, we may thus assume $\lambda=1$. Multiplication by $\omega^\infty_Q$ maps each test function of \eqref{e:max}, relative to a cube $R\subseteq 3Q$, to a bounded multiple of a test function in the same class, relative to the same cube, by the Leibniz rule, as $0\leq\omega\leq1$ and $|\partial^\gamma\omega^\infty_Q|\lesssim\ell(Q)^{-|\gamma|}\lesssim\ell(R)^{-|\gamma|}$, no cancellation being required of the test functions. Hence
\begin{equation}\label{e:vh}
\M^\tau_{3Q}v\lesssim\M^\tau_{3Q}h\qquad \textrm{on } 3Q.
\end{equation}
\emph{Whitney decomposition.} Let \[U\coloneqq\left\{x\in(3Q)^\circ:\M^\tau_{3Q}h(x)>1\right\},\] so that $|U\cap Q|\leq r|Q|$ by \eqref{e:rearr}. The set $U$ is open by the lower semicontinuity noted in \S\ref{ss:max}.  Let
\[
\mathcal W\coloneqq \textrm{maximal cubes } W \in\mathcal D(3Q) \textrm{ such that } \Lambda W\subseteq U, \qquad \Lambda\coloneqq 8d.
\]
As $U$ is open and $3\Lambda Q\not\subseteq U$, the collection $\mathcal W$ is well defined, its cubes are pairwise disjoint, and their union is $U$. This is a dyadic Whitney decomposition of $U$, cf.\ \cite{DGWdom}*{(W1)--(W4)}, and the standard construction, cf.\ \cite{stein70}*{Chapter VI}, provides a smooth partition of unity $\{{\eta_W^\infty}: W\in \mathcal W \}$ of $U$ with \[\supp{\eta_W^\infty}\subseteq 2W, \qquad \inf_{W}{\eta_W^\infty}\gtrsim1, \qquad   \ell(W)^{|\gamma|}|\partial^\gamma{\eta_W^\infty}|\lesssim 1.\]
Let $W\in\mathcal W$ with $2W\cap 2Q\neq\varnothing$. If $\Lambda\widehat W\subseteq(3Q)^\circ$, let $S_W\coloneqq \Lambda\widehat W$, which is not contained in $U$, by the maximality of $W$. Otherwise, $\ell(Q)\leq(2\Lambda+3)\ell(W)$, as $2W\cap2Q\neq\varnothing$; let $S_W\coloneqq 3Q$, which contains points of $Q\setminus U$, as $|U\cap Q|<|Q|$. In both cases, $S_W$ is a cube with $\Lambda W\subseteq S_W\subseteq 3Q$ and $\ell(S_W)\lesssim\ell(W)$, containing a point $z\in(3Q)^\circ\setminus U$, at which $\M^\tau_{3Q}v(z)\lesssim1$, by \eqref{e:vh}. By \eqref{e:change}, each $\varphi\in \Phi^{\tau,1}(W)$ supported in $2W$ is a bounded multiple of a test function of \eqref{e:max} relative to $S_W$, whence \begin{equation}\label{e:test}
|\varphi(v)|\lesssim 1 \qquad \forall W\in\mathcal W, \, 2W\cap 2Q\neq\varnothing, \quad  \forall \varphi\in \Phi^{\tau,1}(W) \textrm{ with } \supp\varphi\subseteq 2W.
\end{equation}
\emph{Decomposition and good part.} We now decompose $v$ at height one. For $W\in\mathcal W$, set
\[
p_W\coloneqq\frac{\int_{\R^d}v{\eta_W^\infty}}{\int_{\R^d}{\eta_W^\infty}},\qquad b_W\coloneqq(v-p_W){\eta_W^\infty},\qquad g\coloneqq v-\sum_{W\in\mathcal W}b_W=v\ind_{\R^d\setminus U}+\sum_{W\in\mathcal W}p_W{\eta_W^\infty}.
\]
Each $b_W$ has vanishing mean. Moreover, $p_W=b_W=0$ when $2W\cap2Q=\varnothing$, as $v$ is supported in $2Q$, so that only the cubes covered by \eqref{e:test} play a role below. For such $W$, the function ${\eta_W^\infty}/\int{\eta_W^\infty}$ belongs to $C\Phi^{\tau,1}(W)$ and is supported in $2W$, as $\int{\eta_W^\infty}\gtrsim|W|$. Therefore \eqref{e:test} yields
\begin{equation}\label{e:polbound}
\sup_{W\in\mathcal W}|p_W|\lesssim1.
\end{equation}
On $U$, the values of $v$ have been replaced in $g$ by the constants $p_W$, so that $|g|\lesssim1$ on $U$, by \eqref{e:polbound} and as the functions ${\eta_W^\infty}$ form a partition of unity of $U$. On $(3Q)^\circ\setminus U$, Lemma \ref{l:diff} and \eqref{e:vh} give $|g|=|v|\lesssim\M^\tau_{3Q}v\lesssim\M^\tau_{3Q}h\leq1$ almost everywhere, while $g$ vanishes outside $3Q$. Therefore
\begin{equation}
\label{e:Tg2}
\|g\|_\infty\lesssim1,\qquad \supp g\subseteq3Q,\qquad \|Vg\|_2\lesssim|Q|^{\frac12},
\end{equation}
the last estimate by \eqref{e:L2}.

\emph{Bad part.} It remains to estimate $V(v-g)$ off $U$. The functions $b_W$ are supported in $U\subseteq3Q$, where the sum of their absolute values is at most $|v|+\sup_W|p_W|$. Hence their series converges to $v-g$ in $L^2(\R^d)$, and, as $V$ is bounded on $L^2(\R^d)$, an almost everywhere convergent subsequence of the images of the partial sums shows that $|V(v-g)|$ is at most the sum of the functions $|Vb_W|$, almost everywhere. Fix $W\in\mathcal W$ with $2W\cap2Q\neq\varnothing$ and $x\in\R^d\setminus U$. As $\Lambda W\subseteq U$,
\[
|x-c(W)|\geq4d\ell(W)\geq4|z-c(W)|,\qquad z\in2W,
\]
so that $x$ and $2W$ are in the separated regime of Lemma \ref{l:kernel}; this is the reason for the factor $8d$ in $\Lambda$. Let
\[
\Xi_{W,x}\coloneqq{\eta_W^\infty}\left[K(x,\cdot)-K(x,c(W))\right].
\]
As $b_W$ has vanishing mean, the value of the kernel at $c(W)$ may be subtracted at no cost and, by \eqref{e:kernelrep}, for almost every such $x$,
\[
Vb_W(x)=\Xi_{W,x}(v)-p_W\int_{\R^d}\Xi_{W,x}.
\]
With $\tau_1\coloneqq\min\{1,\tau\}$, the bracket is $\lesssim\ell(W)^{\tau_1}|x-c(W)|^{-d-\tau_1}$ on $2W$, by \eqref{e:ksize} with $|\alpha|=1$ if $\tau\geq1$, and by \eqref{e:kholder} with $\beta=0$ if $\tau<1$; the remaining derivatives, of orders $1\leq|\alpha|\leq\lfloor\tau\rfloor$, and the H\"older quotient of order $\{\tau\}$ of those of order $\lfloor\tau\rfloor$, controlled by \eqref{e:ksize}--\eqref{e:kholder}, are dominated by the same quantity times $\ell(W)^{-|\alpha|}$ and $\ell(W)^{-\tau}$, respectively. By the Leibniz rule, $\Xi_{W,x}$ is a bump adapted to $W$, damped by the separation:
\[
\Xi_{W,x}\in C\left(\frac{\ell(W)}{|x-c(W)|}\right)^{d+\tau_1}\Phi^{\tau,1}(W),\qquad \supp\Xi_{W,x}\subseteq2W.
\]
Therefore, \eqref{e:test} applies to $\Xi_{W,x}(v)$, while $|p_W|\|\Xi_{W,x}\|_1$ is controlled by \eqref{e:polbound}, and
\[
|Vb_W(x)|\lesssim\left(\frac{\ell(W)}{|x-c(W)|}\right)^{d+\tau_1},\qquad x\in\R^d\setminus U.
\]
Integrating outside $\Lambda W$, which is possible as $\tau_1>0$, and summing over $W$,
\begin{equation}\label{e:wtbad}
\int_{Q\setminus U}|V(v-g)|\lesssim\sum_{W\in\mathcal W}\int_{\R^d\setminus\Lambda W}\left(\frac{\ell(W)}{|x-c(W)|}\right)^{d+\tau_1}\d x\lesssim\sum_{W\in\mathcal W}|W|=|U|\lesssim|Q|.
\end{equation}
\emph{Conclusion.} Chebyshev's inequality, the bound \eqref{e:Tg2} for $\|Vg\|_2$ and \eqref{e:wtbad} imply, for $t\geq1$,
\[
\begin{split}
\left|\{x\in Q:|Vv(x)|>2t\}\right| &\leq|U\cap Q|+\frac{\|Vg\|_2^2}{t^2}+\frac1t\int_{Q\setminus U}|V(v-g)|
\leq\left(r+\frac{C_0}{t}\right)|Q|
\end{split}
\]
with $C_0$ admissible. As $r\leq\frac\eps2$, the choice $t\coloneqq\max\{1,2C_0\eps^{-1}\}$ proves the proposition with $C_\eps\coloneqq2t\lesssim\eps^{-1}$.
\end{proof}

\subsection{The oscillation estimate}\label{ss:osc}
This is the step where the polynomial renormalization of $h_Q$ in \eqref{e:FQ} is used. In particular, for $m\geq1$, it requires the cutoff testing condition \eqref{e:V1}, through Lemma \ref{l:poly}.
\begin{lemma}\label{l:osc}
Let $Q$ be a cube and $f\in L^2(\R^d)$ with compact support, and, if $m\geq1$, assume that $V$ satisfies \eqref{e:V1}. With reference to \eqref{e:FQ},
\[
\left|[F_Q-F_L](x)-[F_Q-F_L](y)\right|\lesssim\M^\tau_{3Q}h_Q(y)
\]
for all $L\in\mathcal D(Q)\setminus\{Q\}$ and almost all $x,y\in L$.
\end{lemma}
\begin{proof}
Fix $L\in\mathcal D(Q)\setminus\{Q\}$ and $x,y\in L$. By \eqref{e:projdiff}, on $L$
\begin{equation}\label{e:FQFL}
F_Q-F_L=V\left[\left(\omega^\infty_Q-\omega^\infty_L\right)h_Q\right]+V\left[\omega^\infty_L\pi_L\right],\qquad \pi_L\coloneqq\PM^m_{3L}h_Q=\PM^m_{3L}f-\PM^m_{3Q}f.
\end{equation}
If $m=0$, the second summand vanishes. If $m\geq1$, Lemma \ref{l:poly}, followed by Lemma \ref{l:proj}(ii) applied to the cube $3L$, yields
\[
\left\|V\left[\omega^\infty_L\pi_L\right]\right\|_\infty\lesssim\left\|\pi_L\right\|_{L^\infty(3L)}\lesssim\inf_{3L}\M^\tau_{3L}h_Q\leq\M^\tau_{3Q}h_Q(y),
\]
which controls the oscillation of the second summand in \eqref{e:FQFL}.

We turn to the first summand in \eqref{e:FQFL}. The cutoff $\omega^\infty_Q-\omega^\infty_L$ vanishes on $\frac32L$ and is supported in $2Q$. Let $J$ be the least integer with $\omega^\infty_{2^{J+1}L}=1$ on $2Q$; then $2^J\ell(L)\sim\ell(Q)$, as $L\subseteq Q$. Telescoping,
\begin{equation}\label{e:telescope}
\omega^\infty_Q-\omega^\infty_L=\sum_{j=-1}^{J}u_j,\qquad u_j\coloneqq\left(\omega^\infty_Q-\omega^\infty_L\right)\left(\omega^\infty_{2^{j+1}L}-\omega^\infty_{2^jL}\right),
\end{equation}
the lower end of the range being dictated by $\supp\omega^\infty_{2^{-1}L}\subseteq L$.

Each $u_j$ is supported in $2^{j+2}L\setminus 2^jL$, and satisfies $|\partial^\gamma u_j|\lesssim(2^j\ell(L))^{-|\gamma|}$, as both factors vary at scale $\gtrsim2^j\ell(L)$ there: the derivatives of $\omega^\infty_L$ occur only for $j\leq0$, those of $\omega^\infty_Q$ only for $2^j\ell(L)\sim\ell(Q)$. Finally, $|z-x|\sim|z-y|\sim2^j\ell(L)$ for $z\in\supp u_j$, as $u_j$ vanishes on $\frac32L$.

If $2^j\gg 1$, then $|x-y|\ll|z-x|$ for $z\in\supp u_j$, and \eqref{e:kx}--\eqref{e:kxz}, with $y$ in place of $x+h$, show that
\[
z\longmapsto u_j(z)\left[K(x,z)-K(y,z)\right]\quad\textrm{belongs to}\quad C2^{-j\nu}\Phi^{\tau,1}\left(2^{j+2}L\right),
\]
and is supported in $2^{j+2}L$: \eqref{e:kx} controls its derivatives in $z$ of order at most $\lfloor\tau\rfloor$, and the mixed estimate \eqref{e:kxz} the H\"older quotients of order $\{\tau\}$ of those of order $\lfloor\tau\rfloor$, as the class $\Phi^{\tau,1}$ requires. For the boundedly many remaining $j$, \eqref{e:ksize}--\eqref{e:kholder} show instead that $u_jK(x,\cdot)$ and $u_jK(y,\cdot)$ belong to $C\Phi^{\tau,1}(2^{j+2}L)$. These dilates contain $y$ and are contained in $3Q$, save at the top scale $2^j\ell(L)\sim\ell(Q)$, where $3Q$ replaces them, as $\supp u_j\subseteq2Q$. Testing \eqref{e:max} at $y$ and summing over $j$, by \eqref{e:kernelrep}, for almost all $x,y\in L$,
\[
\begin{split}
&\quad \left|V\left[\left(\omega^\infty_Q-\omega^\infty_L\right)h_Q\right](x)-V\left[\left(\omega^\infty_Q-\omega^\infty_L\right)h_Q\right](y)\right|\\ &\leq\sum_{j=-1}^{J}\left|\int_{\R^d} u_j\left[K(x,\cdot)-K(y,\cdot)\right]h_Q\right|
\lesssim\M^\tau_{3Q}h_Q(y)\Big[1+\sum_{j\geq 0}2^{-j\nu}\Big]\lesssim\M^\tau_{3Q}h_Q(y).
\end{split}
\]
This completes the proof.
\end{proof}

\subsection{Proof of Theorem \ref{t:model}}\label{ss:pfA}
Fix $0<r\leq r_0$ and $f\in L^2(\R^d)$ supported in a cube $Q'$. Let $Q_0\coloneqq 4Q'$, $\eps\coloneqq \frac12\eps_d$ and, for each cube $Q$, let $h_Q$, $F_Q$, $\lambda_Q$ be as in \eqref{e:FQ}, and $E_Q$ the union of the sets
\[
\left\{x\in Q:|F_Q(x)|>C_{\eps}\lambda_Q\right\},\qquad\left\{x\in Q:\M^\tau_{3Q}h_Q(x)>\lambda_Q\right\},
\]
with $C_\eps$ as in Proposition \ref{p:wt}, and of the null set, given by Fubini's theorem, of those $y\in Q$ at which the bound of Lemma \ref{l:osc}, for some $L\in\mathcal D(Q)\setminus\{Q\}$ containing $y$, fails for a set of $x$ of positive measure. As $r\leq \frac\eps2$, Proposition \ref{p:wt} with $h=h_Q$ and the definition of $\lambda_Q$ yield $|E_Q|\leq(\frac12\eps_d+r)|Q|\leq\eps_d|Q|$, which is \eqref{e:P1}. Assumption \eqref{e:P2} holds by construction, and \eqref{e:P3} follows from Lemma \ref{l:osc}, as $\M^\tau_{3Q}h_Q\leq\lambda_Q$ on $Q\setminus E_Q$. Therefore, Proposition \ref{p:principle} yields a $\frac34$-sparse collection $\mathcal S_0\subseteq\mathcal D(Q_0)$ with $Q_0\in \mathcal S_0$ and
\[
|F_{Q_0}|\lesssim \sum_{S\in\mathcal S_0}\lambda_S\ind_S\qquad \textrm{almost everywhere on } Q_0.
\]
As $\omega^\infty_{Q_0}=1$ on $\frac32Q_0\supseteq Q'$, on $Q_0$ we have
\[
Vf=F_{Q_0}+V(\omega^\infty_{Q_0}\PM^m_{3Q_0}f).
\]
The second term vanishes when $m=0$. Otherwise, Lemma \ref{l:poly} and Lemma \ref{l:po}(iv) bound it by a constant multiple of $\lambda_{Q_0}$, as $|Q_0\cap\supp f|\leq\frac14|Q_0|$.
Outside $Q_0$, write $S_j\coloneqq 2^jQ_0$ and fix $x\in S_j\setminus S_{j-1}$, $j\geq1$. Let $\zeta$ be a smooth function with $\ind_{[-\frac18,\frac18]^d}\leq\zeta\leq\ind_{[-\frac3{16},\frac3{16}]^d}$, and $\zeta_j\coloneqq\zeta^\infty_{S_j}$. Then $\zeta_j=1$ on $\frac14S_j\supseteq Q'$ and $\dist(x,\supp\zeta_j)\geq\frac1{16}\ell(S_j)$, so that, by \eqref{e:ksize}--\eqref{e:kholder}, $\zeta_jK(x,\cdot)\in C\Phi^{\tau,1}(S_j)$ is supported in $S_j$. By \eqref{e:kernelrep} and Lemma \ref{l:po}(iv), as $|S_j\cap\supp f|\leq\frac14|S_j|$, $|Vf(x)|=|(\zeta_jK(x,\cdot))(f)|\lesssim\lambda_{S_j}$ for almost every such $x$. Hence the conclusion of Theorem \ref{t:model} holds with the collection $\mathcal S\coloneqq\mathcal S_0\cup\{S_j:j\geq1\}$, which is $\frac12$-sparse: the sets $G_S$, $S\in\mathcal S_0$, are contained in $Q_0$, while $G_{S_j}\coloneqq S_j\setminus S_{j-1}$ has measure $(1-2^{-d})|S_j|$. This completes the proof of Theorem \ref{t:model}.

\subsection{Necessity of the cutoff testing condition}\label{ss:remA}
\begin{remark}\label{r:V1}
The cutoff testing condition \eqref{e:V1} is also necessary for the conclusion of Theorem \ref{t:model} with $m\geq1$. Indeed, assume the latter holds for $V$, let $f\coloneqq\omega^\infty_Q$, and let $\mathcal S$ be the sparse collection given for $f$. If $S\in\mathcal S$ and $3S\subseteq\frac32Q$, then $f$ is constant on $3S$ and $\nat{f}_{m,S}=0$. Otherwise, Lemma \ref{l:po}(iii), applied with the constant $f(c(S))$ if $\ell(S)<\ell(Q)$ and with $0$ if $\ell(S)\geq\ell(Q)$, gives
\[
\nat{f}_{m,S}\lesssim\frac1r\min\left\{\frac{\ell(S)}{\ell(Q)},\left(\frac{\ell(Q)}{\ell(S)}\right)^d\right\}.
\]
As each point belongs to boundedly many cubes $S\in\mathcal S$ with sidelength in any given interval $[2^i,2^{i+1})$, by \eqref{e:band}, it follows that $\|V(\omega^\infty_Q)\|_\infty\lesssim C/r$, $C$ being the constant of the assumed sparse bound. Thus, for model wavelet operators, the sparse bound of Theorem \ref{t:model} of order $m\geq1$ holds if and only if \eqref{e:V1} does, in analogy with the characterization of the pointwise potential bound of order one by the condition $T1\in L^\infty(\R^d)$ in \cite{lerner26}.
\end{remark}

\section{Weights and Hardy norms}\label{s:cons}
This section turns sparse bounds into weighted norm inequalities, with linear dependence on $[w]_{A_\infty}$, and proves Proposition \ref{c:hardynorm}, through a sparse bound modulo polynomials for the smooth grand maximal function.

\subsection{Weighted bounds}
Let $w\in A_\infty$, see \S\ref{ss:classes}. The sharp reverse H\"older inequality of \cite{HP13}*{Theorem 2.3}, combined with H\"older's inequality, yields
\begin{equation}\label{e:Ainfty}
w(E)\leq2\left(\frac{|E|}{|Q|}\right)^{\frac{1}{c_d[w]_{A_\infty}}}w(Q),\qquad E\subseteq Q .
\end{equation}
\begin{lemma}\label{l:goodl}
Let $1\leq q<\infty$,  $\mathcal F$ be a finite $\eta$-sparse collection of cubes and $\lambda_S\geq0$ for $S\in\mathcal F$. Suppose $w\in A_\infty$ and $\lambda,\gamma>0$.
Then
\[
w\left(\left\{\left(\sum_{S\in\mathcal F}\lambda_S^q\ind_S\right)^{\frac1q}>2\lambda,  \max_{S\in\mathcal F}\lambda_S\ind_S \leq\gamma\lambda\right\}\right)\lesssim\exp\left(-\frac{c}{\gamma^q[w]_{A_\infty}}\right)w\left(\left\{ \left(\sum_{S\in\mathcal F}\lambda_S^q\ind_S\right)^{\frac1q} >\lambda\right\}\right),
\]
where $c>0$ and the implied constant depend on $d$ and $\eta$ only.
\end{lemma}
\begin{proof}
The argument is that of \cite{DPRDF}*{Theorem F}, which compares a linear sparse sum with a quadratic one. Write
\begin{equation}
\label{e:GNGN}
G\coloneqq\left(\sum_{S\in\mathcal F}\lambda_S^q\ind_S\right)^{\frac1q},\qquad N\coloneqq\max_{S\in\mathcal F}\lambda_S\ind_S.
\end{equation}
 For $P\in\mathcal F$ and $k\geq0$, define
\[
\mathcal C_P\coloneqq\sum_{\substack{S\in\mathcal F\\ S\subseteq P}}\ind_S,\qquad E_k(P)\coloneqq\left\{x\in P:\mathcal C_P(x)>k\right\}.
\]
First,
\begin{equation}\label{e:JN}
w\left(E_k(P)\right)\lesssim\exp\left(-\frac{ck}{[w]_{A_\infty}}\right)w(P),\qquad P\in\mathcal F,\quad k\geq0.
\end{equation}
Indeed, letting $n\coloneqq\lceil2\eta^{-1}\rceil$, a standard John--Nirenberg argument entails $|E_k(P)|\leq2^{1-k/n}|P|$, and \eqref{e:Ainfty} yields \eqref{e:JN}.
Second, let $\mathcal M$ be the collection of the maximal cubes $P\in\mathcal F$ with
\[
\sum_{\substack{S\in\mathcal F\\ S\supseteq P}}\lambda_S^q>\lambda^q.
\]
The cubes $P\in\mathcal M$ are pairwise disjoint, and their union is $\{G>\lambda\}$. Furthermore,
\[
P\cap\{G>2\lambda,N\leq\gamma\lambda\}\subseteq E_k(P),\qquad P\in\mathcal M,\quad k\coloneqq\lfloor\gamma^{-q}\rfloor>\gamma^{-q}-1.
\]
Therefore, by \eqref{e:JN} and as $[w]_{A_\infty}\geq1$,
\[
w\left(\{G>2\lambda,\;N\leq\gamma\lambda\}\right)\leq\sum_{P\in\mathcal M}w\left(E_k(P)\right)\lesssim\exp\left(-\frac{c}{\gamma^q[w]_{A_\infty}}\right)\sum_{P\in\mathcal M}w(P),
\]
and the last sum is $w(\{G>\lambda\})$.
\end{proof}
\begin{lemma}\label{l:Ainf}
Let $0<p<\infty$, $w\in A_\infty$, and let $\mathcal S$ be an $\eta$-sparse collection of cubes.
  If $ \lambda_S\geq0$ for $S\in\mathcal S$, then
\[
\Big\|\sum_{S\in\mathcal S}\lambda_S\ind_S\Big\|_{L^p(w)}\lesssim_p[w]_{A_\infty}\Big\|\sup_{S\in\mathcal S}\lambda_S\ind_S\Big\|_{L^p(w)}.
\]
The implied constants also depend on $\eta$, and grow at most linearly in $p$ as $p\to\infty$.
\end{lemma}
\begin{proof}
   By monotone convergence, we may assume that $\mathcal S$ is finite. Use the notation \eqref{e:GNGN} with $q=1$ and $\mathcal S$ in place of $\mathcal F$.
Lemma \ref{l:goodl}, with $q=1$, gives, for all $\lambda,\gamma>0$,
\[
w\left(\{G>2\lambda\}\right)\leq w\left(\{N>\gamma\lambda\}\right)+C\exp\left(-\frac{c}{\gamma[w]_{A_\infty}}\right)w\left(\{G>\lambda\}\right).
\]
Let $\gamma^{-1}\coloneqq c^{-1}[w]_{A_\infty}\log(2^{p+1}C)$. Multiplying by $p\lambda^{p-1}$ and integrating in $\lambda$,
\[
2^{-p}\|G\|_{L^p(w)}^p\leq\gamma^{-p}\|N\|_{L^p(w)}^p+2^{-p-1}\|G\|_{L^p(w)}^p,
\]
where $\|G\|_{L^p(w)}<\infty$, as $G$ is bounded with compact support. Hence
\[
\|G\|_{L^p(w)}\leq2^{1+\frac1p}\gamma^{-1}\|N\|_{L^p(w)},
\]
and $2^{1+\frac1p}\gamma^{-1}\lesssim_p[w]_{A_\infty}$ grows linearly in $p$.
\end{proof}

\subsection{Hardy norms}\label{ss:hardy}
This subsection contains the linear sparse bound modulo polynomials for the grand maximal function, and the proof of Proposition \ref{c:hardynorm}. The proof of the sparse bound is a simplified version of that of Theorem \ref{t:model}: the local weak-type estimate is built into the percentile oscillation, and the polynomials are handled by Lemma \ref{l:proj}(ii). We add it for completeness.
\begin{proposition}\label{p:grand}
Let $m\geq0$ be an integer and $0<r\leq r_0$. For each $f\in L^2(\R^d)$ with compact support there exists a sparse collection $\mathcal S$ of cubes of $\R^d$ such that
\[
\M^\tau_{\R^d}f\lesssim\sum_{S\in\mathcal S}\nat{f}_{m,S}\ind_S\qquad\textrm{almost everywhere on } \R^d.
\]
\end{proposition}
\begin{proof}
Let $f$ be supported in a cube $Q'$, and let $Q_0\coloneqq4Q'$ and $Q_j\coloneqq2^jQ_0$, $j\geq1$. For a cube $Q$, let
\[
h_Q\coloneqq f-\PM^m_{3Q}f,\qquad\lambda_Q\coloneqq\nat{f}_{m,Q},\qquad E_Q\coloneqq\left\{x\in Q:\M^\tau_{3Q}h_Q(x)>\lambda_Q\right\},
\]
so that $|E_Q|\leq r|Q|\leq\eps_d|Q|$. Let $I$ be the set of the pairs $(R,\varphi)$ with $R\in\mathcal D(Q_0)$ and $\varphi\in\Phi^{\tau,1}(3R)$ supported in $3R$, and endow $\C^I$ with the extended norm $\|v\|\coloneqq\sup_{i\in I}|v_i|$. For $Q\in\mathcal D(Q_0)$ and $x\in Q$, let
\[
F_Q(x)_{(R,\varphi)}\coloneqq\begin{cases}\ind_R(x)\,\varphi(h_Q),&R\subseteq Q,\\0,&\textrm{otherwise,}\end{cases}\qquad(R,\varphi)\in I.
\]
If $x\in R\subseteq Q$, the cube $3R$ is admissible in \eqref{e:max} for $\M^\tau_{3Q}$ at $x$, so that $\|F_Q\|\leq\M^\tau_{3Q}h_Q$ on $Q$, and \eqref{e:P2} holds. For \eqref{e:P3}, let $L\in\mathcal D(Q)\setminus\{Q\}$, $x\in L$ and $y\in L\setminus E_Q$. If $R\not\subseteq L$, then $\ind_R(x)=\ind_R(y)$, and the component of index $(R,\varphi)$ of $[F_Q-F_L](x)-[F_Q-F_L](y)$ vanishes. If $R\subseteq L$, it equals $(\ind_R(x)-\ind_R(y))\varphi(h_Q-h_L)$, and, by \eqref{e:projdiff} and Lemma \ref{l:proj}(ii) on the cube $3L$,
\[
\left|\varphi(h_Q-h_L)\right|=\left|\varphi\left(\PM^m_{3L}h_Q\right)\right|\leq\sup_{3L}\left|\PM^m_{3L}h_Q\right|\lesssim\M^\tau_{3L}h_Q(y)\leq\M^\tau_{3Q}h_Q(y)\leq\lambda_Q.
\]
Proposition \ref{p:principle}, in its vector valued form, see Remark \ref{r:principle}, yields a $\frac34$-sparse collection $\mathcal S_0\subseteq\mathcal D(Q_0)$ with $Q_0\in\mathcal S_0$ such that
\begin{equation}\label{e:grandQ0}
\left\|F_{Q_0}\right\|\lesssim\sum_{S\in\mathcal S_0}\lambda_S\ind_S\qquad\textrm{almost everywhere on } Q_0.
\end{equation}
Let $x\in\R^d$, at which \eqref{e:grandQ0} holds if $x\in Q_0$, let $j\geq0$ be the least index with $x\in Q_j$, and let $\varphi\in\Phi^{\tau,1}(\widetilde R)$ be supported in a cube $\widetilde R\ni x$, with $\varphi(f)\neq0$, so that $\widetilde R$ meets $Q'$. Two cases occur.
\begin{itemize}
\item If $j=0$ and $4\ell(\widetilde R)\leq\ell(Q_0)$, let $R\in\mathcal D(Q_0)$ contain $x$, with $2\ell(\widetilde R)\leq\ell(R)<4\ell(\widetilde R)$. Then $\widetilde R\subseteq3R$, so that $(R,\varphi/C)\in I$ for an admissible $C$, by \eqref{e:change}. Writing $f=h_{Q_0}+\PM^m_{3Q_0}f$, \eqref{e:grandQ0} and Lemma \ref{l:po}(iv), as $|Q_0\cap\supp f|\leq\frac14|Q_0|$, give
\[
|\varphi(f)|\leq C\left\|F_{Q_0}(x)\right\|+\left\|\PM^m_{3Q_0}f\right\|_{L^\infty(3Q_0)}\lesssim\sum_{S\in\mathcal S_0}\lambda_S\ind_S(x).
\]
\item Otherwise, $\ell(\widetilde R)\geq\frac18\ell(Q_j)$: by assumption if $j=0$, and, if $j\geq1$, as $\widetilde R$ meets $Q'$ and contains $x\notin Q_{j-1}$. As $x\in Q_j$, the cube $\widetilde R$ lies in the cube of center $c(Q_0)$ and sidelength $\ell(Q_j)+2\ell(\widetilde R)\leq10\ell(\widetilde R)$, so that the least $i\geq1$ with $\widetilde R\subseteq Q_i$ satisfies $\ell(Q_i)\leq20\ell(\widetilde R)$. By \eqref{e:change}, $\varphi\in C\Phi^{\tau,1}(Q_i)$ is supported in $Q_i\ni x$, and Lemma \ref{l:po}(iv), as $|Q_i\cap\supp f|\leq\frac14|Q_i|$, gives $|\varphi(f)|\lesssim\lambda_{Q_i}$.
\end{itemize}
Taking the supremum over $\widetilde R$ and $\varphi$, the proposition follows with the collection
\[
\mathcal S\coloneqq\mathcal S_0\cup\{Q_i:i\geq1\},
\]
which is $\frac12$-sparse: the sets witnessing the sparseness of $\mathcal S_0$ lie in $Q_0$, while the sets $Q_i\setminus Q_{i-1}$, $i\geq1$, are pairwise disjoint, do not meet $Q_0$, and have measure $(1-2^{-d})|Q_i|\geq\frac12|Q_i|$.
\end{proof}
\begin{proof}[Proof of Proposition \ref{c:hardynorm}]
As $\nat{f}_{m,S}\leq\mathcal N_mf$ on $S$, Proposition \ref{p:grand} and Lemma \ref{l:Ainf} give \eqref{e:hardynorm}. The comparability of the norms of $\mathcal N_nf$ and $\M^\tau_{\R^d}f$ follows from \eqref{e:hardynorm}, with $n$ in place of $m$, and from the reverse inequality of \S\ref{ss:hardymain}, and the last claim from the equivalence $\|\M^\tau_{\R^d}f\|_{L^p(\R^d)}\sim\|f\|_{H^p(\R^d)}$, $\frac d{d+\tau}<p\leq1$, recalled there, which we verify for the tests of \eqref{e:max}.

On the one hand, $\M^\tau_{\R^d}f$ dominates a constant multiple of the maximal function of $f$ built on a fixed $\phi\in \mathcal{C}^\infty_c(\R^d)$ with $\int\phi=1$, which controls $\|f\|_{H^p(\R^d)}$ \cite{FS72}. On the other hand, let $a$ be an atom of $H^p(\R^d)$ supported in a cube $R$, with vanishing moments up to order $\lfloor\tau\rfloor$ and $\|a\|_\infty\leq|R|^{-1/p}$. Taylor's formula at $c(R)$ gives
\[
\M^\tau_{\R^d}a(x)\lesssim|R|^{1-\frac1p}\,\frac{\ell(R)^\tau}{|x-c(R)|^{d+\tau}},\qquad |x-c(R)|\geq2\sqrt d\,\ell(R),
\]
which is $p$-integrable at infinity as $p>\frac d{d+\tau}$, while $\M^\tau_{\R^d}a\lesssim\M a$ has $L^2$ norm $\lesssim|R|^{\frac12-\frac1p}$, so that $\|\M^\tau_{\R^d}a\|_{L^p(\R^d)}\lesssim1$. Let $c_j$ and $a_j$ be the coefficients and the atoms of an atomic decomposition of $f$, converging in $\mathcal S'(\R^d)$, and let $\varphi$ be a test of \eqref{e:max} relative to a cube $P\ni x$. Its mollifications $\varphi_\varepsilon$ at the scales $\varepsilon\leq\ell(P)$ are admissible multiples of tests relative to the cube $3P$, by \eqref{e:change}, and converge to $\varphi$ uniformly, with supports in $3P$. Hence
\[
|\varphi(f)|=\lim_{\varepsilon\to0}|\varphi_\varepsilon(f)|\lesssim\sum_j|c_j|\,\M^\tau_{\R^d}a_j(x),
\]
and the $p$-triangle inequality, $p\leq1$, gives $\|\M^\tau_{\R^d}f\|_{L^p(\R^d)}\lesssim\|f\|_{H^p(\R^d)}$.
\end{proof}

\section{Sparse bounds for Calder\'on--Zygmund operators}\label{s:cz}
We prove the representation of Proposition \ref{p:rep} and use it to obtain the sparse bounds of Theorems \ref{t:HS} and \ref{t:jet}, and of Corollaries \ref{c:czpot} and \ref{c:sobjet}. The converse implications follow by testing on cut-off polynomials. The section ends with the sharp $A_p$ bounds of Proposition \ref{c:czweighted}.

\subsection{Test spaces and local notation}\label{ss:czloc}
The notation is that of \S\ref{ss:wavelets}, \S\ref{ss:forms}, \S\ref{ss:modeldef} and \S\ref{ss:pararep}. By a smooth partition of unity, the weak boundedness property \eqref{e:wbp} extends, up to a constant depending on $C_0$, to functions $\zeta_R,\xi_R$ supported in $C_0R$, $C_0$ being any admissible constant. This is used without further mention.

As in \cite{DGWdom}*{Remark 2.6}, the Calder\'on--Zygmund operators of Definition \ref{d:czo} act on classes of functions defined through their wavelet coefficients. For an integer $j\geq0$, let $\mathcal U_j$ be the space of those $f\in L^2(\R^d)$ for which
\[
\|f\|_{\mathcal U_j}\coloneqq\sum_{t\in2^{\Z}}\left(1+t^{-1}\right)^j\Big(\sum_{S\in\mathcal D,\,\ell(S)=t}|S|\left|\psi_S(f)\right|^2\Big)^{\frac12}<\infty.
\]
Then $\mathcal U_{j+1}\subseteq\mathcal U_j\subseteq L^2(\R^d)$, with $\|f\|_{L^2(\R^d)}\leq\|f\|_{\mathcal U_j}\leq\|f\|_{\mathcal U_{j+1}}$, the expansions \eqref{e:onb} of the elements of $\mathcal U_j$ converge in $\mathcal U_j$, and, for $j\leq k$, $\mathcal U_j$ contains $\mathcal{C}^\infty_c(\R^d)$ and the classes $\Phi^{j+\delta,\Subset}(R)$, by (ii) of the next lemma.
\begin{lemma}\label{l:Uj}
Let $0\leq\kappa\leq k$.
\begin{enumerate}[label=(\roman*)]
\item If $|\alpha|=\kappa$ and $g\in\mathcal U_\kappa$, the distribution $\partial^\alpha g$ belongs to $\mathcal U_0$, and $\|\partial^\alpha g\|_{\mathcal U_0}\lesssim\|g\|_{\mathcal U_\kappa}$.
\item If $K\subset\R^d$ is compact and $0<s<1$, then $\|u\|_{\mathcal U_\kappa}\lesssim_{K,s}\|u\|_{\mathcal{C}^{\kappa,s}(\R^d)}$ for all $u\in\mathcal{C}^{\kappa,s}(\R^d)$ supported in $K$, the norm on the right being the sum of the uniform norms of the derivatives of $u$ of order at most $\kappa$ and of the $s$-H\"older seminorms of those of order $\kappa$.
\item If $g$ is a finite linear combination of wavelets $\psi_S$, its mollifications at the scales $0<\eps\leq1$ are supported in a fixed compact set, and converge to $g$ in $\mathcal U_\kappa$ and uniformly as $\eps\to0$.
\end{enumerate}
\end{lemma}
\begin{proof}
Part (i) follows from almost orthogonality between the differentiated wavelets at different scales. For (ii), Taylor's formula and the vanishing moments give a summable factor $t^s$ at scales $t\leq1$, while the size bound gives $t^{-\frac d2}$ at scales $t>1$. Finally, a finite wavelet sum belongs to $\mathcal{C}^{k,1}$ and has compact support, so its mollifications converge in $\mathcal{C}^{\kappa,s}$ for $0<s<1$; (ii) gives (iii).
\end{proof}
As the expansions \eqref{e:onb} converge in $\mathcal U_0$ and in $\mathcal U_j$, they may be inserted in the pairings $\langle Tf,g\rangle$, $f\in\mathcal U_0$, $g\in\mathcal U_j$, by continuity, whenever the resulting series converge absolutely.
The basis $\{\psi_S\}$ of \S\ref{ss:wavelets} is a Daubechies basis, and stems from a multiresolution analysis with compactly supported scaling function \cite{MC97}. Denote by $\mathbb E_\ell$ the orthogonal projection onto the closed span of $\{\psi_S:\ell(S)>\ell\}$, so that
\begin{equation}\label{e:mra}
f-\mathbb E_\ell f=\sum_{\ell(S)\leq\ell}|S|\psi_S(f)\psi_S,\qquad f\in L^2(\R^d).
\end{equation}
The kernel of $\mathbb E_\ell$ is a sum of tensor products of scaling functions, so that $\mathbb E_\ell$ is bounded on $L^1(\R^d)$ and on $L^\infty(\R^d)$ uniformly in $\ell$, and, the wavelets being compactly supported, the identity \eqref{e:mra} extends to $f\in L^\infty(\R^d)$, almost everywhere, the series converging in $L^2$ on compacta.

\subsection{Polynomial actions and symbol size}\label{ss:cutoff}
If $\partial^\alpha Tf$ agrees on $\mathcal{C}^\infty_c(\R^d)$ with a locally integrable function $h$, then \[\langle\partial^\alpha Tf,g\rangle=\int hg\] for every finite wavelet sum $g$ by Lemma \ref{l:Uj}(iii). The next lemma allows us to regard $T$ as an operator of order $\kappa=|\alpha|$, and gives the polynomial actions used in \eqref{e:symbols}.
\begin{lemma}\label{l:reduce}\label{l:cutoff}
Let $0\leq\kappa\leq k$.
\begin{enumerate}[label=(\roman*)]
\item A Calder\'on--Zygmund operator $T$ of order $k$ is a constant multiple of one of order $\kappa$. In particular, for $|\alpha|=\kappa$, $\partial^\alpha T$ is a constant multiple of a Calder\'on--Zygmund operator of order $(\kappa,\kappa)$.
\item Let $\mathfrak L$ be a Calder\'on--Zygmund operator of order $(\kappa,\kappa)$, with kernel $K$, $Q$ a cube, $\pi$ a polynomial of degree at most $\kappa$, and $g\in\mathcal U_\kappa$ supported in a cube $P$, with $\int g=0$ if $\pi$ has degree $\kappa$. Then
\[
\lim_{A\to\infty}\left\langle\mathfrak L\left(\pi\omega^\infty_{AQ}\right),g\right\rangle
\]
exists, does not depend on $Q$, is linear in $\pi$ and in $g$, and equals
\begin{equation}\label{e:actionsplit}
\left\langle\mathfrak L\left(\pi\omega^\infty_P\right),g\right\rangle+\int_{\R^d}\pi(y)\left[1-\omega^\infty_P\right](y)\int_{\R^d}K(x,y)\,g(x)\,\d x\,\d y.
\end{equation}
We denote this limit by $\langle\mathfrak L\pi,g\rangle$; for $g=\psi_S$ it agrees with \eqref{e:polyaction}.
\item For $\mathfrak L$ as in \emph{(ii)} and $S\in\mathcal D$, the limit $\langle\mathfrak L\psi_S,\mathbf 1\rangle$ in \eqref{e:polyaction} exists. If $\mathfrak L=\partial^\alpha T$, with $T$ as in \emph{(i)} and $|\alpha|=\kappa\geq1$, it vanishes.
\end{enumerate}
\end{lemma}
\begin{proof}
We prove (i). The kernel estimates of order $(\kappa,0)$ are among those of order $(k,0)$, and it remains to verify \eqref{e:wbp} with $\kappa$ in place of $k$. Let $\zeta_R\in\Phi^{\kappa+\delta,\Subset}(R)$ and $\xi_R\in\Phi^{\delta,\Subset}(R)$. As $\psi_S$ annihilates $\P^{k+1}$ and is supported in $\w S$, subtracting from $\zeta_R$ its Taylor polynomial of degree $\kappa$ at $c(S)$ if $\ell(S)\leq\ell(R)$, and the bound $\|\psi_S\|_\infty\lesssim|S|^{-1}$ otherwise, give
\[
|\psi_S(\zeta_R)|\lesssim\frac1{|R|}\min\left\{\left(\frac{\ell(S)}{\ell(R)}\right)^{\kappa+\delta},\left(\frac{\ell(R)}{\ell(S)}\right)^d\right\},
\]
and $\psi_S(\zeta_R)=0$ unless $\w S\cap\w R\neq\varnothing$. For such $S$, we claim that
\begin{equation}\label{e:wavewbp}
|\langle T\psi_S,\xi_R\rangle|\lesssim\frac1{|R|}\quad\textrm{if }\ell(S)\leq\ell(R),\qquad |\langle T\psi_S,\xi_R\rangle|\lesssim\frac1{|S|}\left(1+\log\frac{\ell(S)}{\ell(R)}\right)\quad\textrm{if }\ell(S)>\ell(R).
\end{equation}
Granting the claim, as there are $\lesssim(\ell(R)/\ell)^d$ such cubes $S$ of each sidelength $\ell\leq\ell(R)$, and boundedly many of each sidelength $\ell>\ell(R)$,
\[
\sum_{S\in\mathcal D}|S||\psi_S(\zeta_R)||\langle T\psi_S,\xi_R\rangle|\lesssim\frac1{|R|}.
\]
Thus the expansion of $\zeta_R$ in \eqref{e:onb} may be inserted in $\langle T\zeta_R,\xi_R\rangle$, and \eqref{e:wbp} follows. For the claim, let $P$ be the smaller of the cubes $S$ and $R$, and $\phi$ a smooth function equal to one on $4\sqrt d\,\w P$ and supported in $8\sqrt d\,\w P$, with $|\nabla^i\phi|\lesssim\ell(P)^{-i}$. If $P=S$, then $\xi_R\phi$ is $|S||R|^{-1}$ times an element of $C\Phi^{\delta,1}(S)$ supported in $8\sqrt d\,\w S$, and \eqref{e:wbp} at the cube $S$ gives $|\langle T\psi_S,\xi_R\phi\rangle|\lesssim|R|^{-1}$, while the vanishing mean of $\psi_S$ and the H\"older bound in \eqref{e:czkernel} give
\[
\left|\langle T\psi_S,\xi_R(1-\phi)\rangle\right|\lesssim\frac1{|R|}\int_{|x-c(S)|\gtrsim\ell(S)}\frac{\ell(S)^\delta}{|x-c(S)|^{d+\delta}}\,\d x\lesssim\frac1{|R|}.
\]
If $P=R$, then $\psi_S\phi$ is $|R||S|^{-1}$ times an element of $C\Phi^{k+\delta,1}(R)$ supported in $8\sqrt d\,\w R$, and \eqref{e:wbp} at the cube $R$ gives $|\langle T(\psi_S\phi),\xi_R\rangle|\lesssim|S|^{-1}$, while the size bound in \eqref{e:czkernel} gives
\[
\left|\langle T(\psi_S(1-\phi)),\xi_R\rangle\right|\lesssim\int|\xi_R(x)|\int_{\ell(R)\lesssim|y-c(R)|\lesssim\ell(S)}\frac{|S|^{-1}}{|x-y|^d}\,\d y\,\d x\lesssim\frac1{|S|}\left(1+\log\frac{\ell(S)}{\ell(R)}\right).
\]
For $|\alpha|=\kappa$, Lemma \ref{l:Uj}(i) gives the continuity of $\partial^\alpha T$ from $\mathcal U_0$ to $\mathcal U_\kappa'$, and \cite{DGWdom}*{Lemma 5.4} gives its kernel and weak boundedness estimates of order $(\kappa,\kappa)$.

We next prove (ii) and (iii).  For $A$ large, $\omega^\infty_{AQ}-\omega^\infty_P$ vanishes on $\frac32P$, and
\[
\left\langle\mathfrak L\left(\pi\omega^\infty_{AQ}\right),g\right\rangle=\left\langle\mathfrak L\left(\pi\omega^\infty_P\right),g\right\rangle+\int_{\R^d}\pi(y)\left[\omega^\infty_{AQ}-\omega^\infty_P\right](y)\int_{\R^d}K(x,y)\,g(x)\,\d x\,\d y.
\]
If $\pi$ has degree less than $\kappa$, the size bound in \eqref{e:czkernel} bounds the inner integral by a constant multiple of $\|g\|_1|y-c(P)|^{-d-\kappa}$ off $\frac32P$, which is integrable against $|\pi(y)|$. If $\pi$ has degree $\kappa$, as $\int g=0$, $K(x,y)$ may be replaced by $K(x,y)-K(c(P),y)$, and the H\"older bound in \eqref{e:czkernel} provides an integrable majorant at infinity. By dominated convergence, the limit in (ii) exists and equals \eqref{e:actionsplit}, which is linear in $\pi$ and in $g$, and does not depend on $Q$; it is the action of \cite{DGWdom}*{(5.4)}. Likewise, with $\varpi\coloneqq\omega^\infty_{4\sqrt d\w S}$, for $A$ large,
\[
\left\langle\mathfrak L\psi_S,\omega^\infty_{AS}\right\rangle=\langle\mathfrak L\psi_S,\varpi\rangle+\int_{\R^d}\left[\omega^\infty_{AS}-\varpi\right](x)\int_{\R^d}K(x,y)\,\psi_S(y)\,\d y\,\d x,
\]
and the vanishing mean of $\psi_S$ gives \[\left|\int K(x,y)\psi_S(y)\,\d y\right|\lesssim\ell(S)^\delta|x-c(S)|^{-d-\kappa-\delta}\] where $\varpi(x)\neq1$, so that dominated convergence applies. Finally, if $\mathfrak L=\partial^\alpha T$ with $|\alpha|=\kappa\geq1$, then $\langle\partial^\alpha T\psi_S,\omega^\infty_{AS}\rangle=(-1)^\kappa\langle T\psi_S,\partial^\alpha\omega^\infty_{AS}\rangle$ by \eqref{e:dalpha}, where $\partial^\alpha\omega^\infty_{AS}$ is supported in $2AS\setminus\frac32AS$, with $L^1$ norm $\lesssim(A\ell(S))^{d-\kappa}$. The kernel bounds of $T$ and the vanishing mean of $\psi_S$ bound the last pairing by a constant multiple of $\ell(S)^\delta(A\ell(S))^{-d-\delta}(A\ell(S))^{d-\kappa}$, which tends to zero.
\end{proof}
\begin{lemma}\label{l:size}
Let $T$ be a Calder\'on--Zygmund operator of order $k$, $|\alpha|=\kappa\leq k$ and $|\gamma|\leq\kappa$.
\begin{enumerate}[label=(\roman*)]
\item The symbols $b^{\alpha,\gamma}$ satisfy \eqref{e:size}, and, if $\kappa=0$, $|b^{0,\star}_S|\lesssim1$ for all $S\in\mathcal D$.
\item If $|\gamma|<\kappa$ and $X$ is $L^\infty(\R^d)$ or $\mathrm{BMO}(\R^d)$, then $b^{\alpha,\gamma}$ belongs to $X$ if and only if it belongs locally to $X$, and its norm and its local norm differ by at most an admissible constant.
\end{enumerate}
\end{lemma}
\begin{proof}
We prove (i). Let $S\in\mathcal D$, and let $Q$ be the cube of center $c(S)$ and sidelength $8\sqrt d\,\w\ell(S)$. With $K$ the kernel of $\partial^\alpha T$, \eqref{e:actionsplit}, with $P\coloneqq Q$, gives
\[
b^{\alpha,\gamma}_S=\left\langle\partial^\alpha T\left[(\cdot-c(S))^\gamma\omega^\infty_Q\right],\psi_S\right\rangle+\int_{\R^d}(y-c(S))^\gamma\left[1-\omega^\infty_Q\right](y)\int_{\R^d}K(x,y)\,\psi_S(x)\,\d x\,\d y.
\]
As $(\cdot-c(S))^\gamma\omega^\infty_Q$ is $\ell(S)^{|\gamma|}|S|$ times an element of $C\Phi^{\kappa+\delta,1}(S)$ supported in $2Q$, \eqref{e:wbp} with $j=\kappa$ bounds the first term by a constant multiple of $\ell(S)^{|\gamma|-\kappa}$. In the second, $y\notin\frac32Q$, and the vanishing mean of $\psi_S$ and the H\"older bound in \eqref{e:czkernel} give \[\left|\int K(x,y)\psi_S(x)\,\d x\right|\lesssim\ell(S)^\delta|y-c(S)|^{-d-\kappa-\delta},\] whence the term is $\lesssim\ell(S)^{|\gamma|-\kappa}$, as $|\gamma|\leq\kappa$. This is \eqref{e:size}, cf.\ \cite{DGWdom}*{(5.7)}. If $\kappa=0$, the same argument, with the roles of the two arguments exchanged, bounds $b^{0,\star}_S$.

We next move to (ii). At each point and each sidelength, boundedly many $\psi_S$ do not vanish, and $|S||\psi_S|\lesssim1$. Thus, by \eqref{e:size}, the series
\[
g\coloneqq\sum_{\ell(S)>1}|S|b^{\alpha,\gamma}_S\psi_S
\]
converges absolutely and uniformly, with $\|g\|_\infty\lesssim1$, as $|\gamma|<\kappa$, and $\psi_S(g)=b^{\alpha,\gamma}_S\ind_{\{\ell(S)>1\}}$, by orthonormality and uniform convergence of the series. Hence a function $\mathfrak b\in X$ has the wavelet coefficients $b^{\alpha,\gamma}_S\ind_{\{\ell(S)\leq1\}}$ if and only if $\mathfrak b+g\in X$ has the wavelet coefficients $b^{\alpha,\gamma}_S$, and, as $\|g\|_{\mathrm{BMO}}\leq2\|g\|_\infty\lesssim1$, the norms of Definition \ref{d:symbols} differ by at most an admissible constant.
\end{proof}
\subsection{Paraproducts}\label{ss:para}
\begin{lemma}\label{l:para}
Let $b\in\mathrm{BMO}(\R^d)$.
\begin{enumerate}[label=(\roman*)]
\item Up to a constant multiple of $\|b\|_{\mathrm{BMO}}$, $P_b$ and $P^*_b$, as well as their restrictions to subcollections of $\mathcal D$, are model wavelet operators, for which \eqref{e:range} holds for every non-integer $0<\tau<k+1$.
\item If $b\in L^\infty(\R^d)$, then $P_b$ and its restrictions \[ f\mapsto \sum_{\ell(R)\leq\ell}|R|\psi_R(b)\vartheta_R(f )\psi_R,\qquad  \ell\in2^{\Z},\] satisfy \eqref{e:V1} up to a constant multiple of $\|b\|_\infty$.
\item Let $0\leq|\gamma|\leq\kappa\leq k$ and let $\beta=\{\beta_S\}$ be a sequence with $|\beta_S|\leq\ell(S)^{|\gamma|-\kappa}$ for all $S\in\mathcal D$. The operator $\Pi$ defined by $\langle\Pi f,g\rangle\coloneqq\Pi_\beta(\partial^\gamma f,g)$ is a constant multiple of a Calder\'on--Zygmund operator of order $(\kappa,\kappa)$, with
\[\langle\Pi[(\cdot-c(S))^\nu],\psi_S\rangle=\nu!\delta_{\nu\gamma}\beta_S, \qquad |\nu|\leq\kappa,\] and $\langle\Pi\psi_S,\mathbf 1\rangle=0$.
\item Let $\beta=\{\beta_S\}$ be a bounded sequence. The operator $\Pi^*$ defined by $\langle\Pi^*f,g\rangle\coloneqq\Pi_\beta(g,f)$ is a constant multiple of a Calder\'on--Zygmund operator of order $0$, and its actions are $\langle\Pi^*\mathbf 1,\psi_S\rangle=0$, $\langle\Pi^*\psi_S,\mathbf 1\rangle=\beta_S$. When $\beta_S=\psi_S(b)$, $\langle\Pi^*f,g\rangle=\langle f,P_bg\rangle$.
\end{enumerate}
\end{lemma}
\begin{proof}
Begin with (i). As $|\psi_Q(b)|\lesssim\|b\|_{\mathrm{BMO}}$, the operators \eqref{e:canpara} have the form \eqref{e:model}, with the packets $\psi_Q\in c\Phi^{k+1,\Subset}(Q)$ and $\psi_Q(b)\vartheta_Q\in C\|b\|_{\mathrm{BMO}}\Phi^{\sigma,\Subset}(Q)$ for every $\sigma>0$, whence \eqref{e:range} for $0<\tau<k+1$. By orthonormality,
\[
\|P_bh\|_2^2=\sum_{Q\in\mathcal D}|Q||\psi_Q(b)|^2|\vartheta_Q(h)|^2\lesssim\sum_{Q\in\mathcal D}|Q||\psi_Q(b)|^2\langle|h|\rangle_Q^2,
\]
and \eqref{e:L2} follows from the Carleson embedding theorem, as the wavelet coefficients of $b$ satisfy the Carleson condition
\[
\sum_{Q\in\mathcal D(R)}|Q||\psi_Q(b)|^2\lesssim\|b\|^2_{\mathrm{BMO}}|R|,\qquad R\in\mathcal D,
\]
see \cite{MC97}; the restrictions and the adjoints are handled in the same way.

 Next, move to (ii). Let $g\coloneqq\omega^\infty_Q$, so that $\|g\|_\infty+\ell(Q)\|\nabla g\|_\infty\lesssim1$, and fix $x$. The cubes $R$ with $\psi_R(x)\neq0$ satisfy $x\in\w R$, and are boundedly many at each sidelength. If $\ell(R)\geq\ell(Q)$, then $|\vartheta_R(g)|\lesssim|Q||R|^{-1}$; if $\ell(R)<\ell(Q)$, then $|\vartheta_R(g)-g(x)|\lesssim\ell(R)/\ell(Q)$, as $\vartheta_R$ is supported in $R$. These bounds contribute $\lesssim\|b\|_\infty$ to $P_bg(x)$, and what is left is
\[
g(x)\sum_{\ell(R)<\ell(Q)}|R|\psi_R(b)\psi_R(x)=g(x)\left[b-\mathbb E_{\ell'}b\right](x),\qquad \ell'\coloneqq\max\{q\in2^{\Z}:q<\ell(Q)\},
\]
by \eqref{e:mra}, which is bounded by a constant multiple of $\|b\|_\infty$. The restrictions are handled in the same way, with $\min\{\ell,\ell'\}$ in place of $\ell'$.

For (iii),  integrating by parts,
\[
\langle\Pi f,g\rangle=\sum_{S\in\mathcal D}|S|\,\ell(S)^{-|\gamma|}\beta_S\left((-1)^{|\gamma|}\ell(S)^{|\gamma|}\partial^\gamma\vartheta_S\right)(f)\,\psi_S(g)
\]
is a paraproduct form with the averaging functions $(-1)^{|\gamma|}\ell(S)^{|\gamma|}\partial^\gamma\vartheta_S\in C\Phi^{\sigma,\Subset}(S)$, $\sigma>0$, and coefficients bounded by $\ell(S)^{-\kappa}$. As these averaging functions are supported in $S$, the Cauchy--Schwarz inequality over the cubes of each sidelength $t$, and $t^{-\kappa}\leq(1+t^{-1})^\kappa$, give $|\langle\Pi f,g\rangle|\lesssim\|f\|_2\|g\|_{\mathcal U_\kappa}$, so that $\Pi$ is continuous from $\mathcal U_0$ to $\mathcal U_\kappa'$. The standard computations of \cite{DGWdom}*{Remark 5.3}, which use only the size, smoothness and support of the averaging functions, show that $\Pi$ is a constant multiple of a Calder\'on--Zygmund operator of order $(\kappa,\kappa)$. For the actions, let $f_A\coloneqq(\cdot-c(S))^\nu\omega^\infty_{AS}$. By orthonormality and \eqref{e:thetamom}, for $A$ large,
\[
\langle\Pi f_A,\psi_S\rangle=\beta_S\,\vartheta_S\left(\partial^\gamma(\cdot-c(S))^\nu\right)=\nu!\,\delta_{\nu\gamma}\beta_S,
\]
as $\partial^\gamma(\cdot-c(S))^\nu$ equals $\frac{\nu!}{(\nu-\gamma)!}(\cdot-c(S))^{\nu-\gamma}$ if $\gamma\leq\nu$, and zero otherwise. Next,
\[
\langle\Pi\psi_S,\omega^\infty_{AS}\rangle=\sum_{Q\in\mathcal D}|Q|\beta_Q\,\vartheta_Q\left(\partial^\gamma\psi_S\right)\psi_Q\left(\omega^\infty_{AS}\right),
\]
where a nonzero term needs $Q\cap\w S\neq\varnothing$ and, $\psi_Q$ having vanishing mean, $\w Q\not\subseteq\frac32AS$, which forces $\ell(Q)\gtrsim A\ell(S)$ for $A$ large. For such $Q$, the vanishing mean of $\psi_S$ gives $|\vartheta_Q(\partial^\gamma\psi_S)|=|(\partial^\gamma\vartheta_Q)(\psi_S)|\lesssim|Q|^{-1}\ell(Q)^{-|\gamma|-1}\ell(S)$, while $|\psi_Q(\omega^\infty_{AS})|\lesssim(A\ell(S))^d|Q|^{-1}$. As $|\beta_Q|\leq\ell(Q)^{|\gamma|-\kappa}$, and boundedly many $Q$ of each sidelength contribute, the sum is $\lesssim\ell(S)(A\ell(S))^{-\kappa-1}$, which tends to zero.

Concerning (iv),  the same computations apply, with $\kappa=\gamma=0$ and the roles of the two arguments exchanged, and give $\langle\Pi^*\psi_S,\omega^\infty_{AS}\rangle=\beta_S\vartheta_S(\omega^\infty_{AS})=\beta_S$ for $A$ large. The last claim is \eqref{e:canpara}.
\end{proof}

\subsection{Proof of Proposition \ref{p:rep}}\label{ss:rep}
For $R,S\in\mathcal D$ and $\delta_2>0$, let
\[
D(R,S)\coloneqq\max\left\{\ell(R),\ell(S),|c(R)-c(S)|\right\},\qquad (R,S)_{\delta_2}\coloneqq\frac{\min\{\ell(R),\ell(S)\}^{\delta_2}}{D(R,S)^{d+\delta_2}}.
\]
\begin{proof}
By Lemma \ref{l:size}(i), Lemma \ref{l:para}(iii) applies to the symbols $b^{\alpha,\gamma}$, $|\gamma|\leq\kappa$, and Lemma \ref{l:para}(iv) to $b^{0,\star}$ if $\kappa=0$. Let $\mathfrak L$ be the operator with
\[
\langle\mathfrak Lf,g\rangle\coloneqq\langle\partial^\alpha Tf,g\rangle-\sum_{|\gamma|\leq\kappa}\frac1{\gamma!}\Pi_{b^{\alpha,\gamma}}\left(\partial^\gamma f,g\right)-\ind_{\{\kappa=0\}}\Pi_{b^{0,\star}}(g,f).
\]
Then $\mathfrak L$ is a constant multiple of a Calder\'on--Zygmund operator of order $(\kappa,\kappa)$, and, by \eqref{e:symbols}, Lemma \ref{l:para}(iii)--(iv) and Lemma \ref{l:reduce}(iii),
\begin{equation}\label{e:balanced}
\left\langle\mathfrak L\left[(\cdot-c(S))^\nu\right],\psi_S\right\rangle=0,\quad |\nu|\leq\kappa,\qquad \langle\mathfrak L\psi_S,\mathbf 1\rangle=0,\qquad S\in\mathcal D.
\end{equation}
For every $0<\delta_2<\delta$, the almost diagonal estimate of \cite{DGWdom}*{Lemma 5.7} gives
\begin{equation}\label{e:matrix}
\ell(S)^\kappa\left|\langle\mathfrak L\psi_S,\psi_R\rangle\right|\lesssim(R,S)_{\delta_2},\qquad R,S\in\mathcal D.
\end{equation}
Indeed, if $\ell(S)\geq\ell(R)$, use \cite{DGWdom}*{(5.6)} with $m=\kappa$ and $k=\kappa$. Its Taylor subtraction, made when the cubes are close, does not change the pairing by \eqref{e:balanced}; hence, as $\ell(S)\leq D(R,S)$,
\[
\left|\langle\mathfrak L\psi_S,\psi_R\rangle\right|=\left|\left\langle\mathfrak L\left(\psi_S-\PM^{\kappa+1}_R\psi_S\right),\psi_R\right\rangle\right|\lesssim\frac{\ell(R)^{\delta_2}}{D(R,S)^{d+\kappa+\delta_2}}\leq\frac{(R,S)_{\delta_2}}{\ell(S)^\kappa}.
\]
If $\ell(S)<\ell(R)$, then $\PM^1_S\psi_R$ is a constant, and \cite{DGWdom}*{(5.6)} with $m=0$ gives
\[
\left|\langle\mathfrak L\psi_S,\psi_R\rangle\right|=\left|\left\langle\mathfrak L\psi_S,\psi_R-\PM^1_S\psi_R\right\rangle\right|\lesssim\frac{\ell(S)^{\delta_2-\kappa}}{D(R,S)^{d+\delta_2}}=\frac{(R,S)_{\delta_2}}{\ell(S)^\kappa}.
\]

Now let $\delta_1$ be as in the statement, and fix $\delta_1<\delta_2<\delta$. If $\kappa\geq1$, as $\psi_S$ annihilates $\P^\kappa$, Lemma \ref{l:anti}, applied on the cube $\w S$ with smoothness index $\delta_2$, provides $\psi_{S,\beta}\in C\Psi^{\delta_2,\Subset}(S)$, $|\beta|=\kappa$, such that
\begin{equation}\label{e:antiwave}
\psi_S(f)=(-1)^\kappa\ell(S)^\kappa\sum_{|\beta|=\kappa}\psi_{S,\beta}\left(\partial^\beta f\right),\qquad f\in \mathcal{C}^\infty_c(\R^d),
\end{equation}
and \eqref{e:antiwave} holds with $\psi_{S,0}\coloneqq\psi_S$ if $\kappa=0$. Define
\[
\varphi^{\alpha,\beta}_R\coloneqq(-1)^\kappa\sum_{S\in\mathcal D}|S|\ell(S)^\kappa\langle\mathfrak L\psi_S,\psi_R\rangle\psi_{S,\beta}.
\]
By \eqref{e:matrix} and \cite{DGWdom}*{Lemma 2.22}, $\varphi^{\alpha,\beta}_R\in C\Psi^{\delta_1,0}(R)$. Finally, let $f,g\in \mathcal{C}^\infty_c(\R^d)$. By \eqref{e:matrix} and the decay of the wavelet coefficients of $f$ and $g$ at both ends of the scale range, the expansion
\[
\langle\mathfrak Lf,g\rangle=\sum_{S,R\in\mathcal D}|S||R|\langle\mathfrak L\psi_S,\psi_R\rangle\psi_S(f)\psi_R(g)
\]
converges absolutely, and inserting \eqref{e:antiwave} and summing first over $S$ give
\[
\langle\mathfrak Lf,g\rangle=\sum_{R\in\mathcal D}|R|\psi_R(g)\sum_{|\beta|=\kappa}\varphi^{\alpha,\beta}_R\left(\partial^\beta f\right)=\sum_{|\beta|=\kappa}\left\langle V_{\alpha,\beta}\left(\partial^\beta f\right),g\right\rangle,
\]
which is \eqref{e:rep}.
\end{proof}

\subsection{Models for the derivative channels}\label{ss:channels}
By Lemma \ref{l:po}(iii), with $\pi=0$, for every cube $S$, every integer $n\geq0$ and every locally integrable $g$, scalar or vector valued,
\begin{equation}\label{e:jetcrude}
\nat{g}_{n,S}\lesssim\frac1r\left\langle g \right\rangle_{3S},
\end{equation}
while, by Lemma \ref{l:po}(i), $\nat{\nabla^jf}_{m-j,S}=0$ for $0\leq j\leq m$ if $f$ coincides on $3S$ with a polynomial of degree less than $m$.

Let $T$ be a Calder\'on--Zygmund operator of order $k$, $1\leq\kappa\leq k$, $|\alpha|=\kappa$ and $0<\tau<\delta$, and assume that the lower symbols belong to $\mathrm{BMO}(\R^d)$ and that the top symbols belong locally to $\mathrm{BMO}(\R^d)$. If $|\gamma|<\kappa$, then $\Pi_{b^{\alpha,\gamma}}(h,g)=\langle P_{\mathfrak b^{\alpha,\gamma}}h,g\rangle$, by \eqref{e:canpara}. If $|\gamma|=\kappa$, the wavelet coefficients of the function $\mathfrak b^{\alpha,\gamma}$ of the local condition are $b^{\alpha,\gamma}_S\ind_{\{\ell(S)\leq1\}}$, so that
\[
\Pi_{b^{\alpha,\gamma}}(\partial^\gamma f,g)=\left\langle P_{\mathfrak b^{\alpha,\gamma}}(\partial^\gamma f)+Y_{\alpha,\gamma}f,g\right\rangle,
\]
where
\begin{equation}\label{e:Ydef}
Y_{\alpha,\gamma}f\coloneqq\sum_{\ell(S)>1}|S|b^{\alpha,\gamma}_S\vartheta_S\left(\partial^\gamma f\right)\psi_S=(-1)^\kappa\sum_{\ell(S)>1}|S|\ell(S)^{-\kappa}b^{\alpha,\gamma}_S\big(\ell(S)^\kappa\partial^\gamma\vartheta_S\big)(f)\psi_S,
\end{equation}
the second expression by integration by parts. Thus Proposition \ref{p:rep}, with $\tau<\delta_1<\delta$, gives, for $f\in\mathcal{C}^\infty_c(\R^d)$,
\begin{equation}\label{e:channels}
\partial^\alpha Tf=\sum_{|\beta|=\kappa}V_{\alpha,\beta}(\partial^\beta f)+\sum_{|\gamma|\leq\kappa}\frac1{\gamma!}P_{\mathfrak b^{\alpha,\gamma}}(\partial^\gamma f)+\sum_{|\gamma|=\kappa}\frac1{\gamma!}Y_{\alpha,\gamma}f.
\end{equation}
Up to constants $\lesssim1+\mathfrak B(T)$, the operators in \eqref{e:channels} are models for which \eqref{e:range} holds: $V_{\alpha,\beta}$ by Proposition \ref{p:rep}, the paraproducts by Lemma \ref{l:para}(i), and $Y_{\alpha,\gamma}$ by Remark \ref{r:models}(a), as the coefficients in the last expression of \eqref{e:Ydef} are bounded, by \eqref{e:size}, and both of its packet families have vanishing mean. In particular, all terms of \eqref{e:channels} belong to $L^2(\R^d)$. Moreover, $V_{\alpha,\beta}$ and $Y_{\alpha,\gamma}$ satisfy \eqref{e:V1}, by Lemma \ref{l:V1mean}, and so does $P_{\mathfrak b^{\alpha,\gamma}}$ if $\mathfrak b^{\alpha,\gamma}\in L^\infty(\R^d)$, that is, if $b^{\alpha,\gamma}$ belongs to $L^\infty(\R^d)$ when $|\gamma|<\kappa$, and locally when $|\gamma|=\kappa$, by Lemma \ref{l:para}(ii).
\begin{lemma}\label{l:channelmodels}
Under the above assumptions, let $m\geq\kappa$ be an integer and $f\in\mathcal{C}^\infty_c(\R^d)$. If the conditions \emph{(i)} of Theorem \ref{t:jet} hold, there exists a sparse collection $\mathcal S$ such that, with $\nabla^0f\coloneqq f$,
\begin{equation}\label{e:jetfull}
\left|\partial^\alpha Tf\right|\lesssim\sum_{S\in\mathcal S}\sum_{j=0}^\kappa\nat{\nabla^jf}_{m-j,S}\ind_S\qquad\textrm{a.e.\ on }\R^d,
\end{equation}
with constant $\lesssim1+\mathfrak B(T)$. Under the $\mathrm{BMO}$ assumptions alone, the same bound holds with $\nat{\nabla^jf}_{0,S}$ in place of $\nat{\nabla^jf}_{m-j,S}$.
\end{lemma}
\begin{proof}
Under the conditions of the first claim, every operator of \eqref{e:channels} acting at a positive order satisfies \eqref{e:V1}, by the above: $V_{\alpha,\beta}$ and $Y_{\alpha,\gamma}$ always, and $P_{\mathfrak b^{\alpha,\gamma}}$, which acts on $\partial^\gamma f$ at the order $m-|\gamma|$, as this order is positive only if $|\gamma|<\kappa$ or $m>\kappa$, and then $\mathfrak b^{\alpha,\gamma}\in L^\infty(\R^d)$. Thus Theorem \ref{t:model}, applied to the terms of \eqref{e:channels} with the input $\partial^\gamma f$ at the order $m-|\gamma|$, $|\gamma|\leq\kappa$, and to the terms $Y_{\alpha,\gamma}f$ with the input $f$ at the order $m$, gives \eqref{e:jetfull}, the union of the resulting collections being sparse. At order zero, Theorem \ref{t:model} requires no cutoff testing condition, which gives the last claim.
\end{proof}

\subsection{Necessary testing conditions}\label{ss:tests}
The following $L^2$ tests give the local and global $\mathrm{BMO}$ conditions.
\begin{lemma}\label{l:carleson}
Let $T$ be a Calder\'on--Zygmund operator of order $k$, $0\leq\kappa\leq k$ and $|\alpha|=\kappa$.
\begin{enumerate}[label=(\roman*)]
\item Let $\kappa\geq1$, and assume that $\|\partial^\alpha Tf\|_{L^2(\R^d)}\leq C\|f\|_{W^{\kappa,2}(\R^d)}$ for all $f\in \mathcal{C}^\infty_c(\R^d)$. Then, for $|\gamma|\leq\kappa$ and $Q\in\mathcal D$ with $\ell(Q)\leq1$,
\[
\sum_{S\in\mathcal D,\,S\subseteq Q}|S|\left|b^{\alpha,\gamma}_S\right|^2\lesssim(1+C)^2\ell(Q)^{2(|\gamma|-\kappa)}|Q| .
\]
In particular, the top symbols belong locally to $\mathrm{BMO}(\R^d)$, with local norm $\lesssim1+C$.
\item If the lower symbols vanish and $\|\partial^\alpha Tf\|_{L^2(\R^d)}\leq C\|\nabla^\kappa f\|_{L^2(\R^d)}$ for all $f\in \mathcal{C}^\infty_c(\R^d)$, then the top symbols, and $b^{0,\star}$ if $\kappa=0$, belong to $\mathrm{BMO}(\R^d)$, with norms $\lesssim1+C$.
\end{enumerate}
\end{lemma}
\begin{proof}
We first address (i). Let $Q\in\mathcal D$ with $\ell(Q)\leq1$, and \[f_Q\coloneqq(\cdot-c(Q))^\gamma\omega^\infty_{4\sqrt d\w Q}.\] As $\ell(Q)\leq1$, \[\|\partial^\mu f_Q\|_2\lesssim\ell(Q)^{|\gamma|-|\mu|}|Q|^{\frac12}\leq\ell(Q)^{|\gamma|-\kappa}|Q|^{\frac12}, \qquad |\mu|\leq\kappa,\] so that $\|\partial^\alpha Tf_Q\|_2\lesssim C\ell(Q)^{|\gamma|-\kappa}|Q|^{\frac12}$. Let $S\in\mathcal D$ with $S\subseteq Q$. By \eqref{e:actionsplit} with $P=4\sqrt d\,\w Q$, the vanishing mean of $\psi_S$ and the H\"older bound in \eqref{e:czkernel} give
\[
\left\langle\partial^\alpha T\left[(\cdot-c(Q))^\gamma\right],\psi_S\right\rangle=\psi_S\left(\partial^\alpha Tf_Q\right)+O\left(\ell(Q)^{|\gamma|-\kappa}\left(\frac{\ell(S)}{\ell(Q)}\right)^\delta\right),
\]
while, by the binomial formula and \eqref{e:symbols}, the left hand side equals
\[
b^{\alpha,\gamma}_S+\sum_{\gamma'<\gamma}\binom{\gamma}{\gamma'}\left(c(S)-c(Q)\right)^{\gamma-\gamma'}b^{\alpha,\gamma'}_S,\qquad |c(S)-c(Q)|_\infty\leq\ell(Q).
\]
By \eqref{e:onb}, the quantities
\[
\mathcal C_\gamma(Q)\coloneqq\bigg(\frac1{|Q|}\sum_{S\in\mathcal D,\,S\subseteq Q}|S|\left|b^{\alpha,\gamma}_S\right|^2\bigg)^{\frac12}
\]
satisfy the inductive estimate
\[
\mathcal C_\gamma(Q)\lesssim(1+C)\,\ell(Q)^{|\gamma|-\kappa}+\sum_{\gamma'<\gamma}\ell(Q)^{|\gamma|-|\gamma'|}\,\mathcal C_{\gamma'}(Q),
\]
and induction on $|\gamma|$ proves the claim. For $|\gamma|=\kappa$, decomposing a cube $Q\in\mathcal D$ with $\ell(Q)>1$ into its subcubes of unit sidelength, the sequence $\{b^{\alpha,\gamma}_S\ind_{\{\ell(S)\leq1\}}\}$ satisfies the Carleson condition on every $Q\in\mathcal D$, with constant $\lesssim(1+C)^2$, which characterizes the wavelet coefficients of $\mathrm{BMO}(\R^d)$ functions \cite{MC97}.

To prove (ii), let $|\gamma|=\kappa$ and $Q\in\mathcal D$. The third and fourth displays of the proof of (i) hold without the restriction $\ell(Q)\leq1$, and, as the lower symbols vanish, the sum over $\gamma'<\gamma$ in the second one vanishes, while $\|\partial^\alpha Tf_Q\|_2\leq C\|\nabla^\kappa f_Q\|_2\lesssim C|Q|^{\frac12}$. Thus $\mathcal C_\gamma(Q)\lesssim1+C$ for every $Q\in\mathcal D$, and $b^{\alpha,\gamma}$ belongs to $\mathrm{BMO}(\R^d)$, with norm $\lesssim1+C$ \cite{MC97}. If $\kappa=0$, the same argument applies to the $L^2$-bounded adjoint $T^*$, whose symbol $b^{0,0}$ is $b^{0,\star}$.
\end{proof}
For the $L^\infty$ conditions, we use weak-$*$ limits of finite scale sums.
\begin{lemma}\label{l:recon}
Let $b=\{b_S:S\in\mathcal D\}$ be a symbol with $\sup_{\ell(S)=t}|b_S|<\infty$ for each $t\in2^{\Z}$, let $F_1\subseteq F_2\subseteq\cdots$ be finite subsets of $2^{\Z}$ with union $\mathcal T$, and let $A\geq0$. Assume that, for each $n$, the function
\[
b_n\coloneqq\sum_{S\in\mathcal D,\,\ell(S)\in F_n}|S|b_S\psi_S
\]
satisfies $|\langle b_n,g\rangle|\leq A\|g\|_{L^1(\R^d)}$ for all $g\in \mathcal{C}^\infty_c(\R^d)$ supported in a cube of sidelength two. Then there exists $\mathfrak b\in L^\infty(\R^d)$ with $\|\mathfrak b\|_{L^\infty(\R^d)}\leq A$ and $\psi_S(\mathfrak b)=b_S\ind_{\mathcal T}(\ell(S))$ for all $S\in\mathcal D$.
\end{lemma}
\begin{proof}
At each point and each sidelength, boundedly many $\psi_S$ do not vanish, so that $b_n$ is a bounded function, and $\|b_n\|_{L^\infty(\R^d)}\leq A$ by duality. As $L^1(\R^d)$ is separable, a subsequence of $\{b_n\}$ converges weakly-$*$ in $L^\infty(\R^d)$ to a function $\mathfrak b$ with $\|\mathfrak b\|_{L^\infty(\R^d)}\leq A$. As $\psi_S\in L^1(\R^d)$, and as only finitely many terms of $b_n$ do not vanish on the support of $\psi_S$, orthonormality gives $\psi_S(\mathfrak b)=\lim_n\psi_S(b_n)=b_S\ind_{\mathcal T}(\ell(S))$.
\end{proof}

\begin{lemma}\label{l:tests}
Let $T$ be a Calder\'on--Zygmund operator of order $k$, $0\leq\kappa\leq k$, $|\alpha|=\kappa$, $m\geq\kappa$ and $H\geq0$.
\begin{enumerate}[label=(\alph*)]
\item Let $\kappa\geq1$. Fix a sufficiently large admissible $A_0$, let $Q_c,Q'_c$ have center $c$ and sidelengths $A_0,A_0^{1/2}$, and set $f_c^\gamma=(\cdot-c)^\gamma\omega^\infty_{Q_c}$. If the images $\partial^\alpha Tf_c^\gamma$ are locally integrable and
\[
|\partial^\alpha Tf_c^\gamma|\leq H\quad\textrm{a.e.\ on }Q'_c,
\qquad c\in\R^d,\quad |\gamma|\leq\kappa,\quad |\gamma|<m,
\]
then the lower symbols belong to $L^\infty(\R^d)$ and, if $m>\kappa$, the top symbols belong locally to $L^\infty(\R^d)$, with norms $\lesssim1+H$.
\item Suppose that for every cube $Q$ of center $c$, every $A\geq6$, and $|\gamma|\leq\kappa$ with $|\gamma|<m$, the image of $f_A=(\cdot-c)^\gamma\omega^\infty_{AQ}$ is locally integrable and
\begin{equation}\label{e:homtest}
|\partial^\alpha Tf_A|\leq H(A\ell(Q))^{|\gamma|-\kappa}\quad\textrm{a.e.\ on }Q.
\end{equation}
Then the lower symbols vanish, and, if $m>\kappa$, the top symbols belong to $L^\infty(\R^d)$ with norms $\lesssim H$.
\end{enumerate}
\end{lemma}
\begin{proof}
The first task is (a). Let $|\gamma|<m$, and let $h$ be a finite linear combination of wavelets $\psi_S$ with $\ell(S)\leq1$, supported in $Q'_c$. Then $|\langle\partial^\alpha T[(\cdot-c)^\gamma],h\rangle|\lesssim(1+H)\|h\|_1$. Indeed, as $\int h=0$ and $h$ is supported in $Q_c$, \eqref{e:actionsplit} with $P=Q_c$ gives
\[
\left\langle\partial^\alpha T\left[(\cdot-c)^\gamma\right],h\right\rangle=\left\langle\partial^\alpha Tf^\gamma_c,h\right\rangle+\int_{\R^d}(y-c)^\gamma\left[1-\omega^\infty_{Q_c}\right](y)\int_{\R^d}K(x,y)\,h(x)\,\d x\,\d y,
\]
with $K$ the kernel of $\partial^\alpha T$. The first term is $\lesssim H\|h\|_1$ by the hypothesis. In the second, $y\notin\frac32Q_c$, and the vanishing mean of $h$ and the H\"older bound in \eqref{e:czkernel} give
\[\left|\int K(x,y)h(x)\,\d x\right|\lesssim\|h\|_1|y-c|^{-d-\kappa-\delta},\]
 so that it is $\lesssim\|h\|_1$, as $|\gamma|\leq\kappa$.
For $|\gamma|<m$ and $n\geq0$, set
\[
b^{(n)}\coloneqq\sum_{2^{-n}\leq\ell(S)\leq1}|S|b^{\alpha,\gamma}_S\psi_S .
\]
By \eqref{e:size} and Lemma \ref{l:recon}, the local conclusion follows once $|\langle b^{(n)},g\rangle|\lesssim(1+H)\|g\|_1$, uniformly in $n$, for $g\in \mathcal{C}^\infty_c(\R^d)$ supported in a cube of sidelength two with center $c$. Lemma \ref{l:size}(ii) then gives the global conclusion for the lower symbols. The binomial formula and \eqref{e:symbols} give
\[
\left\langle b^{(n)},g\right\rangle=\sum_{\gamma'\leq\gamma}\binom{\gamma}{\gamma'}\left\langle\partial^\alpha T\left[(\cdot-c)^{\gamma'}\right],g_{\gamma'}\right\rangle,\qquad g_{\gamma'}\coloneqq\sum_{2^{-n}\leq\ell(S)\leq1}|S|\,(c-c(S))^{\gamma-\gamma'}\psi_S(g)\,\psi_S,
\]
where each $g_{\gamma'}$ is a finite combination of wavelets of sidelength at most one, supported in $Q'_c$. By the preceding bound, it suffices to show that $\|g_{\gamma'}\|_1\lesssim\|g\|_1$. Let $\mu\coloneqq\gamma-\gamma'$ and $e_S(y)\coloneqq(c-c(S))^\mu-(c-y)^\mu$. Only the cubes $S$ for which $\w S$ meets the support of $g$ contribute below; for them, $|c-c(S)|\lesssim1$, and thus $|e_S|\lesssim\ell(S)$ on $\w S$, which contains the support of $\psi_S$. By \eqref{e:mra},
\[
g_{\gamma'}=\left[\mathbb E_{2^{-n-1}}-\mathbb E_1\right]\left((c-\cdot)^\mu g\right)+\sum_{2^{-n}\leq\ell(S)\leq1}|S|\,\psi_S\left(e_Sg\right)\psi_S .
\]
The first term has $L^1$ norm $\lesssim\|g\|_1$, as $\mathbb E_\ell$ is bounded on $L^1(\R^d)$, and the second $\lesssim\|g\|_1$, as the cubes $S$ of each sidelength $t\leq1$ contribute $\lesssim t\|g\|_1$, the cubes $\w S$ having bounded overlap.

We switch to (b). If $|\gamma|<\kappa$ and $S\in\mathcal D$ has $c(S)=c$ and $\w S\subseteq Q$, then, by \eqref{e:polyaction}, Lemma \ref{l:reduce}(ii) and \eqref{e:homtest},
\[
\left|b^{\alpha,\gamma}_S\right|=\lim_{A\to\infty}\left|\left\langle\partial^\alpha Tf_A,\psi_S\right\rangle\right|\leq\lim_{A\to\infty}H\left(A\ell(Q)\right)^{|\gamma|-\kappa}\|\psi_S\|_1=0,
\]
so that the lower symbols vanish. Consequently, the polynomials of degree less than $\kappa$ are annihilated by the actions of $\partial^\alpha T$ on the wavelets, and, for $|\gamma|=\kappa$,
\begin{equation}\label{e:center}
b^{\alpha,\gamma}_S=\left\langle\partial^\alpha T\left[(\cdot-c)^\gamma\right],\psi_S\right\rangle,\qquad c\in\R^d,\quad S\in\mathcal D.
\end{equation}
Let $|\gamma|=\kappa<m$. Then \eqref{e:homtest} bounds $|\partial^\alpha Tf_A|$ by $H$ on $Q$, uniformly in $A$, so that, for every finite linear combination $h$ of wavelets supported in $Q$, where $Q$ is any cube and $c$ its center,
\[
\left|\left\langle\partial^\alpha T\left[(\cdot-c)^\gamma\right],h\right\rangle\right|=\lim_{A\to\infty}\left|\left\langle\partial^\alpha Tf_A,h\right\rangle\right|\leq H\|h\|_1.
\] For $n\geq1$, let
\[
b^{(n)}\coloneqq\sum_{2^{-n}\leq\ell(S)\leq2^n}|S|b^{\alpha,\gamma}_S\psi_S,
\]
a bounded function by \eqref{e:size}. If $g\in \mathcal{C}^\infty_c(\R^d)$, then $g_n\coloneqq[\mathbb E_{2^{-n-1}}-\mathbb E_{2^n}]g$ is a finite linear combination of wavelets with $\|g_n\|_1\lesssim\|g\|_1$, and, choosing $Q$ containing its support, \eqref{e:center} gives
\[
\left|\left\langle b^{(n)},g\right\rangle\right|=\left|\left\langle\partial^\alpha T\left[(\cdot-c)^\gamma\right],g_n\right\rangle\right|\leq H\|g_n\|_1\lesssim H\|g\|_1.
\]
Thus Lemma \ref{l:recon}, with $F_n\coloneqq\{t\in2^{\Z}:2^{-n}\leq t\leq2^n\}$, provides $\mathfrak b^{\alpha,\gamma}\in L^\infty(\R^d)$ with $\|\mathfrak b^{\alpha,\gamma}\|_\infty\lesssim H$ and $\psi_S(\mathfrak b^{\alpha,\gamma})=b^{\alpha,\gamma}_S$ for all $S\in\mathcal D$.
\end{proof}

\subsection{Proof that (i) and (ii) are equivalent in Theorem \ref{t:HS}}\label{ss:czthm}
\begin{proof}
First we show (i) implies (ii). Assume (i), and let $f\in \mathcal{C}^\infty_c(\R^d)$. As the lower symbols vanish, Proposition \ref{p:rep}, with $\tau<\delta_1<\delta$, writes $\partial^\alpha Tf$ as a sum of boundedly many terms $T_i(\partial^{\beta_i}f)$, $|\beta_i|=\kappa$, each belonging to $L^2(\R^d)$, so that $\partial^\alpha Tf$ is locally integrable. Here $T_i$ is one of the operators $V_{\alpha,\beta}$, $\gamma!^{-1}P_{\mathfrak b^{\alpha,\gamma}}$ with $|\gamma|=\kappa$, and $P^*_{\mathfrak b^{0,\star}}$. Up to a normalization, each $T_i$ is a model for which $\tau$ satisfies \eqref{e:range}, by Remark \ref{r:models} and Lemma \ref{l:para}(i), and, if $m>\kappa$, satisfies \eqref{e:V1}: $V_{\alpha,\beta}$ and $P^*_{\mathfrak b^{0,\star}}$ by Lemma \ref{l:V1mean}, as their input packets have vanishing mean, and $P_{\mathfrak b^{\alpha,\gamma}}$ by Lemma \ref{l:para}(ii), as $\mathfrak b^{\alpha,\gamma}\in L^\infty(\R^d)$. Theorem \ref{t:model} of order $m-\kappa$, applied to each $T_i$ with the input $\partial^{\beta_i}f$, and the union of the resulting collections give \eqref{e:Dfine}, with constant $\lesssim1+\mathfrak B(T)$.

To see that (ii) implies (i),  assume (ii), with implied constant $C$. By \eqref{e:jetcrude}, the right hand side of \eqref{e:Dfine} for $f$ is at most a constant multiple of
\[
\frac Cr\sum_{S\in\mathcal S}\langle|\nabla^\kappa f|\rangle_{3S}\ind_S.
\]
Fix $c\in\R^d$, a cube $Q$ of center $c$, and $|\gamma|\leq\kappa$ with $|\gamma|<m$. For $A\geq6$, let $f_A\coloneqq(\cdot-c)^\gamma\omega^\infty_{AQ}$, and let $\mathcal S$ be the collection of (ii) for $f_A$. On $\frac32AQ$, $\nabla^\kappa f_A$ is a polynomial of degree less than $m-\kappa$, namely zero or a constant, and $|\nabla^\kappa f_A|\lesssim(A\ell(Q))^{|\gamma|-\kappa}$ on its support $2AQ$. Let $x\in Q$. A cube $S\in\mathcal S$ containing $x$ contributes to \eqref{e:Dfine} only if $3S\not\subseteq\frac32AQ$, whence $\ell(S)\geq\frac{A\ell(Q)}6$, and then its term is $\lesssim\frac Cr(A\ell(Q))^{|\gamma|-\kappa+d}\ell(S)^{-d}$. By \eqref{e:band},
\begin{equation}\label{e:necbound}
\left|\partial^\alpha Tf_A\right|\lesssim\frac Cr\left(A\ell(Q)\right)^{|\gamma|-\kappa}\qquad\textrm{almost everywhere on }Q.
\end{equation}
Lemma \ref{l:tests}(b), applied to \eqref{e:necbound}, gives the vanishing lower symbols and, when $m>\kappa$, the bounded top symbols.

By the $L^2$ bound for sparse operators \cite{LN19} and \eqref{e:jetcrude}, $\|\partial^\alpha Tf\|_2\lesssim (C/r)\|\nabla^\kappa f\|_2$. Lemma \ref{l:carleson}(ii) gives the top $\mathrm{BMO}$ conditions when $m=\kappa$ and the adjoint condition when $\kappa=0$, with norms $\lesssim1+C/r$.
\end{proof}

\subsection{Proof of Theorem \ref{t:jet} and Corollary \ref{c:sobjet}}\label{ss:jetthm}
\begin{proof}[Proof of Theorem \ref{t:jet}]
First we show that (i) implies (ii). Decrease $r_0$, as allowed in \S\ref{ss:perc}, so that $s\coloneqq2^{d+1}r$ is admissible in Lemmas \ref{l:sandwichperc} and \ref{l:channelmodels}. By \S\ref{ss:channels}, $\partial^\alpha Tf\in L^2(\R^d)$, and Lemma \ref{l:channelmodels} gives \eqref{e:jetfull} at the level $s$. As the percentiles do not increase with the level, Lemma \ref{l:sandwichperc}(i) bounds the terms of \eqref{e:jetfull} on the cubes $S$ with $\ell(S)<1$ by those of \eqref{e:jet} on $S$, and Lemma \ref{l:sandwichperc}(ii), at the level $s$, bounds those on the cubes with $\ell(S)\geq1$ by the terms of \eqref{e:jet} on $2S$, at the level $2^{-d-1}s=r$. Replacing the latter cubes by their doubles, which divides the sparseness constant by $2^d$, gives \eqref{e:jet}, with $C\lesssim1+\mathfrak B(T)$.

To see that (ii) implies (i), assume (ii), with implied constant $C$. By \eqref{e:jetcrude}, the right hand side of \eqref{e:jet} for $f$ is at most a constant multiple of
\[
\frac Cr\sum_{S\in\mathcal S}\left(\langle|f|\rangle_{3S}+\langle|\nabla^\kappa f|\rangle_{3S}\right)\ind_S.
\]
Let $A_0$, $Q_c$, $Q'_c$ and $f^\gamma_c$ be as in Lemma \ref{l:tests}(a), with $|\gamma|\leq\kappa$ and $|\gamma|<m$, and let $\mathcal S$ be the collection of (ii) for $f^\gamma_c$. Let $x\in Q'_c$. As $f^\gamma_c$ coincides on $\frac32Q_c$ with a polynomial of degree $|\gamma|<m$, a cube $S\in\mathcal S$ containing $x$ contributes only if $3S\not\subseteq\frac32Q_c$, whence $\ell(S)\gtrsim A_0$, and then its term is $\lesssim\frac Cr\ell(S)^{-d}$, as $\|f^\gamma_c\|_1+\|\nabla^\kappa f^\gamma_c\|_1\lesssim1$. By \eqref{e:band},
\[
\left|\partial^\alpha Tf^\gamma_c\right|\lesssim\frac Cr\qquad\textrm{almost everywhere on }Q'_c,
\]
and Lemma \ref{l:tests}(a) shows that the lower symbols belong to $L^\infty(\R^d)$ and, if $m>\kappa$, that the top symbols belong locally to $L^\infty(\R^d)$, with norms $\lesssim(1+C)/r$.

If $m=\kappa$, \eqref{e:jetcrude} and the $L^2$ bound for sparse operators \cite{LN19} give
\[
\|\partial^\alpha Tf\|_2\lesssim\frac Cr\|f\|_{W^{\kappa,2}(\R^d)},
\]
and Lemma \ref{l:carleson}(i) shows that the top symbols belong locally to $\mathrm{BMO}(\R^d)$, with local norms $\lesssim1+C/r$.
\end{proof}

\begin{proof}[Proof of Corollary \ref{c:sobjet}]
First we show that (i) implies (ii). Let $1\leq|\alpha|\leq k$. By \S\ref{ss:channels}, $\partial^\alpha Tf\in L^2(\R^d)$. Decreasing $r_0$ if necessary, Lemma \ref{l:channelmodels} gives \eqref{e:jetfull}, with $\kappa=|\alpha|$, at the level $2^{d+1}r$. As the percentiles do not increase with the level, Lemma \ref{l:sandwichperc}, with $k$ in place of $\kappa$, bounds its terms on the cubes $S$ with $\ell(S)<1$ by those of \eqref{e:sobjet} on $S$, and its terms on the cubes with $\ell(S)\geq1$ by those of \eqref{e:sobjet} on $2S$, at the level $r$. Replacing the latter cubes by their doubles, and taking the union of the resulting collections over $1\leq|\alpha|\leq k$, gives \eqref{e:sobjet}, with constant $\lesssim1+\mathfrak B(T)$.

Conversely, assume (ii), with implied constant $C$. If $|\alpha|=k$, \eqref{e:sobjet} contains \eqref{e:jet} with $\kappa=k$, and Theorem \ref{t:jet} gives the conditions (i) for $\alpha$. Let now $|\alpha|=\kappa<k$, and let $f^\gamma_c$ be as in Lemma \ref{l:tests}(a), with $|\gamma|\leq\kappa$. As $f^\gamma_c$ coincides on $\frac32Q_c$ with a polynomial of degree $|\gamma|<k\leq m$, the terms of \eqref{e:sobjet} for $f^\gamma_c$ vanish on the cubes $S$ with $3S\subseteq\frac32Q_c$, while those of the cubes $S$ containing a point of $Q'_c$ and with $3S\not\subseteq\frac32Q_c$ have $\ell(S)\gtrsim A_0$ and are $\lesssim\frac Cr\ell(S)^{-d}$, by \eqref{e:jetcrude}. By \eqref{e:band}, $|\partial^\alpha Tf^\gamma_c|\lesssim C/r$ almost everywhere on $Q'_c$, and Lemma \ref{l:tests}(a), as $m>\kappa$, gives the conditions (i) of Theorem \ref{t:jet} for $\alpha$, with norms $\lesssim(1+C)/r$.

Finally, if $T1\in L^\infty(\R^d)$ and $T^*1\in\mathrm{BMO}(\R^d)$, Theorem \ref{t:HS} with $\kappa=0$ bounds $|Tf|$ by a sparse sum of the terms $\nat{f}_{m,S}$, and the union of the two collections gives \eqref{e:sobjet} with $|Tf|$ added to its left hand side.
\end{proof}

\enlargethispage{2\baselineskip}\subsection{Proof of the Sobolev sparse bound}\label{ss:czcor}
\begin{proof}[Proof of \eqref{e:D}]
If $m>\kappa$, the second inequality in \eqref{e:poincareperc}, with $\partial^\beta f$ and $m-\kappa$ in place of $f$ and $m$, gives, for every cube $S$,
\begin{equation}\label{e:sobcube}
\nat{\partial^\beta f}_{m-\kappa,S}\lesssim\ell(S)^{m-\kappa}\left(\M^\tau_{3S}\left[\nabla^{m-\kappa}\partial^\beta f\right]\right)^{*,S}(r)\leq\ell(S)^{m-\kappa}\left(\M^\tau_{3S}\left[\nabla^mf\right]\right)^{*,S}(r),
\end{equation}
as $\M^\tau_{3S}[\nabla^{m-\kappa}\partial^\beta f]\leq\M^\tau_{3S}[\nabla^mf]$ by \eqref{e:msigmavec}; if $m=\kappa$, the same bound holds by \eqref{e:po}. Thus \eqref{e:Dfine} implies \eqref{e:D}. The bound \eqref{e:czMW} is proved in Section \ref{s:HS}.
\end{proof}
\subsection{Sharp $A_p$-weighted Sobolev bounds}\label{ss:Ap}
The order-zero sparse bounds give sharp $A_p$ estimates under $\mathrm{BMO}$ conditions on the symbols. Write
\begin{equation}\label{e:sobnorm}
\|f\|_{W^{k,p}(w)}\coloneqq\sum_{|\gamma|\leq k}\left\|\partial^\gamma f\right\|_{L^p(w)}.
\end{equation}
\begin{proposition}\label{c:czweighted}
Let $1<p<\infty$, $w\in A_p$ and $f\in \mathcal{C}^\infty_c(\R^d)$.
\begin{enumerate}[label=(\alph*)]
\item In the setting of Theorem \ref{t:HS} with $m=\kappa$, assume \emph{(i)}. Then
\[
\left\|\partial^\alpha Tf\right\|_{L^p(w)}\lesssim_p[w]_{A_p}^{\max\{1,\frac1{p-1}\}}\sum_{|\beta|=\kappa}\left\|\partial^\beta f\right\|_{L^p(w)}.
\]
\item In the setting of Theorem \ref{t:jet}, assume that the lower symbols belong to $\mathrm{BMO}(\R^d)$ and that the top symbols belong locally to $\mathrm{BMO}(\R^d)$. Then
\begin{equation}\label{e:czAp}
\left\|\partial^\alpha Tf\right\|_{L^p(w)}\lesssim_p[w]_{A_p}^{\max\{1,\frac1{p-1}\}}\,\|f\|_{W^{\kappa,p}(w)}.
\end{equation}
\item Let $T$ be a Calder\'on--Zygmund operator of order $k$. Assume that $T1$ and $T^*1$ belong to $\mathrm{BMO}(\R^d)$ and that, for every $1\leq\kappa\leq k$ and $|\alpha|=\kappa$, the lower symbols $b^{\alpha,\gamma}$, $|\gamma|<\kappa$, belong to $\mathrm{BMO}(\R^d)$, while the top symbols $b^{\alpha,\gamma}$, $|\gamma|=\kappa$, belong locally to $\mathrm{BMO}(\R^d)$. Then
\begin{equation}\label{e:sobolev}
\|Tf\|_{W^{k,p}(w)}\lesssim_p[w]_{A_p}^{\max\{1,\frac1{p-1}\}}\|f\|_{W^{k,p}(w)} .
\end{equation}
\end{enumerate}
\end{proposition}
\begin{proof}[Proof of Proposition \ref{c:czweighted}]
In cases (a) and (b), it suffices to find, at a fixed admissible level, a sparse collection $\mathcal S$ with
\begin{equation}\label{e:avgsparse}
\left|\partial^\alpha Tf\right|\lesssim\sum_{|\gamma|\leq\kappa}\sum_{S\in\mathcal S}\left\langle\left|\partial^\gamma f\right|\right\rangle_{3S}\ind_S,
\end{equation}
with the terms $|\gamma|=\kappa$ only in case (a): the sharp weighted bound for sparse operators \cite{LN19}, applied to the sparse collection $\{3S:S\in\mathcal S\}$, then gives the claims. In case (a), \eqref{e:avgsparse} follows from \eqref{e:Dfine} with $m=\kappa$ and \eqref{e:jetcrude}. In case (b), Lemma \ref{l:channelmodels} gives the order-zero sparse bound, and \eqref{e:jetcrude} gives \eqref{e:avgsparse}. In case (c), the assumptions are those of (a) with $\kappa=m=0$ for $T$, and those of (b) for $\partial^\alpha T$, $1\leq|\alpha|\leq k$, so that summing over $|\alpha|\leq k$ gives \eqref{e:sobolev}.
\end{proof}
The exponent in (c), and in (a) with $\kappa=0$, is optimal for the Riesz transforms $R_j$, whose $L^p(w)$ bounds are sharp \cites{petermichl07,petermichl08}. Up to constant multiples, $R_j$ is a Calder\'on--Zygmund operator of every order $k$. Since $\partial^\alpha R_jf=R_j\partial^\alpha f$, Corollary \ref{c:endpoint}(a) gives the conditions of Theorem \ref{t:HS} with $m=\kappa$ for every $0\leq\kappa\leq k$, and hence the assumptions of (a) and (c). For (c), apply \eqref{e:sobolev} to $f(\lambda\cdot)$ and $w(\lambda\cdot)$ and multiply by $\lambda^{d/p}$. Dilations commute with $R_j$ and preserve $[w]_{A_p}$; letting $\lambda\to0$ gives the $L^p(w)$ bound with the same exponent. The assumptions of (b) do not imply \eqref{e:czextreme}, even for $\kappa=m=1$, see Example \ref{r:DGW}.

\section{Weighted Hardy--Sobolev inequalities}\label{s:HS}
We complete the proof of Theorem \ref{t:HS} and its corollaries, then prove Theorem \ref{t:HSjet} and Proposition \ref{p:rigidity}. The logarithmic tests below give the necessary $L^\infty$ conditions in both settings. The notation is that of Section \ref{s:cz}.

\subsection{Homogeneous bounds and the order-zero conditions}\label{ss:order0}
\begin{proof}[Proof that (ii) implies (iii) in Theorem \ref{t:HS}]
Let $f\in \mathcal{C}^\infty_c(\R^d)$, let $\mathcal S$ be the collection in \eqref{e:Dfine}, and set
\[
\lambda_S\coloneqq\sum_{|\beta|=\kappa}\nat{\partial^\beta f}_{m-\kappa,S},\qquad
\sup_{S\in\mathcal S}\lambda_S\ind_S\leq\sum_{|\beta|=\kappa}\mathcal N_{m-\kappa}(\partial^\beta f).
\]
Lemma \ref{l:Ainf} and the quasi-triangle inequality give \eqref{e:PHS}. Under (i), the sparse constant is $\lesssim1+\mathfrak B(T)$, by \S\ref{ss:czthm}; hence the norm constants are $\lesssim_p1+\mathfrak B(T)$ and grow at most linearly in $p$.
\end{proof}

For the converse, unweighted estimates first give the vanishing lower symbols and the top $\mathrm{BMO}$ conditions.
\begin{lemma}\label{l:order0}
Assume that $\|\partial^\alpha Tf\|_{L^p(\R^d)}\leq C\|\nabla^\kappa f\|_{L^p(\R^d)}$ for all $f\in \mathcal{C}^\infty_c(\R^d)$, for $p=2$ and for some $d<p<\infty$. Then the symbols $b^{\alpha,\gamma}$ vanish for $|\gamma|<\kappa$ and belong to $\mathrm{BMO}(\R^d)$ for $|\gamma|=\kappa$, and, if $\kappa=0$, $b^{0,\star}\in\mathrm{BMO}(\R^d)$, all norms being $\lesssim1+C$.
\end{lemma}
\begin{proof}
Let $|\gamma|<\kappa$, $S\in\mathcal D$ and $f_A\coloneqq(\cdot-c(S))^\gamma\omega^\infty_{AS}$. The derivatives of order $\kappa$ of $f_A$ are supported in $2AS\setminus\frac32AS$ and bounded by a constant multiple of $(A\ell(S))^{|\gamma|-\kappa}$, so that, with $d<p<\infty$ as in the assumption, \eqref{e:symbols} and \eqref{e:polyaction} give
\[
\left|b^{\alpha,\gamma}_S\right|=\lim_{A\to\infty}\left|\left\langle\partial^\alpha Tf_A,\psi_S\right\rangle\right|\leq\limsup_{A\to\infty}C\|\nabla^\kappa f_A\|_p\|\psi_S\|_{p'}\lesssim\limsup_{A\to\infty}C\left(A\ell(S)\right)^{|\gamma|-\kappa+\frac dp}\|\psi_S\|_{p'}=0.
\]
The top symbols, and $b^{0,\star}$ if $\kappa=0$, belong to $\mathrm{BMO}(\R^d)$, with norms $\lesssim1+C$, by Lemma \ref{l:carleson}(ii).
\end{proof}
Assume (iii) of Theorem \ref{t:HS}, with constants $C_p$. By Lemma \ref{l:po}(iii), and by (iii) with $w\equiv1$ and the boundedness of $\M$ on $L^p(\R^d)$,
\[
\mathcal N_{m-\kappa}(\partial^\beta f)\lesssim r^{-1}\M(\partial^\beta f),\qquad
\|\partial^\alpha Tf\|_p\lesssim_pr^{-1}C_p\|\nabla^\kappa f\|_p,\qquad 1<p<\infty.
\]
Lemma \ref{l:order0}, with $p=2$ and $p=2d$, therefore proves (i) when $m=\kappa$. When $m>\kappa$, it remains to show that the top symbols have bounded representatives.

\subsection{Logarithmic tests}\label{ss:signedlog}
Let $b\in\mathrm{BMO}(\R^d)$. The kernel $K_t$ of $\mathbb E_t$ satisfies
\[
\int_{\R^d}K_t(x,y)\,\d y=1,\qquad |K_t(x,y)|\lesssim t^{-d}\ind_{\{|x-y|\lesssim t\}},\qquad |\nabla_xK_t(x,y)|\lesssim t^{-d-1},
\]
so that $\mathbb E_tb(x)=\int K_t(x,y)b(y)\,\d y$ is well defined, $\mathbb E_tb\to b$ as $t\to0$ at the Lebesgue points of $b$, and, subtracting the average of $b$ on a ball of radius comparable to $\max\{t,t'\}+|x-x'|$,
\begin{equation}\label{e:Etb}
\begin{split}
\left|\mathbb E_tb(x)-\mathbb E_{t'}b(x')\right|&\lesssim\left(1+\left|\log\frac t{t'}\right|+\frac{|x-x'|}{\min\{t,t'\}}\right)\|b\|_{\mathrm{BMO}},\\
\left|\mathbb E_tb(x)-\mathbb E_tb(x')\right|&\lesssim\frac{|x-x'|}t\,\|b\|_{\mathrm{BMO}}\qquad\textrm{if }|x-x'|\leq t.
\end{split}
\end{equation}
Moreover, by \eqref{e:mra} applied to the restrictions of $b$, minus suitable constants, to large balls, with convergence in $L^2$ on compact sets,
\[
b-\mathbb E_tb=\sum_{\ell(S)\leq t}|S|\psi_S(b)\psi_S\quad\textrm{a.e.},\qquad\left|\mathbb E_{\frac t2}b-\mathbb E_tb\right|\leq\sum_{\ell(S)=t}|S||\psi_S(b)||\psi_S|\lesssim\|b\|_{\mathrm{BMO}}.
\]
Fix $x_0\in\R^d$ and an integer $N\geq1$, let $Q_j$ be the cube of center $x_0$ and sidelength $2^j$, $1\leq j\leq N$, and set
\[
\Omega=\Omega_{x_0,N}\coloneqq\sum_{j=1}^N\omega^\infty_{Q_j},\qquad f_\gamma\coloneqq\frac{(\cdot-x_0)^\gamma}{\gamma!}\,\Omega,\qquad A_N\coloneqq\frac1N\sum_{j=1}^N\mathbb E_{2^j},
\]
for multi-indices $\gamma$ with $|\gamma|\leq k$. As the dilations by powers of $2$ preserve $\mathcal D$ and the wavelets, the sidelengths $2^j$ are only a normalization. The center $x_0$, on the other hand, matters in (c) and (d) below, as $\mathcal D$ is not translation invariant.
\begin{lemma}\label{l:signedlog}
There is an admissible $c\in2^{-\N}$ such that, for all $x_0\in\R^d$, integers $N\geq1$ and multi-indices $|\gamma|\leq k$, the following hold.
\begin{enumerate}[label=(\alph*)]
\item If $|\beta|=|\gamma|$ and $1\leq n\leq m$, then
\[
\mathcal N_n\left(\partial^\beta f_\gamma\right)(x)\lesssim\frac1r\min\Big\{1,\frac{2^{Nd}}{|x-x_0|^d}\Big\},\qquad x\in\R^d.
\]
\item If $1\leq|\mu|\leq k$ and $0\leq n\leq m$, then
\[
\mathcal N_n\left(\partial^\mu\Omega\right)(x)\lesssim\frac1r\min\Big\{1,\frac{2^{Nd}}{|x-x_0|^d}\Big\},\qquad x\in\R^d.
\]
\item If $|\beta|=|\gamma|$ and $b\in\mathrm{BMO}(\R^d)$, then, for almost every $x\in B(x_0,c)$,
\[
\frac1NP_b\left(\partial^\beta f_\gamma\right)(x)=\delta_{\beta\gamma}\left[b(x)-A_Nb(x_0)\right]+O\left(\|b\|_{\mathrm{BMO}}\right).
\]
\item Let $\{b_S:S\in\mathcal D\}$ be a sequence as in \eqref{e:paraform}, and $b^*_t\coloneqq\sup_{\ell(S)=t}|b_S|$ for $t\in2^{\Z}$. On $B(x_0,c)$, in the sense of distributions,
\[
\sum_{S\in\mathcal D}|S|b_S\vartheta_S(\Omega)\psi_S=N\sum_{\ell(S)\leq c}|S|b_S\psi_S+O\Big(N\sum_{t>c}b^*_t\Big),
\]
and, if $1\leq|\mu|\leq k$,
\[
\Big|\sum_{S\in\mathcal D}|S|b_S\vartheta_S\left(\partial^\mu\Omega\right)\psi_S\Big|\lesssim\sum_{t>c}\frac{b^*_t}{t^{|\mu|}},
\]
whenever the sums over $t\in2^{\Z}$ on the right hand sides are finite.
\end{enumerate}
\end{lemma}
\begin{proof}
First, $\Omega=N$ on $\frac32Q_1\supseteq B(x_0,1)$ and, by homogeneity,
\begin{equation}\label{e:fgamma}
\partial^\beta f_\gamma=\sum_{j=1}^N\Sy_{Q_j}\omega_\beta,\qquad |\beta|=|\gamma|,
\end{equation}
where $\omega_\beta\coloneqq\frac1{\gamma!}\partial^\beta(x^\gamma\omega)$ is supported in $(-1,1)^d$ and equals $\delta_{\beta\gamma}$ on $[-\frac34,\frac34]^d$; in particular, $\partial^\beta f_\gamma=\delta_{\beta\gamma}N$ on $\frac32Q_1$.

We prove (a). Let $P$ be a cube. The terms of \eqref{e:fgamma} with $2^j\geq\ell(P)$ are Lipschitz on $P$ with constant $\lesssim2^{-j}$, and those with $2^j<\ell(P)$ have averages on $P$ bounded by a constant multiple of $(2^j/\ell(P))^d$, so that
\begin{equation}\label{e:logbmo}
\inf_{a\in\C}\left\langle\partial^\beta f_\gamma-a\right\rangle_P\lesssim\sum_{j=1}^N\min\left\{\frac{\ell(P)}{2^j},\left(\frac{2^j}{\ell(P)}\right)^d\right\}\lesssim1.
\end{equation}
By Lemma \ref{l:po}(iii) and \eqref{e:logbmo} with $P=3Q$, as $n\geq1$, $\nat{\partial^\beta f_\gamma}_{n,Q}\lesssim r^{-1}$ for every cube $Q$. If $|x-x_0|\geq4\sqrt d\,2^N$ and $Q\ni x$, then either $3Q$ does not meet the support of $\partial^\beta f_\gamma$, and $\nat{\partial^\beta f_\gamma}_{n,Q}=0$, or $\ell(Q)\gtrsim|x-x_0|$, and Lemma \ref{l:po}(iii) with $\pi=0$ gives
\[
\nat{\partial^\beta f_\gamma}_{n,Q}\lesssim\frac{\|\partial^\beta f_\gamma\|_1}{r\,\ell(Q)^d}\lesssim\frac{2^{Nd}}{r\,|x-x_0|^d}.
\]

We turn to (b). The functions $\partial^\mu\omega^\infty_{Q_j}$ are bounded by a constant multiple of $2^{-j|\mu|}\leq1$ and supported in the pairwise disjoint sets $2Q_j\setminus\frac32Q_j$, so that
\[
\|\partial^\mu\Omega\|_\infty\lesssim1,\qquad\|\partial^\mu\Omega\|_1\lesssim2^{Nd}.
\]
By Lemma \ref{l:po}(iii) with $\pi=0$, $\nat{\partial^\mu\Omega}_{n,Q}\lesssim r^{-1}\langle|\partial^\mu\Omega|\rangle_{3Q}$, and the two cases of (a) conclude the proof.

To prove (c), fix a large admissible $Z\in2^{\N}$, $c\leq1/Z$, and $x\in B(x_0,c)$. Set $g\coloneqq\partial^\beta f_\gamma$. At each sidelength $t$, boundedly many cubes $S$ have $\psi_S(x)\neq0$, and all satisfy $x\in\w S$.

If $t\leq c$, then $S\subseteq B(x_0,1)$, where $g=\delta_{\beta\gamma}N$, so these scales contribute $\delta_{\beta\gamma}N(b-\mathbb E_cb)(x)$. If $t>Z2^N$, then $|\vartheta_S(g)|\lesssim t^{-d}\|g\|_1\lesssim(2^N/t)^d$, and their contribution is $O(\|b\|_{\mathrm{BMO}})$.

For $c<t\leq Z2^N$, the cube $S$ lies in $B(x_0,C't)$ for an admissible $C'$. In \eqref{e:fgamma}, the terms with $2^j\geq Zt$ equal $\delta_{\beta\gamma}$ on $S$, while the remaining terms have $\vartheta_S$-averages $\lesssim(2^j/t)^d$. Consequently,
\[
\vartheta_S(g)=\delta_{\beta\gamma}\#\{j:2^j\geq Zt\}+O(1).
\]
There are $N+O(1)$ such scales. Their contribution is therefore
\[
\delta_{\beta\gamma}\sum_{j=1}^N\sum_{c<t\leq2^j/Z}(\mathbb E_{t/2}b-\mathbb E_tb)(x)+O(N\|b\|_{\mathrm{BMO}}).
\]
The inner sum equals $\mathbb E_cb(x)-\mathbb E_{2^j/Z}b(x)$, and \eqref{e:Etb} gives
\[
\mathbb E_{2^j/Z}b(x)=\mathbb E_{2^j}b(x_0)+O(\|b\|_{\mathrm{BMO}}).
\]
Adding the small scales and dividing by $N$ proves (c).

Finally, we prove (d). Let $x\in B(x_0,c)$ and $x\in\w S$. If $\ell(S)\leq c$, then $S\subseteq\frac32Q_1$, where $\Omega=N$, so that $\vartheta_S(\Omega)=N$ and $\vartheta_S(\partial^\mu\Omega)=0$; these scales contribute the main term in the first display of (d), and nothing to the second. If $\ell(S)=t>c$, then $|\vartheta_S(\Omega)|\lesssim N$, and, as $\partial^\mu\vartheta_S$ has vanishing mean, \eqref{e:logbmo} with $\gamma=\beta=0$ gives $|\vartheta_S(\partial^\mu\Omega)|=|(\partial^\mu\vartheta_S)(\Omega-a)|\lesssim t^{-|\mu|}$ for a suitable constant $a$. As boundedly many $S$ of each sidelength have $\psi_S(x)\neq0$, and $|S||\psi_S(x)|\lesssim1$, (d) follows.
\end{proof}

The tests of Lemma \ref{l:signedlog} are combined with the power weights
\[
w_\varepsilon(x)\coloneqq|x-x_0|^{-d+\varepsilon},\qquad x\in\R^d,\quad 0<\varepsilon\leq\frac12,
\]
which concentrate at their pole. Below, $v_d$ is the volume of the unit ball of $\R^d$.
\begin{lemma}\label{l:power}
Let $x_0\in\R^d$, $0<\varepsilon\leq\frac12$, $\rho_0>0$ and $B\coloneqq B(x_0,\rho_0)$. Then
\[
[w_\varepsilon]_{A_\infty}\sim\frac1\varepsilon,\qquad w_\varepsilon(B)=\frac{dv_d}\varepsilon\,\rho_0^\varepsilon,
\]
and, for every measurable $F\geq0$ on $B$,
\begin{equation}\label{e:concentrate}
\frac1{w_\varepsilon(B)}\int_BFw_\varepsilon=\frac\varepsilon d\,\langle F\rangle_B+\Big(1-\frac\varepsilon d\Big)\int_0^{\rho_0}\langle F\rangle_{B(x_0,\rho)}\,\frac{\varepsilon\rho^{\varepsilon-1}}{\rho_0^\varepsilon}\,\d\rho.
\end{equation}
In particular, if $F$ is integrable on $B$ and $\langle F\rangle_{B(x_0,\rho)}\to0$ as $\rho\to0$, then the left hand side of \eqref{e:concentrate} tends to zero as $\varepsilon\to0$. Finally, for every $t>0$,
\[
\int_{\R^d}\min\Big\{1,\frac{t^d}{|x-x_0|^d}\Big\}^2w_\varepsilon(x)\,\d x=dv_d\Big(\frac1\varepsilon+\frac1{2d-\varepsilon}\Big)t^\varepsilon.
\]
\end{lemma}
\begin{proof}
For $x\in B$,
\[
w_\varepsilon(x)=\rho_0^{-d+\varepsilon}+(d-\varepsilon)\int_{|x-x_0|}^{\rho_0}\rho^{-d+\varepsilon-1}\,\d\rho,
\]
so that, by Tonelli's theorem, as $|B(x_0,\rho)|=v_d\rho^d$,
\[
\int_BFw_\varepsilon=v_d\rho_0^\varepsilon\langle F\rangle_B+(d-\varepsilon)v_d\int_0^{\rho_0}\langle F\rangle_{B(x_0,\rho)}\,\rho^{\varepsilon-1}\,\d\rho.
\]
With $F=1$, this gives $w_\varepsilon(B)$, and then \eqref{e:concentrate}. The measure $\varepsilon\rho^{\varepsilon-1}\rho_0^{-\varepsilon}\,\d\rho$ in \eqref{e:concentrate} is a probability measure on $(0,\rho_0)$, which assigns the mass $(\rho_1/\rho_0)^\varepsilon$ to $(0,\rho_1)$. Given $\eta>0$, let $0<\rho_1<\rho_0$ be such that $\langle F\rangle_{B(x_0,\rho)}\leq\eta$ for $0<\rho<\rho_1$. As $\langle F\rangle_{B(x_0,\rho)}\leq|B(x_0,\rho_1)|^{-1}\int_BF$ for $\rho\geq\rho_1$,
\[
\frac1{w_\varepsilon(B)}\int_BFw_\varepsilon\leq\frac\varepsilon d\,\langle F\rangle_B+\eta+\Big(1-\Big(\frac{\rho_1}{\rho_0}\Big)^\varepsilon\Big)\frac1{|B(x_0,\rho_1)|}\int_BF,
\]
and the right hand side tends to $\eta$ as $\varepsilon\to0$. In polar coordinates, the integral of the last claim equals
\[
dv_d\left(\int_0^t\rho^{\varepsilon-1}\,\d\rho+t^{2d}\int_t^\infty\rho^{-2d+\varepsilon-1}\,\d\rho\right).
\]

Finally, let $P$ be a cube. If $\dist(x_0,P)\geq\ell(P)$, the values of $w_\varepsilon$ on $P$ are comparable, so that $\langle w_\varepsilon\rangle_P\lesssim\inf_Pw_\varepsilon$. Otherwise, $P\subseteq B(x_0,2\sqrt d\,\ell(P))$, and
\[
\langle w_\varepsilon\rangle_P\leq\frac{w_\varepsilon\big(B(x_0,2\sqrt d\,\ell(P))\big)}{|P|}\lesssim\frac{\ell(P)^{-d+\varepsilon}}\varepsilon\lesssim\frac1\varepsilon\inf_Pw_\varepsilon.
\]
Hence $\M(w_\varepsilon\ind_Q)\leq\M w_\varepsilon\lesssim\varepsilon^{-1}w_\varepsilon$ for every cube $Q$, and $[w_\varepsilon]_{A_\infty}\lesssim\varepsilon^{-1}$. Conversely, let $Q$ be a cube of center $x_0$, and $x\in Q\setminus\{x_0\}$. The cube $P$ of center $x_0$ and sidelength $\min\{\ell(Q),4|x-x_0|\}$ contains $x$ and is contained in $Q$, so that
\[
\M(w_\varepsilon\ind_Q)(x)\geq\langle w_\varepsilon\rangle_P\geq\frac{w_\varepsilon\big(B(x_0,\frac12\ell(P))\big)}{|P|}\gtrsim\frac{\ell(P)^{-d+\varepsilon}}\varepsilon\gtrsim\frac{w_\varepsilon(x)}\varepsilon.
\]
Thus $[w_\varepsilon]_{A_\infty}\gtrsim\varepsilon^{-1}$.
\end{proof}

\subsection{Proof that (iii) implies (i) in Theorem \ref{t:HS}}\label{ss:Linfty}
Assume (iii) and $m>\kappa$. By \S\ref{ss:order0}, the lower symbols vanish, and the top symbols and, when $\kappa=0$, the adjoint symbol belong to $\mathrm{BMO}(\R^d)$, with norms $\lesssim C_*\coloneqq r^{-1}(1+C_2+C_{2d})$. By Proposition \ref{p:rep}, with $\tau<\delta_1<\delta$, for $f\in \mathcal{C}^\infty_c(\R^d)$,
\[
\partial^\alpha Tf=\sum_{|\beta|=\kappa}V_{\alpha,\beta}\left(\partial^\beta f\right)+\ind_{\{\kappa=0\}}P^*_{\mathfrak b^{0,\star}}f+\Pi f,\qquad \Pi f\coloneqq\sum_{|\gamma'|=\kappa}\frac1{\gamma'!}P_{\mathfrak b^{\alpha,\gamma'}}\left(\partial^{\gamma'}f\right).
\]
The operators $V_{\alpha,\beta}$ and $P^*_{\mathfrak b^{0,\star}}$ are models with mean zero input packets, see Remark \ref{r:models} and Lemma \ref{l:para}(i), which satisfy \eqref{e:V1}, by Lemma \ref{l:V1mean}. By Theorem \ref{t:model} and Lemma \ref{l:Ainf}, with $m-\kappa$ in place of $m$, the first two terms satisfy \eqref{e:PHS} with constants $\lesssim_p1+C_*$, and so does $\Pi$, with constants $\lesssim_pC_p+1+C_*$.

Fix $|\gamma|=\kappa$, let $b\coloneqq\mathfrak b^{\alpha,\gamma}$, and let $x_0$ be a point with $\langle|b-b(x_0)|^2\rangle_{B(x_0,\rho)}\to0$ as $\rho\to0$, which holds for almost every $x_0$, as $b\in L^2_{\mathrm{loc}}(\R^d)$. Let $N\geq2$ and $f\coloneqq f_\gamma$, see \S\ref{ss:signedlog}. Lemma \ref{l:signedlog}(c), applied to each $\partial^{\gamma'}f$ with symbol $\mathfrak b^{\alpha,\gamma'}$, gives, for almost every $x\in B\coloneqq B(x_0,c)$,
\begin{equation}\label{e:mainterm}
\frac{\gamma!}N\,\Pi f(x)=b(x)-A_Nb(x_0)+O\left(C_*\right).
\end{equation}
Let $w\coloneqq w_{1/N}$, so that $[w]_{A_\infty}\lesssim N$ and $w(B)\gtrsim N$, by Lemma \ref{l:power}. By \eqref{e:PHS} for $\Pi$ with $p=2$, Lemma \ref{l:signedlog}(a) and Lemma \ref{l:power}, as $(2^N)^{1/N}=2$,
\[
\left\|\Pi f\right\|_{L^2(w)}\lesssim(1+C_2+C_*)N\sum_{|\beta|=\kappa}\left\|\mathcal N_{m-\kappa}\left(\partial^\beta f\right)\right\|_{L^2(w)}\lesssim\frac{1+C_2+C_*}r\,N^{\frac32}.
\]
Comparing with \eqref{e:mainterm} on $B$, uniformly in $N\geq2$,
\[
\bigg(\frac1{w(B)}\int_B\left|b-A_Nb(x_0)\right|^2w\bigg)^{\frac12}\leq C_\sharp\coloneqq C'\,\frac{1+C_2+C_*}r,
\]
for an admissible $C'$. By Minkowski's inequality and Lemma \ref{l:power} with $F=|b-b(x_0)|^2$ and $\varepsilon=\frac1N$,
\[
\limsup_{N\to\infty}\left|b(x_0)-A_Nb(x_0)\right|\leq C_\sharp.
\]
If $x_0'$ is another such point, the sequence $\mathbb E_{2^j}b(x_0)-\mathbb E_{2^j}b(x_0')$, $j\geq1$, is bounded and tends to zero, by \eqref{e:Etb}, so that its Ces\`aro means $A_Nb(x_0)-A_Nb(x_0')$ tend to zero as $N\to\infty$. Therefore $|b(x_0)-b(x_0')|\leq2C_\sharp$, and $b$ belongs to $L^\infty(\R^d)$, with $\inf_{a\in\C}\|b-a\|_\infty\leq2C_\sharp$. As $C_\sharp\lesssim r^{-2}(1+C_2+C_{2d})$, this also gives the quantitative converse in Theorem \ref{t:HS}.

\subsection{Homogeneous corollaries}
\begin{proof}[Proof of \eqref{e:czMW}]
Let $f\in \mathcal{C}^\infty_c(\R^d)$. By \eqref{e:sobcube} if $m>\kappa$, and by \eqref{e:po} and \eqref{e:msigmavec} if $m=\kappa$, taking suprema over the cubes containing a given point gives $\mathcal N_{m-\kappa}(\partial^\beta f)\lesssim\mathcal N^{m-\kappa}[\nabla^mf]$ for $|\beta|=\kappa$. Thus \eqref{e:PHS}, which holds under (i) by Theorem \ref{t:HS}, implies \eqref{e:czMW}, with the same constants.
\end{proof}
\begin{proof}[Proof of Corollary \ref{c:endpoint}]
We prove (a). Under (i), \eqref{e:Dfine} at a fixed admissible level and \eqref{e:jetcrude} give
\[
|\partial^\alpha Tf|\lesssim\sum_{S\in\mathcal S}\langle|\nabla^\kappa f|\rangle_{3S}\ind_S.
\]
The $L^p$ bounds for sparse operators \cite{LN19}, applied to the sparse collection $\{3S:S\in\mathcal S\}$, give the asserted estimates. Conversely, Lemma \ref{l:order0} gives (i).

We turn to (b). Under (i), \eqref{e:czpot}, at a fixed admissible level, and H\"older's inequality in Lorentz spaces give \eqref{e:Linfty}, since $|\cdot|^{s-d}\in L^{\frac d{d-s},\infty}(\R^d)$. When $\kappa=0$, (i) also gives the $L^2$ boundedness of $T$ by Proposition \ref{p:rep} and Lemma \ref{l:para}.

Conversely, let $Q$ have center $c$, let $A\geq6$, and set $f_A\coloneqq(\cdot-c)^\gamma\omega^\infty_{AQ}$, $|\gamma|\leq\kappa$. As $m>\kappa$, $\nabla^mf_A$ is supported in $2AQ\setminus\frac32AQ$ and bounded by a constant multiple of $(A\ell(Q))^{|\gamma|-m}$. Thus \eqref{e:Linfty} gives
\[
\|\partial^\alpha Tf_A\|_\infty\lesssim\|\nabla^mf_A\|_{L^{\frac ds,1}}
\lesssim(A\ell(Q))^{|\gamma|-m+s}=(A\ell(Q))^{|\gamma|-\kappa}.
\]
Lemma \ref{l:tests}(b) gives the vanishing lower symbols and the bounded top symbols. When $\kappa=0$, the assumed $L^2$ boundedness and Lemma \ref{l:carleson}(ii) give the adjoint $\mathrm{BMO}$ condition.
\end{proof}

\begin{proof}[Proof of Corollary \ref{c:hardy}]
By \S\ref{ss:hardymain}, at the level $r_0$, $\|\mathcal N_0g\|_{L^p(w)}\lesssim\|\M^\tau_{\R^d}g\|_{L^p(w)}$, with a constant depending on $p$ and $w$, for each $g=\partial^\beta f$, $|\beta|=\kappa$. Thus the first claim follows from \eqref{e:PHS} with $m=\kappa$ at the same level, and the second from the first with $w\equiv1$, as $\|\M^\tau_{\R^d}g\|_{L^p(\R^d)}$ is equivalent to the norm of $g$ in $H^p(\R^d)$ for $\frac d{d+\tau}<p\leq1$, see Proposition \ref{c:hardynorm}.
\end{proof}

\subsection{Intermediate channels}\label{ss:middle}
Throughout \S\S\ref{ss:middle}--\ref{ss:separation}, $1\leq\kappa\leq m$, $f\in \mathcal{C}^\infty_c(\R^d)$, and
\[
G\coloneqq\mathcal N_mf+\sum_{|\beta|=\kappa}\mathcal N_{m-\kappa}\left(\partial^\beta f\right).
\]
The next lemma gives decay at large scales for positive derivatives, and at small scales for tests with vanishing mean.
\begin{lemma}\label{l:sandwich}
Let $L$ be a cube and $\varphi\in\Phi^{\kappa+1,1}(L)$ be supported in $L$. Then
\begin{equation}\label{e:sandwich1}
\left|\varphi\left(\partial^\nu f\right)\right|\lesssim\ell(L)^{-|\nu|}\inf_L\mathcal N_mf,\qquad 1\leq|\nu|\leq\kappa.
\end{equation}
If moreover $\ell(L)\leq1$ and $\varphi$ has vanishing mean, then
\begin{equation}\label{e:sandwich2}
\left|\varphi\left(\partial^\nu f\right)\right|\lesssim\ell(L)\left(1+\log\frac1{\ell(L)}\right)\inf_LG,\qquad 1\leq|\nu|<\kappa.
\end{equation}
\end{lemma}
\begin{proof}
Let $L_j\coloneqq2^jL$ and $P_j\coloneqq\PM^m_{3L_j}f$, $j\geq0$. By \eqref{e:projdiff}, $P_j-P_{j+1}=\PM^m_{3L_j}(f-P_{j+1})$, so that Lemma \ref{l:proj}(ii) on the cube $3L_j$, the inequality $\M^\tau_{3L_j}\leq\M^\tau_{3L_{j+1}}$ on $3L_j$, and $|L_j|>r|L_{j+1}|$ give
\begin{equation}\label{e:polytel}
\sup_{9L_j}\left|P_j-P_{j+1}\right|\lesssim\inf_{L_j}\M^\tau_{3L_{j+1}}\left[f-P_{j+1}\right]\leq\nat{f}_{m,L_{j+1}}\leq\inf_L\mathcal N_mf .
\end{equation}
Let $1\leq|\nu|\leq\kappa$. As $f$ has compact support, $\partial^\nu P_j\to0$ uniformly on $L$, by \eqref{e:proj}. Telescoping, and applying Markov's inequality for polynomials of degree less than $m$ on the cubes $9L_j\supseteq L$, \eqref{e:polytel} yields
\begin{equation}\label{e:taylorsize}
\left\|\partial^\nu P_0\right\|_{L^\infty(L)}\lesssim\sum_{j\geq0}\ell(L_j)^{-|\nu|}\inf_L\mathcal N_mf\lesssim\ell(L)^{-|\nu|}\inf_L\mathcal N_mf .
\end{equation}
As $\ell(L)^{|\nu|}\partial^\nu\varphi\in C\Phi^{\tau,1}(L)$ is supported in $L$, \eqref{e:max}, tested at every point of $L$, gives
\[
\left|\varphi\left(\partial^\nu(f-P_0)\right)\right|=\left|\left(\partial^\nu\varphi\right)(f-P_0)\right|\lesssim\ell(L)^{-|\nu|}\inf_L\M^\tau_{3L}\left[f-P_0\right]\leq\ell(L)^{-|\nu|}\nat{f}_{m,L},
\]
while $|\varphi(\partial^\nu P_0)|\leq\|\partial^\nu P_0\|_{L^\infty(L)}$. With \eqref{e:taylorsize}, this proves \eqref{e:sandwich1}.

For \eqref{e:sandwich2}, let $n\geq0$ be the least integer with $2^n\ell(L)\geq\frac12$. Then $\frac12\leq\ell(L_n)\leq1$ and $n\lesssim1+\log(1/\ell(L))$. We prove the claim by downward induction on $|\nu|$, together with the bound
\begin{equation}\label{e:sandwich3}
\left|\varphi\left(\partial^\nu f\right)\right|\lesssim\left(1+\log\frac1{\ell(L)}\right)\inf_LG,\qquad 1\leq|\nu|\leq\kappa,
\end{equation}
for general tests $\varphi\in\Phi^{\kappa+1,1}(L)$ supported in $L$. Throughout, $\inf_{L_i}G\leq\inf_LG$.

Let first $|\nu|=\kappa$. If $m=\kappa$, \eqref{e:max}, tested at every point of $L$, gives $|\varphi(\partial^\nu f)|\lesssim\nat{\partial^\nu f}_{0,L}\leq\inf_LG$. If $m>\kappa$, let $g\coloneqq\partial^\nu f$ and $\pi_i\coloneqq\PM^{m-\kappa}_{3L_i}g$, which equals $\partial^\nu\PM^m_{3L_i}f$, by Lemma \ref{l:proj}(iii). Then
\[
\varphi(g)=\varphi\left(g-\pi_0\right)+\sum_{i<n}\varphi\left(\pi_i-\pi_{i+1}\right)+\varphi\left(\pi_n\right),
\]
where the first term is $\lesssim\nat{g}_{m-\kappa,L}\leq\inf_LG$, as above, each summand is $\lesssim\nat{g}_{m-\kappa,L_{i+1}}\leq\inf_LG$, by \eqref{e:polytel} with $g$ and $m-\kappa$ in place of $f$ and $m$, and the last term is $\lesssim\ell(L_n)^{-\kappa}\inf_{L_n}\mathcal N_mf\lesssim\inf_LG$, by \eqref{e:taylorsize} on the cube $L_n$. This proves \eqref{e:sandwich3} for $|\nu|=\kappa$.

Let now $1\leq|\nu|<\kappa$, and assume \eqref{e:sandwich3} for the multi-indices of length $|\nu|+1$. If $\varphi$ has vanishing mean, Lemma \ref{l:anti} with $m=1$ writes
\[
\varphi=\ell(L)\sum_{i=1}^d\partial_i\varphi_i,\qquad \varphi_i\in C\Phi^{\kappa+1,1}(L),\quad \supp\varphi_i\subseteq L,
\]
and \eqref{e:sandwich3} for $\nu+e_i$ gives \eqref{e:sandwich2} for $\nu$. Let $a\coloneqq\int\varphi$, so that $|a|\leq1$, and write
\[
\varphi=\left(\varphi-a\vartheta_L\right)+a\sum_{i<n}\left(\vartheta_{L_i}-\vartheta_{L_{i+1}}\right)+a\vartheta_{L_n},
\]
where the first two terms are mean zero elements of $C\Phi^{\kappa+1,1}(L)$, respectively $C\Phi^{\kappa+1,1}(L_{i+1})$, supported in $L$, respectively $L_{i+1}$. By \eqref{e:sandwich2} for $\nu$, their contribution to $\varphi(\partial^\nu f)$ is at most a constant multiple of
\[
\sum_{i\leq n}\ell(L_i)\left(1+\log\frac1{\ell(L_i)}\right)\inf_{L_i}G\lesssim\inf_LG,
\]
and the last term contributes $\lesssim\ell(L_n)^{-|\nu|}\inf_{L_n}\mathcal N_mf\lesssim\inf_LG$, by \eqref{e:sandwich1}. This proves \eqref{e:sandwich3} for $\nu$.
\end{proof}

We also need the weighted John--Nirenberg inequality for the localized pieces of a $\mathrm{BMO}$ function.
\begin{lemma}\label{l:wJN}
Let \[\mathfrak b\in\mathrm{BMO}(\R^d), \qquad L\in\mathcal D, \qquad  \mathfrak b_L\coloneqq\sum_{S\in\mathcal D,\,S\subseteq L}|S|\psi_S(\mathfrak b)\psi_S.\]
 For $0<p<\infty$ and $w\in A_\infty$,
\[
\int_{\R^d}|\mathfrak b_L|^pw\leq C\,\Gamma(p+1)\left(C[w]_{A_\infty}\|\mathfrak b\|_{\mathrm{BMO}}\right)^pw(\w L),
\]
with an admissible constant $C$. Here $\w\geq1$ is the dilation parameter of \S\ref{ss:classes}, so that $\w L$ contains the supports of the wavelets $\psi_S$, $S\subseteq L$, and thus that of $\mathfrak b_L$.
\end{lemma}
\begin{proof}
The wavelet coefficients of $\mathfrak b_L$ are a subsequence of those of $\mathfrak b$, so that $\mathfrak b_L\in L^2(\R^d)$ and $\|\mathfrak b_L\|_{\mathrm{BMO}}\lesssim\|\mathfrak b\|_{\mathrm{BMO}}$, by the Carleson characterization of $\mathrm{BMO}(\R^d)$ \cite{MC97}. Moreover, $\mathfrak b_L$ is supported in $\w L$, and it has vanishing mean, as each $\psi_S$ does and the series converges in $L^2(\w L)$. Thus the John--Nirenberg inequality on the cube $\w L$ and \eqref{e:Ainfty} give
\[
w\left(\left\{|\mathfrak b_L|>\lambda\right\}\right)\lesssim\exp\left(-\frac{\lambda}{C[w]_{A_\infty}\|\mathfrak b\|_{\mathrm{BMO}}}\right)w(\w L),\qquad\lambda>0,
\]
and integration against $p\lambda^{p-1}\,\d\lambda$ concludes the proof.
\end{proof}

Combining Lemmas \ref{l:wJN} and \ref{l:sandwich} gives the intermediate-channel estimate.
\begin{lemma}\label{l:middle}
Let $1\leq|\gamma|<\kappa$ and $\mathfrak b\in\mathrm{BMO}(\R^d)$. For $0<p<\infty$ and $w\in A_\infty$,
\[
\left\|P_{\mathfrak b}\left(\partial^\gamma f\right)\right\|_{L^p(w)}\lesssim_p[w]_{A_\infty}\|\mathfrak b\|_{\mathrm{BMO}}\|G\|_{L^p(w)},
\]
with an implied constant growing at most linearly in $p$ as $p\to\infty$.
\end{lemma}
\begin{proof}
Set $g\coloneqq\partial^\gamma f$ and $\varrho(t)\coloneqq t(1+\log(1/t))$, $0<t\leq1$, and write $P_{\mathfrak b,>1}$ for the restriction to $\ell(S)>1$. By \eqref{e:sandwich1} on $\w S$, $|\vartheta_S(g)|\lesssim\ell(S)^{-|\gamma|}\inf_{\w S}G$. Hence \eqref{e:count} and $|\psi_S(\mathfrak b)|\lesssim\|\mathfrak b\|_{\mathrm{BMO}}$ give
\[
|P_{\mathfrak b,>1}g(x)|\lesssim\|\mathfrak b\|_{\mathrm{BMO}}G(x)\sum_{t>1}t^{-|\gamma|}\lesssim\|\mathfrak b\|_{\mathrm{BMO}}G(x),
\]
where $t$ ranges over $2^{\Z}$.

For $\ell(S)\leq1$, telescoping $\vartheta_S(g)$ along the dyadic ancestors of $S$ up to the sidelength one gives, in the sense of distributions,
\[
\sum_{\ell(S)\leq1}|S|\psi_S(\mathfrak b)\vartheta_S(g)\psi_S=\sum_{\ell(L)=1}\vartheta_L(g)\,\mathfrak b_L+\sum_{\ell(L)<1}\left(\vartheta_L-\vartheta_{\widehat L}\right)(g)\,\mathfrak b_L\eqqcolon\sum_{\ell(L)\leq1}a_L\mathfrak b_L,
\]
the sums being taken over $L\in\mathcal D$, with $\mathfrak b_L$ as in Lemma \ref{l:wJN}. The last series converges in $L^2_{\mathrm{loc}}(\R^d)$, as $|a_L|\lesssim\ell(L)\|\nabla\partial^\gamma f\|_\infty$ for $\ell(L)<1$ and $\|\mathfrak b_L\|_2\lesssim\|\mathfrak b\|_{\mathrm{BMO}}|L|^{\frac12}$; at each scale the dilates $\w L$ have bounded overlap. The estimates below apply to finite partial sums, and an almost everywhere convergent subsequence and Fatou's lemma conclude. If $\ell(L)=1$, \eqref{e:sandwich1} on the cube $\w L$ gives $|a_L|\lesssim\inf_{\w L}G$. If $\ell(L)<1$, then $\vartheta_L-\vartheta_{\widehat L}$ is a mean zero element of $C\Phi^{\kappa+1,1}(\w\widehat L)$ supported in $\w\widehat L$, and \eqref{e:sandwich2} on that cube, or \eqref{e:sandwich1} if $\ell(\w\widehat L)>1$, gives $|a_L|\lesssim\varrho(\ell(L))\inf_{\w\widehat L}G$. As $\w L\subseteq\w\widehat L$, in all cases
\[
|a_L|\lesssim\varrho\left(\ell(L)\right)G(y),\qquad y\in\w L .
\]
Fix $t\in2^{\Z}$ with $t\leq1$. At each point, at most $(\w+1)^d$ of the functions $\mathfrak b_L$ with $\ell(L)=t$ do not vanish, by \eqref{e:count}. Hence, by Lemma \ref{l:wJN} and the last display, for an admissible $C$,
\[
\begin{split}
\int_{\R^d}\Big|\sum_{\ell(L)=t}a_L\mathfrak b_L\Big|^pw&\leq C^p\sum_{\ell(L)=t}|a_L|^p\int_{\R^d}|\mathfrak b_L|^pw\\
&\leq C^{p+1}\,\Gamma(p+1)\left(\varrho(t)[w]_{A_\infty}\|\mathfrak b\|_{\mathrm{BMO}}\right)^p\sum_{\ell(L)=t}\int_{\w L}G^pw,
\end{split}
\]
and the last sum is at most $(\w+1)^d\|G\|^p_{L^p(w)}$, by \eqref{e:count} again. If $p\geq1$, Minkowski's inequality over $t$, and, if $p<1$, the subadditivity of the $p$-th power of the quasi-norm, conclude the proof, as $\varrho$ and $\varrho^p$ are summable over $t\in2^{\Z}$, $t\leq1$, $\Gamma(p+1)^{1/p}\lesssim p$ for $p\geq1$, and $[w]_{A_\infty}\geq1$.
\end{proof}

\subsection{Proof of Theorem \ref{t:HSjet} and Proposition \ref{p:rigidity}(a)--(b)}\label{ss:HSjet}
\begin{proof}[Proof of Theorem \ref{t:HSjet}]
Let $f\in \mathcal{C}^\infty_c(\R^d)$. The decomposition \eqref{e:channels} writes $\partial^\alpha Tf$ as the sum of the intermediate paraproducts and of boundedly many terms $T_i(\partial^{\gamma_i}f)$ with $\gamma_i=0$ or $|\gamma_i|=\kappa$. By assumptions (a) and (c), and by \S\ref{ss:channels}, the latter operators satisfy the hypotheses of Theorem \ref{t:model} at order $m-|\gamma_i|$.
By Theorem \ref{t:model} and Lemma \ref{l:Ainf},
\[
\left\|T_i\left(\partial^{\gamma_i}f\right)\right\|_{L^p(w)}\lesssim_p[w]_{A_\infty}\left\|\mathcal N_{m-|\gamma_i|}\left(\partial^{\gamma_i}f\right)\right\|_{L^p(w)}\leq[w]_{A_\infty}\|G\|_{L^p(w)},
\]
with $G$ as in \S\ref{ss:middle}. If $0<|\gamma|<\kappa$, Lemma \ref{l:middle} bounds $\|P_{\mathfrak b^{\alpha,\gamma}}(\partial^\gamma f)\|_{L^p(w)}$ by a constant multiple of $[w]_{A_\infty}\|\mathfrak b^{\alpha,\gamma}\|_{\mathrm{BMO}}\|G\|_{L^p(w)}$. The quasi-triangle inequality gives \eqref{e:czextreme}, with $C_p(T)\lesssim_p1+\mathfrak B(T)$ growing at most linearly in $p$.
\end{proof}

\begin{proof}[Proof of Proposition \ref{p:rigidity}(a)--(b)]
We prove (a). Assume \eqref{e:czextreme}. With $w\equiv1$ and $p=2$, Lemma \ref{l:po}(iii) gives $\mathcal N_mf\lesssim r^{-1}\M f$ and $\mathcal N_{m-\kappa}(\partial^\beta f)\lesssim r^{-1}\M(\partial^\beta f)$, so that $\|\partial^\alpha Tf\|_2\lesssim r^{-1}C_2(T)\|f\|_{W^{\kappa,2}}$, and the top symbols belong locally to $\mathrm{BMO}(\R^d)$, by Lemma \ref{l:carleson}(i).

Let $b\coloneqq b^{\alpha,0}$, fix $x_0\in\R^d$ and $N\geq2$, let $c$ be as in Lemma \ref{l:signedlog}, $\Omega=\Omega_{x_0,N}$ and $B\coloneqq B(x_0,c)$. By Proposition \ref{p:rep}, with $\tau<\delta_1<\delta$,
\[
\partial^\alpha T\Omega=\sum_{|\beta|=\kappa}V_{\alpha,\beta}\left(\partial^\beta\Omega\right)+\sum_{|\gamma|\leq\kappa}\frac1{\gamma!}P^{\alpha,\gamma}\left(\partial^\gamma\Omega\right),
\]
where $\langle P^{\alpha,\gamma}u,v\rangle\coloneqq\Pi_{b^{\alpha,\gamma}}(u,v)$, see \eqref{e:paraform}. By \eqref{e:size}, the quantities $b^*_t$ of Lemma \ref{l:signedlog}(d) for the symbol $b^{\alpha,\gamma}$ are $\lesssim t^{|\gamma|-\kappa}$, and $t^{-\kappa}$ is summable over $t\in2^{\Z}$, $t>c$, as $\kappa\geq1$, so that Lemma \ref{l:signedlog}(d), with $\mu=\gamma$ if $\gamma\neq0$, gives, on $B$ and in the sense of distributions,
\[
\sum_{|\gamma|\leq\kappa}\frac1{\gamma!}P^{\alpha,\gamma}\left(\partial^\gamma\Omega\right)=N\,\mathfrak b_{\leq c}+O(N),\qquad\mathfrak b_{\leq c}\coloneqq\sum_{\ell(S)\leq c}|S|\,b_S\psi_S .
\]
The functions $V_{\alpha,\beta}(\partial^\beta\Omega)$ belong to $L^2(\R^d)$, and so does $\partial^\alpha T\Omega$, by \eqref{e:czextreme} with $w\equiv1$. Hence $\mathfrak b_{\leq c}\in L^2(B)$, and the last identity holds almost everywhere on $B$. This distribution is independent of $x_0$ and $N$; since $x_0$ is arbitrary, it belongs to $L^2_{\mathrm{loc}}(\R^d)$. By \eqref{e:size} and the Cauchy--Schwarz inequality at each sidelength, the series defining it is continuous in the norm of $\mathcal U_\kappa$ on the functions supported in a fixed compact set, so that Lemma \ref{l:Uj}(iii) and orthonormality give $\psi_S(\mathfrak b_{\leq c})=b_S\ind_{\{\ell(S)\leq c\}}$ for all $S\in\mathcal D$. Let $w\coloneqq w_{1/N}$, so that $[w]_{A_\infty}\lesssim N$ and $w(B)\gtrsim N$, by Lemma \ref{l:power}. The operators $V_{\alpha,\beta}$ satisfy the assumptions of Theorem \ref{t:model} of order $m-\kappa$, so that this theorem, Lemma \ref{l:Ainf}, \eqref{e:czextreme}, Lemma \ref{l:signedlog}(a)--(b) and Lemma \ref{l:power} give
\[
\left\|V_{\alpha,\beta}\left(\partial^\beta\Omega\right)\right\|_{L^2(w)}\lesssim N\left\|\mathcal N_{m-\kappa}\left(\partial^\beta\Omega\right)\right\|_{L^2(w)}\lesssim\frac{N^{\frac32}}r,\qquad\left\|\partial^\alpha T\Omega\right\|_{L^2(w)}\lesssim\frac{C_2(T)N^{\frac32}}r.
\]
Hence,
\[
\bigg(\frac1{w(B)}\int_B\left|\mathfrak b_{\leq c}\right|^2w\bigg)^{\frac12}\lesssim\frac{1+C_2(T)}r,
\]
uniformly in $x_0$ and $N\geq2$. Letting $N\to\infty$ at the points $x_0$ with $\langle|\mathfrak b_{\leq c}-\mathfrak b_{\leq c}(x_0)|^2\rangle_{B(x_0,\rho)}\to0$ as $\rho\to0$, which form a set of full measure, Lemma \ref{l:power} and Minkowski's inequality give $|\mathfrak b_{\leq c}(x_0)|\lesssim(1+C_2(T))/r$. As the cubes with $c<\ell(S)\leq1$ contribute to the high frequency part of $b$ a function bounded by a constant multiple of $c^{-\kappa}\lesssim1$, by \eqref{e:size}, $b$ belongs locally to $L^\infty(\R^d)$, and hence to $L^\infty(\R^d)$, by Lemma \ref{l:size}(ii).

We turn to (b). If $\kappa=m=1$, there are no intermediate symbols, and the conditions of Theorem \ref{t:HSjet} coincide with (i) of Theorem \ref{t:jet}, as do the conclusions of (a).
\end{proof}
Part (a) of Proposition \ref{p:rigidity} holds, with the same proof, under the weaker assumption in which the right hand side of \eqref{e:czextreme} also contains the intermediate channels $\partial^\gamma f$, $0<|\gamma|<\kappa$, through the norms $\|\mathcal N_{m-|\gamma|}(\partial^\gamma f)\|_{L^p(w)}$.

\subsection{An unbounded intermediate symbol: proof of Proposition \ref{p:rigidity}(c)}\label{ss:separation}
\begin{proof}[Proof of Proposition \ref{p:rigidity}(c)]
Let $Q_j$ be the cube of center $0$ and sidelength $q_j\coloneqq2^{-j}$, $j\geq0$. Fix a real valued $U\in\mathcal{C}^\infty_c((-\frac12,\frac12)^d)$ with $\partial^{\alpha+e_1}U(0)=1$, and let $u\coloneqq\partial_1U$ and $\upsilon\coloneqq\partial^\alpha u$, so that $u$ has vanishing mean and $\upsilon(0)=1$. Let
\[
\langle Tf,g\rangle\coloneqq\sum_{j\geq0}q_j^\kappa|Q_j|\,\vartheta_{Q_j}\left(\partial^\gamma f\right)u_{Q_j}(g),\qquad f,g\in\mathcal U_0.
\]
By \eqref{e:norm1} and integration by parts in $\vartheta_{Q_j}(\partial^\gamma f)$, the kernel $K$ of $T$ is the sum of the functions
\[
K_j(x,y)\coloneqq(-1)^{|\gamma|}q_j^{\kappa-|\gamma|-d}\,(\partial^\gamma\vartheta)\Big(\frac y{q_j}\Big)\,u\Big(\frac x{q_j}\Big),\qquad j\geq0,
\]
and $K_j(x,y)=0$ unless $|x|_\infty,|y|_\infty<\frac12q_j$. Let $0\leq i\leq k$, $x\neq y$, and $|x-x'|+|y-y'|\leq\frac12|x-y|$. The terms $K_j$ which do not vanish at $(x,y)$ or at $(x',y')$ have $|x-y|\lesssim q_j\leq1$, and
\[
\left|\nabla_x^iK_j(x,y)\right|\lesssim q_j^{\kappa-|\gamma|-d-i},\qquad\left|\nabla_x^iK_j(x',y')-\nabla_x^iK_j(x,y)\right|\lesssim q_j^{\kappa-|\gamma|-d-i}\left(\frac{|x-x'|+|y-y'|}{q_j}\right)^\delta.
\]
Summing over $q_j\gtrsim|x-y|$, as $\kappa-|\gamma|>0$, gives \eqref{e:czkernel} up to an admissible constant; moreover, $K(x,y)=0$ if $|x-y|_\infty\geq1$. For \eqref{e:wbp}, let $\zeta_R\in\Phi^{k+\delta,\Subset}(R)$ and $\xi_R\in\Phi^{\delta,\Subset}(R)$. If $q_j\leq\ell(R)$, then $|\vartheta_{Q_j}(\partial^\gamma\zeta_R)|\lesssim\ell(R)^{-|\gamma|}|R|^{-1}$ and, by the vanishing mean of $u$, $|u_{Q_j}(\xi_R)|\lesssim|R|^{-1}(q_j/\ell(R))^\delta$. If $q_j>\ell(R)$, then $|\vartheta_{Q_j}(\partial^\gamma\zeta_R)|=|(\partial^\gamma\vartheta_{Q_j})(\zeta_R)|\lesssim q_j^{-d-|\gamma|}$ and $|u_{Q_j}(\xi_R)|\lesssim q_j^{-d}$. Hence
\[
|R|\left|\langle T\zeta_R,\xi_R\rangle\right|\lesssim\frac{\ell(R)^{-|\gamma|}}{|R|}\sum_{q_j\leq\min\{1,\ell(R)\}}q_j^{\kappa+d}\Big(\frac{q_j}{\ell(R)}\Big)^\delta+|R|\sum_{\ell(R)<q_j\leq1}q_j^{\kappa-|\gamma|-d}\lesssim1,
\]
again as $\kappa-|\gamma|>0$. Finally, as $q_j^\kappa|Q_j|\|\partial^\gamma\vartheta_{Q_j}\|_2\|u_{Q_j}\|_2\lesssim q_j^{\kappa-|\gamma|}$, $T$ is bounded on $L^2(\R^d)$, hence continuous from $\mathcal U_0$ to its dual, and $T$ is a constant multiple of a Calder\'on--Zygmund operator of order $k$. As $\partial^\alpha u_{Q_j}=q_j^{-\kappa}\upsilon_{Q_j}$, \eqref{e:dalpha} and integration by parts give
\[
\partial^\alpha Tf=\sum_{j\geq0}|Q_j|\,\vartheta_{Q_j}\left(\partial^\gamma f\right)\upsilon_{Q_j},\qquad \Upsilon\coloneqq\sum_{j\geq0}|Q_j|\,\upsilon_{Q_j}=\sum_{j\geq0}\upsilon(2^j\cdot).
\]
As $\upsilon$ is supported in $(-\frac12,\frac12)^d$ and $\upsilon(0)=1$, $\Upsilon(x)=\log_2^+\frac1{2|x|_\infty}+O(1)$, so that $\Upsilon\in\mathrm{BMO}(\R^d)\setminus L^\infty(\R^d)$. By \eqref{e:polyaction}, \eqref{e:symbols} and \eqref{e:thetamom}, $b^{\alpha,\gamma}_S=\gamma!\,\psi_S(\Upsilon)$ for all $S\in\mathcal D$, and $b^{\alpha,\gamma}$ does not belong to $L^\infty(\R^d)$, by uniqueness of the $\mathrm{BMO}$ representative modulo constants. Hence the conditions (i) of Theorem \ref{t:jet} fail for every $m\geq\kappa$.

On the other hand, let $f\in \mathcal{C}^\infty_c(\R^d)$ and $g\coloneqq\partial^\gamma f$. Telescoping $\vartheta_{Q_j}(g)$ along the chain $\{Q_j\}$,
\[
\partial^\alpha Tf=\vartheta_{Q_0}(g)\Upsilon+\sum_{i\geq1}\left(\vartheta_{Q_i}-\vartheta_{Q_{i-1}}\right)(g)\,\Upsilon(2^i\cdot),
\]
where $\Upsilon(2^i\cdot)$ is supported in $Q_i$, has vanishing mean (the defining series converges in $L^1$, and $\int\upsilon=0$), and has the same $\mathrm{BMO}$ norm as $\Upsilon$, so that the John--Nirenberg inequality on $Q_i$ and \eqref{e:Ainfty} give the bound of Lemma \ref{l:wJN} for it, with $Q_i$ in place of $\w L$. By \eqref{e:sandwich1} on the cube $Q_0$ and \eqref{e:sandwich2} on the cubes $Q_{i-1}$, the coefficients are $\lesssim\inf_{Q_0}G$ and $\lesssim i\,2^{-i}\inf_{Q_{i-1}}G$. As $i2^{-i}$ and $(i2^{-i})^p$ are summable over $i\geq1$, Minkowski's inequality for $p\geq1$, or $p$-subadditivity for $p<1$, gives
\[
\left\|\partial^\alpha Tf\right\|_{L^p(w)}\lesssim_p[w]_{A_\infty}\|\Upsilon\|_{\mathrm{BMO}}\|G\|_{L^p(w)}.
\]
Thus $T$ satisfies \eqref{e:czextreme} for every $m\geq\kappa$.
\end{proof}

\section{Examples and open questions}\label{s:examples}
The examples distinguish the homogeneous, inhomogeneous and $A_p$ theories, and show why the $L^2$ assumption at the Lorentz endpoint is needed.
\begin{example}[Multiplication operators]\label{r:mult}
Let $a$ be a bounded function and $Tf\coloneqq af$. Up to a constant multiple, $T$ is a Calder\'on--Zygmund operator of every order $k$, with vanishing kernel. Let $1\leq\kappa\leq k$ and $|\alpha|=\kappa$. If $\nabla^\kappa a\in L^\infty(\R^d)$, then $\nabla^ja\in L^\infty(\R^d)$ for $0\leq j\leq\kappa$, by the Gagliardo--Nirenberg inequality, and each term of
\[
\partial^\alpha Tf=\sum_{\nu\leq\alpha}\binom\alpha\nu\partial^\nu f\,\partial^{\alpha-\nu}a
\]
is the image of $\partial^\nu f$ under the multiplication by a bounded function, an operator of order zero whose symbols $b^{0,0}$ and $b^{0,\star}$ are the wavelet coefficients of that function. Theorem \ref{t:HS} with $\kappa=0$ and $m-|\nu|$ in place of $m$, applied to each term at the level $2^{d+1}r$, gives \eqref{e:jetfull} at that level, and Lemma \ref{l:sandwichperc}, at the same level and with the cubes of sidelength at least one doubled, gives \eqref{e:jet} at the level $r$, for every $m\geq\kappa$, and thus \eqref{e:czextreme}, by Theorems \ref{t:jet} and \ref{t:HSjet}.

Conversely, $b^{\alpha,0}_S=\psi_S(\partial^\alpha a)$ in the sense of distributions, and, by Proposition \ref{p:rigidity}(a), \eqref{e:czextreme} provides $\mathfrak b\in L^\infty(\R^d)$ with these wavelet coefficients; as the expansions \eqref{e:onb} of the test functions with vanishing mean, and their termwise derivatives of order $\kappa$, converge absolutely in $L^1(\R^d)$, the distribution $\partial^\alpha a-\mathfrak b$ is constant, and $\partial^\alpha a\in L^\infty(\R^d)$. Thus \eqref{e:jet} and \eqref{e:czextreme}, for all $|\alpha|=\kappa$, are both equivalent to $\nabla^\kappa a\in L^\infty(\R^d)$: the pointwise and norm theories coincide for multiplication operators. By the same argument, the lower symbols $b^{\alpha,0}$ vanish for all $|\alpha|=\kappa$ only if $\nabla^\kappa a$ is constant, that is, as $a$ is bounded, only if $a$ is constant; otherwise, the conditions of Theorem \ref{t:HS} fail for some $|\alpha|=\kappa$, cf.\ \S\ref{ss:introjet}.
\end{example}
\begin{example}[The $L^2$ assumption at the endpoint]\label{r:D}
When $\kappa=0$, the $L^2$ boundedness in Corollary \ref{c:endpoint}(b) is not implied by \eqref{e:Linfty}. Indeed, if $\beta=\{\beta_S\}$ is a bounded sequence which is not a Carleson sequence, the operator $\Pi^*$ of Lemma \ref{l:para}(iv) has the symbol $b^{0,\star}=\beta$, which does not belong to $\mathrm{BMO}(\R^d)$, so that $\Pi^*$ is unbounded on $L^2(\R^d)$, by the $T(1)$ theorem \cite{DJ84}; however, for $d\geq2$, as $|\psi_S(f)|\lesssim\ell(S)\langle|\nabla f|\rangle_{\w S}$ by Lemma \ref{l:anti} on the cube $\w S$, we have
\[
|\Pi^*f(x)|\lesssim\sum_{S\in\mathcal D}\ell(S)\langle|\nabla f|\rangle_{\w S}\ind_S(x)
\lesssim\int_{\R^d}\frac{|\nabla f(y)|}{|x-y|^{d-1}}\,\d y.
\]
Indeed, at each scale there is one dyadic cube containing $x$, and a contributing cube has $\ell(S)\gtrsim|x-y|$; summing $\ell(S)^{1-d}$ proves the second inequality. This also gives absolute convergence almost everywhere and a locally integrable representative of $\Pi^*f$. H\"older's inequality in Lorentz spaces gives \eqref{e:Linfty} with $s=1$.
\end{example}
\begin{example}[Sobolev bounds and testing conditions]\label{r:DGW}
Let $\kappa=m=1$. The assumptions of Proposition \ref{c:czweighted}(b) do not imply \eqref{e:czextreme}. Indeed, let $\phi\in \mathcal{C}^\infty_c(\R^d)$ be equal to one near the origin, and $a(x)\coloneqq x_1\log(1/|x|)\phi(x)$, so that $a$ is bounded and $\nabla a\in\mathrm{BMO}(\R^d)$. For the multiplication by $a$ and $\alpha=e_1$,
\[
b^{e_1,0}_S=\psi_S(\partial_1a),\qquad b^{e_1,e_j}_S=\delta_{1j}\psi_S(a)+\psi_S\big((\cdot-c(S))_j\,\partial_1a\big),
\]
where the last term is $O(\ell(S))$: the test $\ell(S)^{-1}(\cdot-c(S))_j\psi_S$ has vanishing mean, is supported in $\w S$, and is bounded by $C|S|^{-1}$, so subtracting the average of $\partial_1a$ on $\w S$ gives the bound. This term therefore contributes a bounded function at the scales at most one. Hence the assumptions of Proposition \ref{c:czweighted}(b) hold, while \eqref{e:czextreme} fails, as $\partial_1a\notin L^\infty(\R^d)$, see Example \ref{r:mult}.

For a single exponent and $w\equiv1$, \eqref{e:sobolev} does not force the lower symbols to belong to $\mathrm{BMO}(\R^d)$. Indeed, fix $p>d$ and $0<\varepsilon<d/p$, and set $a_\varepsilon(x)=|x|^{1-\varepsilon}\phi(x)$. Then $a_\varepsilon$ is bounded, $\nabla a_\varepsilon\in L^p$, and Sobolev embedding gives
\[
\|a_\varepsilon f\|_{W^{1,p}}\lesssim\|a_\varepsilon\|_\infty\|f\|_{W^{1,p}}+\|\nabla a_\varepsilon\|_p\|f\|_\infty\lesssim\|f\|_{W^{1,p}}.
\]
However, $\partial_1a_\varepsilon=(1-\varepsilon)x_1|x|^{-1-\varepsilon}$ near zero has mean oscillation comparable to $\rho^{-\varepsilon}$ on balls of radius $\rho$ centered at zero, and so is not in $\mathrm{BMO}$.
\end{example}
\begin{remark}[Open questions]\label{r:open}
It remains open whether \eqref{e:czextreme} with $m>\kappa$ forces the top symbols to belong locally to $L^\infty(\R^d)$. The logarithmic argument of \S\ref{ss:Linfty} does not settle this question: the lowest-order norm of the polynomial times the logarithmic cutoff grows at the large scales used there. Restricting the tests to small scales avoids that growth but loses the comparison of centers which identifies a single limiting constant.

It also remains to characterize the intermediate symbols compatible with \eqref{e:czextreme} when $\kappa\geq2$. Theorem \ref{t:HSjet} gives the sufficient condition that they belong to $\mathrm{BMO}(\R^d)$, while Lemma \ref{l:carleson}(i) supplies the necessary Carleson conditions of order $|\gamma|-\kappa$, as \eqref{e:czextreme} with $p=2$ and $w\equiv1$ implies its assumption, by Lemma \ref{l:po}(iii).

Finally, we do not know whether \eqref{e:sobolev}, for all $1<p<\infty$ and $w\in A_p$, with the sharp dependence on $[w]_{A_p}$, forces the lower symbols to belong to $\mathrm{BMO}(\R^d)$; by Example \ref{r:DGW}, a single exponent does not.
\end{remark}

\subsection*{Declaration of generative AI and AI-assisted technologies
in the manuscript preparation process}
During the preparation of this work, the author used Claude
(Anthropic)  to assist in writing up and revising the text and the proofs from the author's drafts and outlines, in
checking the mathematical arguments, cross-references and citations,
and in  editing the \LaTeX{} source. The author reviewed
and edited the manuscript, independently verified the mathematical
arguments and references, and takes full responsibility for the
content of the article.
 \bibliographystyle{plain}
\bibliography{cancellativesparse}
\end{document}